\documentclass[english]{extarticle}
\usepackage[T1]{fontenc}
\usepackage[utf8]{inputenc}
\usepackage{geometry}
\usepackage{comment}
\usepackage{booktabs}
\usepackage{subcaption}
\usepackage{caption}
\usepackage{comment}
\usepackage{array}
\usepackage{multirow}
\usepackage{amsmath}
\usepackage{amsthm}
\usepackage{amssymb}
\usepackage{stmaryrd}
\usepackage{graphicx}
\usepackage{xcolor}
\usepackage{algorithm}
\usepackage{algpseudocode}
\usepackage[numbers]{natbib}
\usepackage[title]{appendix}

\makeatletter

\theoremstyle{plain}
\newtheorem{thm}{\protect\theoremname}
\theoremstyle{remark}
\newtheorem{rem}{\protect\remarkname}
\theoremstyle{definition}

\usepackage{authblk}

\@ifundefined{showcaptionsetup}{}{%
 \PassOptionsToPackage{caption=false}{subfig}}
\usepackage{subfig}
\makeatother

\usepackage{babel}
\providecommand{\definitionname}{Definition}
\providecommand{\remarkname}{Remark}
\providecommand{\theoremname}{Theorem}

\begin{document}
\title{Predicting Generalized Steady States in Aperiodically Forced Mechanical Systems}
\author{ Roshan S. Kaundinya$^{1}$, Isabella Thiel$^{2}$, B\'alint Kasz\'as$^{3}$, Shobhit Jain$^{4}$ and  George Haller$^{1}$ \thanks{georgehaller@ethz.ch}\\ 

$^{1}$Institute for Mechanical Systems, ETH Z\"{u}rich\\ $^{2}$SAND Lab, MIT  \\
$^{3}$Center for Turbulence Research, Stanford University\\
$^{4}$ TU Delft\\}
\maketitle
\begin{abstract}

The existence of generalized steady states (GSSs) in nonlinear mechanical systems under moderate temporally aperiodic forcing has only been shown recently. Here we derive systematic expansions for such GSSs and construct a numerical algorithm that yields explicit and arbitrarily refinable approximations for GSSs without the need for an initial convergence period. This is to be contrasted with a direct numerical integration of the system, whose convergence is hard to assess or is even undefined for short, transient forcing. When at least the linear part of the equations of motion is known, our GSS algorithm outperforms available data-driven neural-network-based techniques for predicting forced response in structural dynamics problems. In a fully equation-driven setting, our GSS computations are shown to be faster than a direct numerical integration of forced nonlinear finite-element models of beams and shells.
\begin{comment}
Dynamic steady states are defining characteristics of a non-autonomous dynamical system. Such states capture long-term behavior and key statistical properties of the dynamical system. Vibration and civil engineers expect dynamic steady states to exist for externally forced structural mechanical systems and seek the same through direct numerical time-integration. The most common seeking process is to launch the non-autonomous dynamical system at an arbitrary initial condition for a long duration. From the resulting trajectory, the dynamic steady state is recovered by ignoring the initial transient behavior. Naive seeking approaches are costly when the dynamical system describes a forced lowly damped high dimensional mechanical device or civil structure.  We prove the existence of dynamic steady states which we define as generalized steady states (GSS) for aperiodically forced dynamical systems using only hyperbolicity properties of the unforced dynamical systems. We offer a computational methodology that directly computes the GSS and offers computational speed ups compared to naive shooting approaches. Specially, our method provides an analytic asymptotic Taylor expansion of the GSS. We successfully uncover GSSs for classic mechanical systems and finite-element models describing beams and plates subject to various aperiodic forcing profiles. We also demonstrate the utility of our approach on a spectral submanifold-reduced model  and compare with recent approaches using long-short term neural networks (LSTM).
\end{comment}

\end{abstract}

\section{Introduction}

%(please cite naturally occuring forcings, also cite works that look at limiting forced response for damage)
Civil engineering structures, airplane wings and wind turbine blades are constantly tested for damage under external environmental conditions (see, e.g., \citet{sause2021structural} and \citet{Vidal2014}). For vibration engineering analysis, these environmental conditions are modeled as external forces acting on a structural system. Naturally occurring external forces stem from earthquake ground accelerations (\citet{PEER_NGA_West2}), wind pressure fluctuations (\citet{carta09}), soil-structure interactions (\citet{harte12}) and fluid-structure interactions (\citet{cadot16}). Such external forcings have two key characteristics: they always act over a finite time interval and are temporally aperiodic. 

Within such a finite time interval, damage detection and assessment schemes crucially rely on the forced response of the structures. Of particular concern to the practitioner is the prolonged structural deformation caused by persistent forcing, which requires understanding the structure's steady state (i.e., converged) forced response. Establishing that a system has reached such a steady state is straightforward under periodic or quasiperiodic external forcing: one simply waits until the response reaches a periodic or quasiperiodic state that has the same frequency content as the forcing. This approach, however, is inapplicable under general aperiodic forcing. One might still hope that a statistical steady state is reached at some point, but this would assume a well-defined, steady, long-term statistics for the external forcing. In practice, this assumption is typically not satisfied, or the time interval of interest is too short to infer any statistics for the forcing or to conclude the convergence of the forced response definitively. A practically relevant example is forcing caused by earthquakes, which is generally short-lived and hence has ill-defined temporal statistics.

%( cite balanced trunctaiom, harmonic balance, shooting methods for periodic, aperiodic and quasiperidic, 
% Guesswork
Direct numerical integration is routinely used to probe and estimate aperiodic steady states (see \citet{newmark59}, \citet{hht_77} and \citet{felippa79_DNI} for a review). Such approaches, however, do not address the number of such steady states, the optimal initial condition required to locate them, and the transients to be removed. Moreover, structural engineering systems are generally high-dimensional and nonlinear, resulting in computationally intensive time-integrations (see \citet{avery7}, \citet{jain19}). 

Shooting and collocation methods for periodic steady states address some of the above limitations of direct numerical integration, particularly those arising from sensitivity to initial conditions and the elimination of transients (see \citet{keller68}, \citet{Diekhoff1977} and \citet{Dankowicz13}). However, they remain computationally expensive for multi-degree-of-freedom systems (see \citet{renson16} for a review). In contrast, the harmonic balance method assumes that a periodic steady state can be represented by a formal Fourier expansion, where the Fourier coefficients satisfy a nonlinear algebraic system of equations (see \citet{Kryloff32}). Similar balancing approaches have been extended to quasi-periodic forcing (see  \citet{chua81} and \citet{lau83}), and this methodology has been automated in the \textit{NLvib} software toolbox (see \citet{krack2019harmonic}).

All the above-discussed numerical approaches to computing periodic or quasi-periodic steady states assume the existence of such states, with all other solutions converging to them. While the assumption is trivially satisfied for damped linear systems, the same condition does not automatically extend to nonlinear dynamical systems. For second-order nonlinear dynamical systems with linear damping, the existence of periodic steady states has been established by \citet{Rouche80}. Their ideas have been extended in \citet{breunung19} to provide sufficient conditions on the existence of periodic forced response in general nonlinear mechanical systems. \citet{breunung21} has recently extended these results to cover quasi-periodic steady states.

On the computational side, \citet{jain19} and \citet{buza21} have derived numerical techniques grounded in existence proofs based on the Banach fixed-point theorem. These techniques yield an iterative algorithm, \textit{SSRTool} (see \citet{shobhit_jain_2021_4421944}), that computes periodic and quasi-periodic steady states by solving nonlinear integral equations. However, these rigorous numerical approaches also become challenging to apply when the system dimension is large, or the unforced linear operator of the full model is unknown, as \textit{SSRTool} relies on exponential kernels associated with that operator.

In contrast, reducing a high-dimensional dynamical system onto a robust, low-dimensional invariant manifold (see \citet{guckenheimer83} and \citet{shaw93}) and seeking steady state solutions on such manifolds overcomes these challenges. Spectral submanifolds (SSMs), defined first by  \citet{haller16}, are generalizations of the invariant manifolds first appearing in \citet{shaw93}. Specifically, slow SSMs are normally hyperbolic invariant manifolds that capture the dominant dynamics of the system, and hence their internal dynamics represent a mathematically justified reduced-order model (see \citet{haller25} for further details).

\textit{SSMTool}, an open-source numerical toolkit, automates the computation of SSMs and the computation of periodic forced response curves (FRCs) directly from the equations of motion of high-dimensional dynamical systems (see \citet{jain23}). Another open-source toolkit, \textit{SSMLearn}, extracts slow SSMs from experimental or numerical data and predicts FRCs under arbitrary periodic forcing (see \citet{cenedese21}). Recent work by \citet{kaszas2025}, have enlarged the domain of validity of SSM-reduction in both equation- and data- driven settings via the use of Pad\'e approximations instead of Taylor expansions. This extension offers predictive power under large amplitude periodic forcing as well.

%( cite Burd, Haller book, challenging to estimate exponentional integrals
% analytic results proving they exist

For truly aperiodic forcing, however, no comparable rigorous results or automated toolboxes have been available to date. As a broadly used operational approach, proper orthogonal decomposition (POD) has been employed to approximate aperiodically forced response in known governing equations (see \citet{lumley1970stochastic} and \citet{SIROVICH1989126}). In this setting, POD modes are first computed from forced response data, followed by a projection of the known evolution equations onto the POD modes. This approach is therefore fully equation-driven, and the computed POD modes depend on the forcing realization used. Space-time and spectral POD methods (see \citet{Towne18}) circumvent this issue by computing modes in the frequency domain, producing modes that are independent of the specific forcing realization. However, these approaches are justified only when the applied forcing is statistically stationary.

%We summarize here, existing approaches that address the existence assumptions for GSSs and also computational methdologies that are based on this. Starting with linear dynamical systems with additive forcing allow for closed form Duhammel-Fredholm type integral solutions of the GSS. In particular, these results provide the exact GSS of the linear forced system for arbitrary large forcing under the assumption the unforced linear system admits a hyperbolic fixed point. Local existence results for periodic GSSs for weakly nonlinear dynamical systems with weak additive periodic forcing have been provided in . This was strengthened by Breueng et.al for general periodically forced nonlinear dynamical systems and recently similar approaches were use to extend these results for quasiperiodic GSS.  Jain et al, provide tool kit \textit{SSRtool} that computes periodic and quasiperiodic GSSs by solving for the periodic GSS as a solution to the nonlinear integral equations. However, if the system dimension is large or if the unforced linear operator is unknown, all these informed computational approaches become challenging as they crucially depend on the exponential kernels involving the unforced linear operator. 

%(for equation-data based , statistical assumptions, method devised here)
%(data-based) (reduced-modeling) (data-assisted)  

%(DMD extensions for nonautonomous data)
Data-driven approaches for modeling non-autonomous dynamical systems extend dynamic mode decomposition (DMD) (see \citet{schmid10}) by fitting correction terms deviating from the DMD analysis (see \citet{mrdmd16}) or by fitting a linear dynamical system in the extended phase space by including temporal information as an additional observable (see \citet{mezic16}). Both approaches implicitly assume that the non-autonomous system evolves linearly and are therefore not justified for use in the presence of nonlinearities.

Alternatively, several data-driven methods focus on using autoregressive neural networks to fit temporally aperiodic dynamic data. Notably, \citet{CHAMPNEYS2024474} provides a thorough comparison of various recurrent neural network methods applied to benchmark experimental datasets, and \citet{schar25} uses nonlinear autoregressive functions to model aperiodically excited civil engineering structures. The primary difficulty in fitting general-purpose autoregressive neural networks lies in the unclear roles played by their parameters and hyperparameters. Even if these issues are resolved, achieving reasonable predictive power still demands training on vast datasets that thoroughly explore the phase space of the underlying temporally aperiodic dynamical system. However, even vast amounts of data offer no guarantees of success, as there exist no theoretical results supporting a convergence to a robust neural network model.

%(neural networks to forced data) start with a lot of data nn
To improve black-box neural network models for randomly forced dynamical systems,  \citet{worden17} fit a static Gaussian neural network to positional data to infer nonlinear normal modes of the unforced system. This approach assumes that the unforced nonlinear normal modes persist unchanged under random forcing but offers no predictive reduced dynamical model to predict a general forced response. \citet{simpson2021} overcome this limitation by employing a long-short-term memory (LSTM) neural network to model dynamics in a reduced modal space and an autoencoder mapping that links the reduced space to the phase space. The optimization employed, however, does not guarantee that the resulting reduced model remains predictive under other forcing realizations.

%(finding structure to unforced data that survive forcing)
All approaches discussed so far ultimately aim to predict aperiodic forced responses and implicitly assume the existence of an aperiodic steady state. Theoretical guarantees for such steady states have long existed. For linear systems with external additive forcing, \citet{burd07} derives a closed-form integral expression for the unique aperiodic steady state under arbitrarily large forcing, provided the unforced system has a hyperbolic fixed point. For nonlinear systems, \citet{palmer73} establishes existence and uniqueness for small forcing under similar hyperbolicity conditions. Recently, \citet{haller24_wa} have built on these classic results and extended spectral submanifold (SSM) theory to moderate, uniformly bounded temporally aperiodic forcing. They show that the dominant dynamics lie on an SSM anchored to a unique generalized steady state (GSS), for which they also derive asymptotic equations. 

Building on the latter results, we develop here an automated methodology to predict nonlinear forced response under aperiodic external forcing. We illustrate our methodology by computing the GSS for four benchmark structural dynamical systems, each under the presence of a physically relevant aperiodic external forcing. For the practitioner, we present our methodology as a publicly available MATLAB toolkit, \textit{GSSTool} (see \citet{kaundinya_gsstool}), which takes as input the unforced equations of the nonlinear mechanical system and the external forcing history. We show that \textit{GSSTool} is also applicable in a data-driven fashion when only the linear part of a mechanical system and the aperiodic forcing are known explicitly.

%Here, we focus, on uncovering the GSS from the governing equations of motion of the system or a reliable SSM-reduced model describing the system. We use general results from  \citet{haller24_wa} which provide analytic derivations for a limiting forced response, we will translate these to our setting in section \ref{sec:2} and further outline an automated computation of the GSS in section \ref{sec:3}. In section \ref{sec:4}, we compute the GSS for four different structural systems, a axially moving beam under chirp forcing, an oscillator chain subject to stochastic forcing, a von K\`arm\`an beam undergoing an earthquake and  a von K\`arm\`an shell subject to uniform aperiodic pressure fluctuations. We also include GSS computations performed on two dimensional SSM-reduced models of the chain and von K\`arm\`an beam setups.

\section{Generalized steady state}
\label{sec:2}
\subsection{Setup}
The governing equations of a forced mechanical system are
\begin{equation}
\label{eq:eom}
\mathbf{M}\ddot{\mathbf{x}} + \mathbf{C} \mathbf{\dot{x}}+\mathbf{K} \mathbf{x} + \mathbf{f}(\mathbf{x},\mathbf{\dot{x}}) = \mathbf{g}(\mathbf{x},\mathbf{\dot{x}},t),
\end{equation}
where $\mathbf{x} \in \mathbb{R}^n$ is the vector of generalized positions, $\mathbf{M} \in \mathbb{R}^{n\times n}$ is a symmetric positive definite mass matrix, $\mathbf{K} \in \mathbb{R}^{n\times n}$ is a symmetric positive definite stiffness matrix, $\mathbf{C} \in \mathbb{R}^{n\times n}$ is a symmetric positive semi-definite damping matrix, $\mathbf{f}$ is a smooth nonlinear function with $\mathbf{f}(\mathbf{0},\mathbf{0})=\mathbf{0}$, and $\mathbf{g}$ represents parametric or external forcing. We assume that $\mathbf{g}$ is uniformly bounded, i.e.,  
\begin{equation}
\label{eq:uniform_boundedness}
\|\mathbf{g}(\mathbf{x},\mathbf{\dot{x}},t)\| \leq \Delta, \quad t\in \mathbb{R}, \quad  (\mathbf{x},\dot{\mathbf{x}}) \in V  \subset \mathbb{R}^{2n}, \quad \mathbf{g}(\mathbf{x},\mathbf{\dot{x}},\cdot) \in C^0(\mathbb{R}),
\end{equation}
for constant $\Delta > 0$, where $\|\cdot\|$ denotes the uniform supremum norm, and $V$ is a compact set in $\mathbb{R}^{2n}$. For small enough $\Delta$, we can treat the external forcing $\mathbf{g}$ as a perturbative term and invoke related mathematical results from \citet{haller24_wa}. As we shall see, however, those results also extend to higher values of $\Delta$ in practice. We also note that in the absence of forcing (i.e., for $\mathbf{g}\equiv \mathbf{0}$), system (\ref{eq:eom}) has a fixed point at $(\mathbf{x},\mathbf{\dot{x}}) = (\mathbf{0},\mathbf{0})$, which we assume to be asymptotically stable.

The second-order system (\ref{eq:eom}) can be recast in a first-order form
\begin{equation}
\label{eq:first_order_system}
\mathbf{B} \mathbf{\dot{z}} = \mathbf{A}\mathbf{z} + \mathbf{F}(\mathbf{z}) + \mathbf{G}(\mathbf{z}, t).
\end{equation}

Here $\mathbf{z} = (\mathbf{x},\mathbf{\dot{x}})$, and, as in  \citet{jain2022}, we have
\begin{align}
\label{eq:second_to_first}
\mathbf{B} = \begin{bmatrix}
\mathbf{C} & \mathbf{M} \\
\mathbf{M} & \mathbf{0}
\end{bmatrix}, \quad \mathbf{A} = \begin{bmatrix}
-\mathbf{K} & \mathbf{0} \\
\mathbf{0} & \mathbf{M}
\end{bmatrix}, \quad \mathbf{F}(\mathbf{z}) = \begin{pmatrix}
-\mathbf{f}(\mathbf{x},\mathbf{\dot{x}}) \\
\mathbf{0} \end{pmatrix},  \quad
\mathbf{G}(\mathbf{z},t) = \begin{pmatrix}
\mathbf{g}(\mathbf{x},\mathbf{\dot{x}},t) \\
\mathbf{0}
\end{pmatrix}.
\end{align}

By our assumptions, the origin is asymptotically stable and hence the linear part of the system (\ref{eq:first_order_system}) admits an exponential dichotomy for $\kappa, K >0$ (see \citet{palmer73} and \citet{haller24_wa}): 
\begin{equation}
\label{eq:dichotmy}
|e^{\mathbf{B}^{-1}\mathbf{A}t}| \leq K e^{-\kappa t}, \quad t\geq 0. 
\end{equation}
Our assumptions also imply that the eigenvalues in the spectrum of $\mathbf{B}^{-1}\mathbf{A}$,
\begin{equation}
\label{eq:specBA}
\text{spect}(\mathbf{B}^{-1}\mathbf{A}) = \{\lambda_1,\dots,\lambda_{2n}\}, 
\end{equation}
can be ordered as
\begin{equation}
    \label{eq:ordering}
\text{Re}\lambda_{2n} \leq \text{Re}\lambda_{2n-1} \leq \dots \leq \text{Re}\lambda_2 \leq   \text{Re}\lambda_1  <0.
\end{equation}
We further assume that $\mathbf{B}^{-1}\mathbf{A}$ is semisimple, and hence we have $2n$ complex eigenvectors 
\begin{equation}
\mathbf{v}_1,\dots,\mathbf{v}_{2n} \in \mathbf{C}^{2n}
\end{equation} corresponding to the eigenvalues appearing in the ordered set (\ref{eq:specBA}). 

To present all the computationally relevant results without inverting $\mathbf{B}$ or $\mathbf{M}$, we diagonalize system $\mathbf{B}\mathbf{\dot{z}} = \mathbf{A}\mathbf{z}$ using eigenvectors satisfying the generalized eigenvalue problem
\begin{equation}
\label{eq:gen_eigen_problem}
(\mathbf{A} - \lambda_j \mathbf{B} ) \mathbf{v}_j = \mathbf{0}, \quad j = 1,2,\dots, 2n.
\end{equation}
We construct a diagonalizing transformation matrix $\mathbf{V}\in \mathbb{C}^{2n\times 2n}$ by arranging the eigenvectors $\mathbf{v}_j$ as its columns, in descending order of $\text{Re}\lambda_j$. We also normalize the eigenvectors such that 
\begin{equation}
\label{eq:normalization_fo}
\mathbf{V}^{*} \mathbf{B} \mathbf{V} = \mathbf{I},
\end{equation}  
with $ \mathbf{I}$ the identity matrix and $()^*$ denotes complex conjugate transpose. 

In case of structural damping, the damping matrix can be expressed as $\mathbf{C} = c_{M}\mathbf{M} + c_{K} \mathbf{K}$ for $c_{K},c_{M} \in \mathbb{R}$,  which allows diagonalizing the linear part of (\ref{eq:first_order_system}) using only eigenvectors satisfying the generalized eigenvalue problem
\begin{equation}
\label{eq:second_order_ev_problem}
\left(\mathbf{K} - \omega_j^2 \mathbf{M} \right) \mathbf{u}_j = \mathbf{0}, \quad j = 1,2,\dots, n.
\end{equation}
Here $\omega_j$ is the undamped natural frequency of the undamped and unforced real eigenmode of vibration, $\mathbf{u}_j$. We further normalize these eigenmodes with the mass matrix $\mathbf{u}^\top_i \mathbf{M} \mathbf{u}_j = \delta_{ij}$. We collect the eigenmodes $\mathbf{u}_j$ as  column vectors in ascending order in $\omega_j$ to obtain a matrix $\mathbf{U} \in \mathbb{R}^{n\times n}$. The coordinate change $\mathbf{x} = \mathbf{U}\mathbf{y}$ then transforms the linear part of (\ref{eq:eom}) into a system of $n$ decoupled damped harmonic oscillators, 
\begin{align}
\label{eq:second_order_l}
\ddot{\mathbf y}_j + 2\omega_j \zeta_j \dot{\mathbf  y}_j + \omega_j^2 \mathbf{ y}_j &= 0, \quad j = 1,\dots, n,
\end{align}
where $\zeta_j = \frac{1}{2 \omega_j}\left(c_M + c_K \omega_j^2\right)$ are the modal damping coefficients.  

To streamline these diagonalization transformations for the general and structural damping cases, we define an $n \times n$ transformation matrix $\mathbf{P}$ that diagonalizes the linear part of system (\ref{eq:first_order_system}) via the coordinate change $\mathbf{z} = \mathbf{P}\boldsymbol{\eta}$ into
\begin{equation}
    \label{eq:modal_equations}
    \dot{\boldsymbol\eta} = \boldsymbol\Lambda \boldsymbol\eta.
\end{equation}
Specifically, for a general damping matrix $\mathbf{C}$, we let
\begin{equation}
\label{eq:g_damping_t}
\mathbf{P} = \mathbf{V}
\end{equation}
and
\begin{equation}
\label{eq:g_damping_l}
\boldsymbol{\Lambda} = \mathrm{diag}(\lambda_1, \dots \lambda_{2n})\in \mathbb{C}^{2n \times 2n}. 
\end{equation}
For the case of structural damping, we include a permutation matrix $\mathbf{S} \in \mathbb{R}^{2n\times 2n}$ by letting 
\begin{equation}
\mathbf{S}_{ij} = \begin{cases}
1 & \text{if } i \text{ is odd and } j = \frac{i+1}{2} ,\\[6pt]
1 & \text{if } i \text{ is even and } j = n + \frac{i}{2}, \\[6pt]
0 & \text{otherwise},
\end{cases} \qquad i,j \geq 1,
\end{equation}
and set
\begin{equation}
\label{eq:s_damping_t}
\mathbf{P} =  \begin{bmatrix}
\mathbf{U} & \mathbf{0} \\
\mathbf{0} & \mathbf{U}
\end{bmatrix} \mathbf{S}^{\mathrm{T}},
\end{equation}
which yields 
\begin{equation}
\label{eq:s_damping_l}
\boldsymbol\Lambda = \mathrm{diag}(\boldsymbol{\lambda}_1, \dots, \boldsymbol{\lambda}_n) \in \mathbb{R}^{2n \times 2n}, \quad \boldsymbol{\lambda}_j = \begin{bmatrix}  0 & 1 \\ 
-\omega_j^2 & -2 \zeta_j \omega_j \end{bmatrix}.
\end{equation}
Therefore, $\mathbf{\Lambda}$ is complex and diagonal in the general damping case, but can be chosen as real and block-diagonal in the structural damping case.

 %In challenging applications, it is routine to compute reduced models of the dynamical system (\ref{eq:first_order_system}) from data or equations. Specifically, if we were to perform model reduction to system  (\ref{eq:first_order_system}) onto a $d$-dimensional slow spectral submanifold (SSM) (see \citet{haller16} and \citet{jain2022}), the leading-order generally forced SSM-reduced-model is 
%\begin{align}
%\label{eq:reduced_SSM}
%\dot{\mathbf{r}} &= \boldsymbol{\Lambda} \mathbf{r} + \mathbf{R}(\mathbf{r}) +  (\mathbf{P}^{\dagger})_E \mathbf{G}(\mathbf{0},t,\mathbf{p}), \\ 
%\mathbf{z}(t) &= \mathbf{W}(\mathbf{r}(t)) +  \int_{-\infty}^{t} \mathbf{P}e^{\mathbf{\Lambda} (t-s)} \mathbf{P}^{\dagger} (1-\mathbf{P}_E (\mathbf{P}^\dagger)_E) \mathbf{G}(\mathbf{0},t,\mathbf{p}). \nonumber
%\end{align}
%Here, $(\mathbf{P})_E$ denotes selecting the first $d$ rows, and $\mathbf{P}_E$ denotes selecting the first $d$ columns of the diagonalizing transformation matrix $\mathbf{P}$, which satisfies the normalization condition (\ref{eq:normalization_fo}). The derivation of the additive forcing term in the SSM-reduced model can be found in Appendix B in \citet{Xu25}. This construct implies computing the SSM as a graph over the linear slow $d$ dimensional subspace $E$. We observe that the SSM-reduced dynamics can always be recast in the form (\ref{eq:first_order_system}).

Spectral subspaces of the linear system (\ref{eq:modal_equations}) are defined as direct sums of eigenspaces of $\boldsymbol{\Lambda}$. All spectral subspaces are invariant under the dynamics of system (\ref{eq:modal_equations}). A $d$-dimensional slow subspace $E$ is the direct sum of the eigenspaces corresponding to the eigenvalues $\lambda_1, \dots, \lambda_d,$ as ordered in eq.(\ref{eq:ordering}). The internal dynamics of such a subspace $E$ provides a $d$-dimensional reduced model of system (\ref{eq:modal_equations}) with which all trajectories of (\ref{eq:modal_equations}) synchronize exponentially fast. The theory of spectral submanifolds (SSMs) guarantees that a $d$-dimensional slow subspace has a unique smoothest continuation $\mathcal{W}_E$ in the nonlinear system (\ref{eq:first_order_system}) for $\mathbf{G}\equiv 0$ under appropriate nonresonance conditions on the spectrum of $\boldsymbol{\Lambda}$ (see \citet{haller25}). Specifically, the SSM $\mathcal{W}_E$ is a $d$-dimensional invariant manifold of system that is tangent to E at the origin.  

For small enough $\Delta >0$ in (\ref{eq:uniform_boundedness}), $\mathcal{W}_E$ has been shown to perturb into a time-dependent SSM $\mathcal{W}_E(t)$ that is attached to a generalized steady state (GSS) (see \citet{haller24_wa} and \citet{haller25}). 

In what follows, we develop a detailed computational algorithm for GSSs in mechanical systems of the form (\ref{eq:eom}). As a consequence, the nonlinear steady state response of system (\ref{eq:first_order_system}) is determined by the GSS without the need to compute $\mathcal{W}_E(t)$.

%Our goal is to develop a tool that efficiently computes the underlying core response for additively forced systems of the form (\ref{eq:eom}), (\ref{eq:first_order_system}) and (\ref{eq:reduced_SSM}) based on the assumptions (i) the forcing satisfies (\ref{eq:uniform_boundedness}) and (ii) the unforced system admits an asymptotically stable fixed point. We will also utilize the spectral information of the system to enhance the computational speed and accuracy of our tool. 

\subsection{GSS uniqueness and formal expansion}
For small enough $\Delta$, we can conclude the existence of a unique generalized steady state (GSS) for system (\ref{eq:first_order_system}) using Theorem 1 from \citet{haller24_wa}, under the uniformly bounded forcing assumption (\ref{eq:uniform_boundedness}) and the existence of an asymptotically stable fixed point for the unforced system. Specifically, the theorem states the existence of a unique uniformly bounded GSS $\mathbf{z}^*(t)$ in an open ball $B_{\delta} \subset \mathbb{R}^{2n}$ around $\mathbf{z} =0$. The GSS is asymptotically stable and is as smooth in any parameter as is system (\ref{eq:first_order_system}). For completeness in Appendix \ref{app:thm2}, we present the proof for the local existence and uniqueness of the GSS using the Picard iteration approach by \citet{jain19}, and verify its consistency with the conditions derived in Theorem 1 of \citet{haller24_wa}.

In contrast to a Picard iteration method, for numerical GSS computation, \citet{haller24_wa} seek the GSS as a formal Taylor expansion up to a finite order $N$, 
\begin{equation}
    \label{eq:anchor}
    \mathbf{z}^*(t) = \sum_{\nu = 1}^{N} \mathbf{z}_{\nu}(t)  + o(\Delta^N),
\end{equation}
which is a solution of the nonlinear system (\ref{eq:first_order_system}), i.e. satisfies 

\begin{equation}
\label{eq:gss_equation}
\mathbf{B}\dot{\mathbf{z}}^*(t) = \mathbf{A}\mathbf{z}^*(t) + \mathbf{F}(\mathbf{z}^*(t)) + \mathbf{G}(\mathbf{z}^*(t),t).
\end{equation}
To obtain the time-dependent Taylor coefficients $\mathbf{z}_\nu(t)$ from (\ref{eq:gss_equation}), \citet{haller24_wa} first introduce a small perturbation parameter $\epsilon>0$ by letting $\mathbf{G} \to \epsilon \mathbf{G}$ and $\mathbf{z}_{\nu} \to \epsilon^\nu \mathbf{z}_{\nu}$. They then expand the GSS in $\epsilon$ in (\ref{eq:gss_equation}) and compare terms to $O(\epsilon^\nu)$. At each order, they arrive at a system of inhomogeneous linear ODEs with the homogeneous part given by the linear part of system (\ref{eq:first_order_system}) and an inhomogeneous part that recursively depends on the GSS Taylor coefficients of order $O(\epsilon^{\nu-1})$ and lower.  The recursive formulas are derived from more general multivariate Faa Di Bruno formulas originally obtained by \citet{constantine96}. 

The exact solution for the recursively defined $\mathbf{z}_{\nu}(t)$ is
    \begin{align}
    \label{eq:recursive_terms}
    \mathbf{z}_{\nu}(t) = \int_{-\infty}^t \mathbf{P} e^{\mathbf{\Lambda}(t-s)} \left( \mathbf{P}^* \mathbf{B}\mathbf{P}\right)^{-1}\mathbf{P}^* \Biggl(& \sum_{1\leq |\boldsymbol{\gamma}|\leq \nu} \mathbf{D}_{\mathbf{z}}^{\boldsymbol{\gamma}}\mathbf{F}|_{\mathbf{z}=0} \sum_{q=1}^{\nu}\sum_{\mathbf{k},\mathbf{l} \in p_q(\nu,\boldsymbol{\gamma})}  \mathbf{z}_{\mathbf{l}q}^{\mathbf{k}}(s) \\ &+ 
    \sum_{1\leq |\boldsymbol{\gamma}|\leq \nu-1} \mathbf{D}_{\mathbf{z}}^{\boldsymbol{\gamma}}\mathbf{G}|_{\mathbf{z}=0}\sum_{q=1}^{\nu-1}\sum_{\mathbf{k},\mathbf{l} \in p_q(\nu-1,\boldsymbol{\gamma})}^{} \mathbf{z}_{\mathbf{l}q}^{\mathbf{k}}(s)
    \Biggr) ds , \quad \nu>1,\nonumber
    \end{align}
    with initial condition 
    \begin{align}
    \label{eq:recursive_terms_0}
    \mathbf{z}_1(t) = \int_{-\infty}^{t}  \mathbf{P} e^{\mathbf{\Lambda}(t-s)} \left( \mathbf{P}^* \mathbf{B}\mathbf{P}\right)^{-1}\mathbf{P}^* \mathbf{G}(\mathbf{0},s) ds.
    \end{align}
    Here $\boldsymbol{\gamma}\in \mathbb{N}^{2n}$ with  $|\boldsymbol{\gamma}|=\sum_{i=1}^{2n} \boldsymbol{\gamma}_i$. The factor $\mathbf{z}_{\mathbf{l}q}^{\mathbf{k}}(s)$ is defined as 
    \begin{equation}
    \label{eq:z_factor}
    \mathbf{z}^{\mathbf{k}}_{\mathbf{l}q}(s) = \prod_{j=1}^{q} \prod_{i=1}^{2n} \frac{(\mathbf{z}^i_{({l}_j)}(s))^{\mathbf{k}_{ji}}}{\mathbf{k}_{ji}!}, 
    \end{equation}
    the operator $\mathbf{D}_{\mathbf{z}}^{\boldsymbol{\gamma}}$ as
    \begin{equation}
    \mathbf{D}_{\mathbf{z}}^{\boldsymbol{\gamma}} = \frac{\partial^{|\boldsymbol{\gamma}|}}{\partial \mathbf{z}_1^{\boldsymbol{\gamma}_1} \dots \mathbf{z}_{2n}^{\boldsymbol{\gamma}_{2n}}},
    \end{equation}
    and the index set $p_q(\nu,\boldsymbol{\gamma})$ is defined as
    \begin{equation}
    \label{eq:index_set}
    p_q(\nu,\boldsymbol{\gamma}) = \left\{ (\mathbf{k},\mathbf{l}): \mathbf{k} \in \mathbb{N}^{q \times 2n}-\{\mathbf{0}\},\mathbf{l} \in \mathbb{N}^q - \{\mathbf{0}\}, {l}_1<\dots<{l}_q, \sum_{j=1}^q \mathbf{k}_{ji} = \boldsymbol{\gamma}_i, \sum_{i=1}^{2n}\sum_{j=1}^q \mathbf{k}_{ji} {l}_j = \nu \right\}.
    \end{equation}

    The detailed proof of these formulas is provided in Appendix A2 in \citet{haller24_wa}. We highlight practical advantages of these formulas in the following three remarks:

\begin{rem}
\textbf{[Applicability to practical settings]} 
The formal asymptotic GSS expansion (\ref{eq:anchor}) can be readily computed for any given smooth functions $\mathbf{F}$ and $\mathbf{G}$ in $\mathbf{z}$. These nonlinear functions are polynomials in benchmark structural engineering problems modeled using the method of finite elements, which allows for a quick computation of the coefficients in eq. (\ref{eq:recursive_terms}). Our formulas do not assume that the external forcing $\mathbf{G}(\mathbf{0},t)$ is $C^0$ in $t$, thus we can compute unique GSS approximations for any uniformly bounded stochastic forcing realization. 
\end{rem}

\begin{rem}
\textbf{[GSS expansion with an explicit scalar forcing level]} 
It is common in structural vibration studies to introduce a scalar parameter that controls the magnitude of forcing on the system. In our context, this can be achieved by letting $\mathbf{G} \to \Delta \mathbf{\tilde{G}}$, where $\mathbf{\tilde{\mathbf{G}}} = \frac{\mathbf{G}}{\|\mathbf{G}\|}$ is a normalized forcing signal. The GSS expansion (\ref{eq:anchor}) with the forcing parameter $\Delta$ made explicit is of the form \begin{equation}
\label{eq:GSS_single_forcing}
\mathbf{z}^*(t,\Delta) = \sum_{\nu=1}^{N} \mathbf{\tilde{z}}_{\nu}(t) \Delta^\nu + o(\Delta^N).
\end{equation}
Here $\tilde{\mathbf{z}}_\nu(t)$ is computed by substituting the redefined forcing function $\mathbf{\tilde{G}}$ in our GSS into formulas (\ref{eq:recursive_terms}) and (\ref{eq:recursive_terms_0}). Setting $\Delta = \|\mathbf{G}\|$ in (\ref{eq:GSS_single_forcing}), we recover the GSS for the forced system (\ref{eq:first_order_system}). This parametric formula for the GSS allows for straightforward computation of the GSS for any user-defined forcing level using the already computed GSS Taylor expansion coefficients for the normalized forcing function. 
\end{rem}

\begin{rem}
\textbf{[Computational advantage]} 
The expansion coefficients (\ref{eq:recursive_terms}) are uniformly bounded solutions of recursively defined linear inhomogeneous ODEs. Solving such linear ODEs is faster than solving nonlinear forced ODEs. In addition to this speed-up, solution methods for linear inhomogeneous ODEs are guaranteed to converge unlike their nonlinear counterparts, which need further tuning by the user. Here, we will provide numerical techniques to approximate the integrals in (\ref{eq:recursive_terms}) directly. 
\end{rem}

\begin{rem}
\textbf{[Computing GSS on a forced SSM-reduced model]} 
When the nonlinearities of system (\ref{eq:eom}) are unknown, we compute the GSS approximation (\ref{eq:anchor}) using the reduced equations of the slow SSM $\mathcal{W}_E(t)$ obtained from data. Specifically, given access to unforced decaying trajectories, we use SSMLearn \citep{cenedese21} to identify a $d$-dimensional autonomous slow SSM with reduced coordinates $\mathbf{r} \in \mathbb{R}^d$
, SSM parametrization $\mathbf{W}(\mathbf{r})$, and reduced dynamics $\mathbf{R}(\mathbf{r})$. Combining this SSM with the known linear parts of system (\ref{eq:eom}), we then infer the leading-order forced SSM-reduced model, 
\begin{align}
\label{eq:reduced_SSM}
\dot{\mathbf{r}} &= \mathbf{R}(\mathbf{r}) +  (\mathbf{P}^{*})_E \mathbf{G}(\mathbf{0},t), \\ 
\mathbf{z}(t) &= \mathbf{W}(\mathbf{r}(t)) +  \int_{-\infty}^{t} \mathbf{P}e^{\mathbf{\Lambda} (t-s)} \mathbf{P}^{*} (1-\mathbf{P}_E (\mathbf{P}^*)_E) \mathbf{G}(\mathbf{0},t). \nonumber
\end{align}
Here $(\mathbf{P})_E$ denotes selecting the first $d$ rows, and $\mathbf{P}_E$ denotes selecting the first $d$ columns of the diagonalizing transformation matrix $\mathbf{P}$, which satisfies the normalization condition (\ref{eq:normalization_fo}). The derivation of the leading order terms in the SSM-reduced model (\ref{eq:reduced_SSM}) can be found in Appendix B in \citet{Xu25}. Setting $\mathbf{B} \equiv \mathbf{I}$, $\mathbf{A} = \mathbf{D}_{\mathbf{r}}\mathbf{R}(\mathbf{0})$, $\mathbf{F} \equiv \mathbf{R}(\mathbf{r}) - \mathbf{D}_{\mathbf{r}}\mathbf{R}(\mathbf{0}) $  and $\mathbf{G} \equiv (\mathbf{P}^{*})_E \mathbf{G}(\mathbf{0},t)$ in eq. (\ref{eq:gss_equation}), we compute the GSS Taylor coeffiencits using formulas (\ref{eq:recursive_terms}) in the SSM reduced coordinates, then using the SSM parametrization map (\ref{eq:reduced_SSM}) we lift these GSS Taylor coefficients to the physical space. 
\end{rem}

\section{Automated computation of the GSS expansion}
\label{sec:3}
The computation of the index set (\ref{eq:index_set}) can be expedited using recursive definitions for automated computation of functional compositions of Taylor coefficients based on radial derivatives, as implemented in \textit{SSMTool} for polynomial nonlinearities (see \citet{haro06} and \citet{ponsioen2020}). We discuss this approach below with numerical methods that accurately evaluate the kernel integral in (\ref{eq:recursive_terms}).

\subsection{Iterated recursion via multi-index notation}

We assume the function $\mathbf{F}$ and $\mathbf{G}$ contain only polynomial nonlinearities in $\mathbf{z}$.  Using the multi-index convention for polynomial nonlinearities introduced in \citet{jain2022} and \citet{thurner23}, we let 
\begin{equation}
\label{eq:multiindex_notation}
\mathbf{F}(\mathbf{z}) = \sum_{\mathbf{m} \in \mathbb{N}^{2n}} \mathbf{F}_{\mathbf{m}} \mathbf{z}^{\mathbf{m}} 
\end{equation}
and
\begin{equation}
\mathbf{G}(\mathbf{z},t) = \sum_{\mathbf{m} \in \mathbb{N}^{2n}} \mathbf{G}_{\mathbf{m}}(t) \mathbf{z}^{\mathbf{m}}, 
\end{equation}

with the notation $\mathbf{z}^{\mathbf{m}} = {z}_1^{{m_1}}\dots {z}_{2n}^{{m_{2n}}} $. We substitute these polynomial expansions into (\ref{eq:recursive_terms}) and focus on the evaluation of the integrand
\begin{equation}
\label{eq:poly_expr}
 \sum_{1\leq |\boldsymbol{\gamma}|\leq \nu} \mathbf{F}_{\boldsymbol{\gamma}} \prod_{i=1}^{2n} \boldsymbol{\gamma}_i! \sum_{q=1}^{\nu}\sum_{p_q(\nu,\boldsymbol{\gamma})}  \mathbf{z}_{\mathbf{l}q}^{\mathbf{k}}(s) +
    \sum_{1\leq |\boldsymbol{\gamma}|\leq \nu-1} \mathbf{G}_{\boldsymbol{\gamma}}(s) \prod_{i=1}^{2n} \boldsymbol{\gamma}_i!
    \sum_{q=1}^{\nu-1}\sum_{p_q(\nu-1,\boldsymbol{\gamma})}^{} \mathbf{z}_{\mathbf{l}q}^{\mathbf{k}}(s)
    .
\end{equation}

As shown by \citet{thurner23} and implemented in \textit{SSMTool}, the $O(x^\nu)$ terms in the expansion of a 1D autonomous SSM with reduced coordinate $x$ is of the form 
\begin{equation}
 \sum_{1\leq |\boldsymbol{\gamma}|\leq \nu} \mathbf{F}_{\boldsymbol{\gamma}} \mathbf{H}_{\boldsymbol{\gamma},\nu} , 
\end{equation}
where coefficients $\mathbf{H}_{\boldsymbol{\gamma},\nu}$ records contributions arising from function compositions. In the case of the GSS expansions, we must allow for general time-dependence in these coefficients and include the contributions due to external forcing, which gives
\begin{equation}
\label{eq:radial_GSS}
\boldsymbol{\Phi}_{\nu}(s) = 
\begin{cases}
\mathbf{G}_{\mathbf{0}}(s) & \text{if } \nu = 1, \\[1em]
\displaystyle
\sum_{1 \leq |\boldsymbol{\gamma}| \leq \nu} \mathbf{F}_{\boldsymbol{\gamma}} \mathbf{H}_{\boldsymbol{\gamma},\nu}(s)
+ \sum_{1 \leq |\boldsymbol{\gamma}| \leq \nu-1} \mathbf{G}_{\boldsymbol{\gamma}}(s) \mathbf{H}_{\boldsymbol{\gamma},\nu-1}(s) & \text{if } \nu > 1.
\end{cases}
\end{equation}
with the time-dependent coefficients $\mathbf{H}_{\boldsymbol{\gamma},\nu}(s)$ defined as
\begin{equation}
\label{eq:radial_derivative}
\mathbf{H}_{\boldsymbol{\gamma},\nu}(s)  =\Big[\left(\mathbf{z}^*(s)  \right)^{\boldsymbol{\gamma} }\Big]_{O(\Delta^{\nu})} = \left[\sum_{q=1}^{\nu}\left(\sum_{j=1}^{q} \mathbf{z}_{({l}_j)}(s)  \right)^{\boldsymbol{\gamma} }\right]_{O(\Delta^{\nu})}.
\end{equation}
Here $[\cdot]_{O(\Delta^\nu)}$ denotes all terms contributing at order $\nu$. As we show in Appendix \ref{app:radial_derivative_Faa_Di_bruno}, the algebraic expression (\ref{eq:radial_GSS}) is equivalent to the GSS integrand in (\ref{eq:poly_expr}). We will use formula (\ref{eq:radial_GSS}) to compute efficiently the integrand (\ref{eq:poly_expr}) for our GSS algorithm.  

The quantities  $\mathbf{H}_{\boldsymbol{\gamma},\nu}(s)$ at a fixed value of $s$, are defined recursively in \citet{thurner23}, and are implemented in \textit{SSMTool}. For autonomous problems, judicious use of these compositional coefficients have provided major speedup for the computation of spectral submanifolds (SSMs). For the time-dependent calculations required for GSSs, the same approach is computationally inefficient.

To this end, we modify the compositional structure to store all time-dependent functions $\mathbf{H}_{\boldsymbol{\gamma},\nu}(s)$ as 3-tensors with size $\mathbf{H} \in \mathbb{R}^{2n \times N \times T}$ where $N$ is the maximal order of the GSS expansion and $T$ is the total number of temporal data points of the forcing signal. In practice there will always be a finite but large time interval over which the forcing function $\mathbf{G}$ is given. The value of $T$ will then be fixed by the temporal data matrix $\mathbf{G}_{\mathbf{m}} \in \mathbb{R}^{2n \times T}$ supplied by the user. These tensorial definitions retain the same recursive properties of $\mathbf{H}_{\boldsymbol{\gamma},\nu}(s)$ but with new rules that utilize tensor products to enhance computational speed.   

\subsection{Evaluation of integrals in the GSS formulas}
For $\nu=1$, the formula (\ref{eq:radial_GSS}) is the external forcing applied to the system, 
\begin{equation}
\boldsymbol{\Phi}_1(t) = \mathbf{G}(\mathbf{0},t) = (\mathbf{g}(\mathbf{0},\mathbf{0},t),\mathbf{0})^\mathrm{T}.
\end{equation}
Due to the recursive nature of the definitions and the block structure of the mechanical system, the integrand (\ref{eq:radial_GSS}) for $\nu > 1$ always has the form $\boldsymbol{\Phi}_{\nu}(t) = (\boldsymbol{\phi}_{\nu}(t), \mathbf{0})^\mathrm{T}$ with $\boldsymbol{\phi}_{\nu}(t) \in \mathbb{R}^{n}\times \mathbb{R}$.

 In practical applications, the external forcing will be defined at $T$ equally spaced temporal points over a finite time interval $[t_0,t_f]$ with uniform spacing $\delta t$. For all times $t<t_0$, we set the forcing to zero in evaluating the integrals in formulas (\ref{eq:recursive_terms})-(\ref{eq:recursive_terms_0}). 
 Hence, we shall always pad the forcing signal with zeros up-to a finite length $T_p \delta t$, where $T_p$ controls the padding length.     
 
 With these considerations, we obtain 
\begin{equation}
\label{eq:Integral_Fredholm}
\mathbf{z}_{\nu} (t) =  \int_{-\infty}^t \mathbf{P}e^{\mathbf{\Lambda}(t-s)} \left( \mathbf{P}^* \mathbf{B}\mathbf{P}\right)^{-1}\mathbf{P}^* \boldsymbol{\Phi}_{\nu}(s) ds \approx \int_{-T_p \delta t}^t \mathbf{P}e^{\mathbf{\Lambda}(t-s)} \left( \mathbf{P}^* \mathbf{B}\mathbf{P}\right)^{-1}\mathbf{P}^* \boldsymbol{\Phi}_{\nu}(s) ds,
\end{equation}
 a Fredholm integral of the second kind (see \citet{Atkinson2015}). \citet{jain19} evaluated similar integrals to (\ref{eq:Integral_Fredholm}) as convolutions with Newton-Cotes quadrature rules. Newton-Cotes rules require integrand values at equally spaced intervals (see \citet{abramowitz1965handbook}). These rules are then formulated by approximating the integral as a weighted sum of these values. However, for exponential kernels there is no guarantee that these integral evaluations with Newton-Cotes rules converge for decreasing interval spacing or higher-order rules. 

An alternative strategy is to perform  Gaussian quadrature (see \citet{abramowitz1965handbook}), which are guaranteed to converge but require values of the integrand at arbitrary temporal spacings. Such schemes suffer numerically for stiff ODEs, which is commonly the case for structural systems that contain stiff degrees of freedom.

For the mechanical system (\ref{eq:eom}), implicit Newmark methods circumvent stiffness issues (see \citet{newmark59} and \citet{GeradinRixen2015}). In fact, the expression (\ref{eq:Integral_Fredholm}) can be obtained by an implicit Newmark time integration of the second-order linear inhomogeneous system 
\begin{equation}
\label{eq:newmark}
\mathbf{M} \ddot{\mathbf{x}}_{\nu} + \mathbf{C} \dot{\mathbf{x}}_{\nu} + \mathbf{K} \mathbf{x}_{\nu}  = \boldsymbol{\phi}
_\nu(t+(t_0-T_p\delta t)), \quad \mathbf{z}_{\nu}(0) = \mathbf{0}, \quad  \mathbf{z}_{\nu} = (\mathbf{x}_{\nu},\dot{\mathbf{x}}_\nu)^\mathrm{T},
\end{equation}
over the time interval $[T_p \delta t ,T_p \delta t +t_f ]$. 

Instead of this approach, however, we use here a method that directly calculates the integral (\ref{eq:Integral_Fredholm}) without the stiffness issues faced by quadrature rules and is also applicable to first-order systems of the form (\ref{eq:first_order_system}). 

We set this up by rewriting (\ref{eq:Integral_Fredholm}) as a recursive solution 
\begin{equation}
\label{eq:fund_integral}
 \mathbf{z}_{\nu}(t+\delta t) = \mathbf{P}e^{\mathbf{\Lambda}\delta t } \left( \mathbf{P}^* \mathbf{B}\mathbf{P}\right)^{-1}\mathbf{P}^*\mathbf{z}_{\nu}(t) +  \int_{0}^{\delta t} \mathbf{P}e^{\mathbf{-\Lambda}s} \left( \mathbf{P}^* \mathbf{B}\mathbf{P}\right)^{-1}\mathbf{P}^* \boldsymbol{\Phi}_\nu(s+t) ds,
\end{equation}
with initial condition
\begin{equation}
\mathbf{z}_{\nu}(-T_p\delta t)=\mathbf{0},
\end{equation}
and compute the finite integral over the interval $[0,\delta t]$.

\subsubsection{Case 1: General damping}
\label{subsec:general_damping}
From Section 2.1, we substitute for $\boldsymbol{\Lambda}$ and $\mathbf{P}$ in the finite integral in (\ref{eq:fund_integral}) with their  general damping counterparts (\ref{eq:g_damping_l}) and (\ref{eq:g_damping_t}), respectively. Using the normalization condition (\ref{eq:normalization_fo}), the resulting integral expression can be represented as 
\begin{equation}
\label{eq:L_j_gd}
\mathbf{V}\mathbf{L}(t,\delta t),
\end{equation}
with $\mathbf{L}(t,\delta t) \in \mathbb{R}^{2n}$ defined as 
\begin{equation}
\label{eq:l_ref_gd}
\mathbf{L}_j(t,\delta t) =\int_0^{\delta t}  e^{-\lambda_j s} \left[\mathbf{V}^* \boldsymbol{\Phi}_\nu(s+t) \right]_{j} ds\quad  j =1,\dots, 2n .
\end{equation}

We approximate the forcing vector signal as a piecewise linear function over the interval $[t, t+\delta t]$. This implies 
\begin{equation}
\label{eq:pwl}
\boldsymbol\Phi_\nu(s+t) = \boldsymbol\Phi_\nu(t) + \frac{s}{\delta t } \left(\boldsymbol\Phi_\nu(t+\delta t) - \boldsymbol\Phi_\nu(t)\right).
\end{equation}
This allows us to explicitly calculate scalar integrals defining $\mathbf{L}_j$ in (\ref{eq:l_ref_gd}) as 
\begin{equation}
\label{eq:first_order_integral}
\mathbf{L}_j(t,\delta t) = \mathbf{Q}_j^{\mathrm{T}}(\delta t )  \begin{pmatrix}
 \left[\mathbf{V}^* \boldsymbol{\Phi}_\nu(t) \right]_{j} \\
\left[\mathbf{V}^* \boldsymbol{\Phi}_\nu(t+\delta t) \right]_{j}
\end{pmatrix},
\end{equation}
with 
\begin{align}
\mathbf{Q}_j(\delta t ) = \begin{pmatrix}
\frac{e^{\lambda_j\delta t}-1}{\lambda_j} - \frac{1}{\delta t}\left(\frac{e^{\lambda_j\delta t}-1}{\lambda_j^2}-\frac{\delta t}{\lambda_j}\right)\\
\frac{e^{\lambda_j\delta t}-1}{\lambda_j^2}-\frac{\delta t}{\lambda_j}
\end{pmatrix}.
\end{align}
We have to compute $\mathbf{Q}_j$ only once for a user supplied value of $\delta t$ for the evaluation of the integral (\ref{eq:Integral_Fredholm}). The vectors $\mathbf{Q}_j$ are well defined for $\lambda_j \neq 0$, which is always the case for damped mechanical systems. For periodic or quasi-periodic forcing, we can evaluate these integrals exactly without the piecewise approximation for the function $\boldsymbol{\Phi}_\nu(s)$. However, in those cases, the vector $\mathbf{Q}_j$ is well defined only when none of the eigenvalues $\lambda_j$ are in resonance with the external forcing frequencies.

\subsubsection{Case 2: Structural damping}
\label{subsec:structural_damping}
For structural damping, the matrix $\mathbf{\Lambda}$ takes a block-diagonal structure (see eq. (\ref{eq:s_damping_l})) given by 
\begin{equation}
\label{eq:block_second_order}
\boldsymbol\Lambda = \mathrm{diag}(\boldsymbol{\lambda}_1, \dots, \boldsymbol{\lambda}_n), \quad \boldsymbol{\lambda}_j \in \mathbb{R}^{2 \times 2},\qquad  \boldsymbol{\lambda}_j = \begin{bmatrix}  0 & 1 \\ 
-\omega_j^2 & -2 \zeta_j \omega_j \end{bmatrix},\quad    \omega_1<\dots < \omega_n.
\end{equation}
Substituting for the modal transformation $\mathbf{P}$ with eq. (\ref{eq:s_damping_t}) and $\mathbf{B}$ from eq. (\ref{eq:second_to_first}),  the finite integral in (\ref{eq:fund_integral}) takes the form
\begin{equation}
   \int_{0}^{\delta t} \begin{bmatrix}
\mathbf{U} & \mathbf{0} \\
\mathbf{0} & \mathbf{U}
\end{bmatrix} \mathbf{S}^{\mathrm{T}}e^{\mathbf{-\Lambda}s} \mathbf{S} \begin{bmatrix} 
\mathbf{0}& \mathbf{1} \\
\mathbf{1} & -\mathbf{U}^\mathrm{T}\mathbf{C}\mathbf{U}
\end{bmatrix} \begin{bmatrix} 
\mathbf{U}^\mathrm{T} & \mathbf{0} \\
\mathbf{0} & \mathbf{U}^\mathrm{T}
\end{bmatrix}  \boldsymbol{\Phi}_\nu(s+t) ds,
\end{equation} 
which can be simplified and written as
\begin{equation}
   \int_{0}^{\delta t} \begin{bmatrix}
\mathbf{U} & \mathbf{0} \\
\mathbf{0} & \mathbf{U}
\end{bmatrix} \mathbf{S}^{\mathrm{T}}e^{\mathbf{-\Lambda}s} \mathbf{S} \begin{pmatrix}   \mathbf{0} \\\mathbf{U}^\mathrm{T} \boldsymbol{\phi}_\nu(s+t) \end{pmatrix}  ds.
\end{equation}
Rearranging the terms in the integral, we obtain
\begin{equation}
\begin{bmatrix}
\mathbf{U} & \mathbf{0} \\
\mathbf{0} & \mathbf{U}
\end{bmatrix} \mathbf{S}^{\mathrm{T}} \cdot \mathbf{L}(t,\delta t),
\end{equation}
with $\mathbf{L}(t,\delta t) \in \mathbb{R}^{2n}$ defined as 
\begin{equation}
\label{eq:sd_L_int}
\mathbf{L}_j(t,\delta t) = \int_{0}^{\delta t} e^{-\boldsymbol{\lambda}_js} \begin{pmatrix}\boldsymbol{0} \\
[\mathbf{U}^\top \boldsymbol\phi_\nu(s+t)]_j\end{pmatrix}  ds,\quad  j =1,\dots, n .
\end{equation}
Under the assumption of piecewise linearity for the forcing function, eq. (\ref{eq:pwl}) implies that we can express the integral (\ref{eq:sd_L_int}) as
\begin{equation}
\label{eq:second_order_integral}
\mathbf{L}_j(t,\delta t)= \mathbf{Q}_j(\delta t )  \begin{pmatrix}
 [\mathbf{U}^\top \boldsymbol\phi_\nu(t)]_j\\
 [\mathbf{U}^\top \boldsymbol\phi_\nu(t+\delta t)]_j
\end{pmatrix}.
\end{equation} 
Here, the matrix $\mathbf{Q}_j(\delta t) \in \mathbb{R}^{2 \times 2}$  is 
\begin{align}
\label{eq:sd_q}
\mathbf{Q}_j(\delta t ) = \begin{cases}
\mathbf{Q}^u_j (\delta t ), \quad \zeta_j < 1 \quad (\text{underdamped}) \\ 
\mathbf{Q}^c_j(\delta t ), \quad  \zeta_j = 1 \quad (\text{critically damped})\\
\mathbf{Q}^o_j (\delta t ), \quad \zeta_j > 1 \quad (\text{overdamped})
\end{cases}.
\end{align}

The exact expressions for these matrices and their derivation is detailed in Appendix \ref{app:second_qj}. The above formulas preserve the second order structure and are solely expressed in terms of the vibration modes of the mechanical system. Once the matrices $\mathbf{Q}_j(\delta t)$ are precomputed for the dynamical system (\ref{eq:eom}), they can be  readily used in GSS calculations for arbitrary forcing functions. 

Formula (\ref{eq:second_order_integral}) for a single-degree-of-freedom, underdamped oscillator were first derived by \citet{Dempsey_Duhammel_78}. Our more general formulas (\ref{eq:sd_L_int})-(\ref{eq:sd_q}) apply to multi-degree-of-freedom systems with arbitrary damping. 

\subsubsection{Modal truncation}
\label{subsec:mode_select}
A complete diagonalization of the linear part of a high-dimensional mechanical system is computationally expensive. For this reason, we employ a modal truncation to approximate formulas (\ref{eq:first_order_integral}) and (\ref{eq:second_order_integral}). Specifically, the matrices $\mathbf{V}$ or $\mathbf{U}$ are truncated to contain only $m$ eigenvectors or mode shapes. If we represent these truncated transformation matrices as  $\mathbf{V}_{\mathrm{trunc}} \in \mathbb{C}^{2n \times m}$ or $\mathbf{U}_{\mathrm{trunc}} \in\mathbb{R}^{n \times m}$, the corresponding formulas read
\begin{equation}
\label{eq:formulas}
\mathbf{Q}_j^\mathrm{T}(\delta t )  \begin{pmatrix}
 \left[\mathbf{V}_{\mathrm{trunc}}^* \boldsymbol{\Phi}_\nu(t) \right]_{j} \\
\left[\mathbf{V}_{\mathrm{trunc}}^* \boldsymbol{\Phi}_\nu(t+\delta t) \right]_{j}
\end{pmatrix},  \quad j \in \{i_1,\dots i_m\},
\end{equation}

\begin{equation}
\label{eq:formulas}
  \mathbf{Q}_j(\delta t )  \begin{pmatrix}
 [\mathbf{U}_{\mathrm{trunc}}^\mathrm{T}\boldsymbol\phi_\nu(t)]_j\\
 [\mathbf{U}_{\mathrm{trunc}}^\mathrm{T} \boldsymbol\phi_\nu(t+\delta t)]_j
\end{pmatrix}, \quad \quad j \in \{i_1,\dots i_m\}.
\end{equation}
Here the index set $\{i_1,\dots i_m\}$ labels the $m$ selected eigenvectors or modes. We typically select eigenmodes whose corresponding eigenvalues satisfy 
\begin{equation}
 e^{\delta t \mathrm{Re\lambda_j}} > \epsilon,
\end{equation}
for a tolerance value $\epsilon>0$, which we set by default to $\epsilon =10^{-3}$. The dimensionless quantity $e^{\delta t\mathrm{Re\lambda_j}}$ controls the magnitude of the entries in $\mathbf{Q}_j(\delta t)$. This modal approximation selects the eigenvectors corresponding to the $m$ slowest eigenvalues. 

\subsection{Vector Pad\'e approximations for the GSS}
\label{subsec:vector_pade}
Even though a unique GSS exists, its time-dependent Taylor approximation may diverge for large magnitudes of the forcing vector. A strategy to enhance the domain of convergence of our computed GSS Taylor approximation is to construct a Pad\'e approximation from its coefficients (see \citet{Baker_Graves-Morris_1996}). 
\citet{kaszas2025} discuss in detail Pad\'e approximations to enlarge the domain of convergence of Taylor approximations of spectral submanifolds. In their setting, the Taylor coefficients are constant for all times. We could apply similar methods at fixed time instances, but this would lead to recomputing a Pad\'e approximation of the GSS for each each time instance, which is inefficient. 

We seek a universal rational function approximation of the GSS that is proven to have convergent properties for all temporal points. \citet{lane24} have recently used a vector Pad\'e approximation to approximate the steady state solution for the forced Navier-Stokes equations. Applying this idea to each phase space variable $\mathbf{z}^j$, we employ a vector Pad\'e  approximation,   
\begin{equation}
\label{eq:pade}
\mathbf{z}^{*j}(t) \approx \boldsymbol{z}^j = [\mathbf{z}^{*j}(t_1), \dots, \mathbf{z}^{*j}(t_p)]   = \text{Pad\'{e}}[L,M] \overset{\mathrm{def.}}{=} \frac{\sum_{\nu = 1}^{\nu = L} \boldsymbol{a}^j_\nu  }{1 + \sum_{\nu = 1}^{\nu = M} b^j_\nu } + o(\Delta^{L+M}), 
\end{equation}
with $\boldsymbol{z}^j,\boldsymbol{a}^j_\nu \in \mathbb{R}^{p} \text{ and }  b^j_\nu \in \mathbb{R}$. Here $p$ is the total number of temporal points for which the Taylor GSS values are known. We solve for $\boldsymbol{a}^j_\nu$ and $b^j_\nu$ in eq. (\ref{eq:pade}), first calculating a Taylor GSS approximation of order $N = L +M$,
\begin{equation}
\label{eq:taylor}
\boldsymbol z^j =  \sum_{\nu = 1}^{N = L+M}  \boldsymbol z^j_{\nu} + o(\Delta^N).  
\end{equation}
Equating eq. (\ref{eq:taylor}) with the Pad\'e approximation in eq. (\ref{eq:pade}) and collecting terms at $o(\Delta^{L+1})$ and higher, we obtain the following linear system of equations in $b^j$ 
\begin{align}
\label{eq:lin_pade}
\begin{pmatrix}
\boldsymbol{z}^j_L & \cdots & \boldsymbol{z}^j_{L-M+1} \\
\boldsymbol{z}^j_{L+1} & \cdots & \boldsymbol{z}^j_{L-M+2} \\
\vdots & \vdots & \vdots \\
\boldsymbol{z}^j_{L+M-1} & \cdots & \boldsymbol{z}^j_L
\end{pmatrix}
 \begin{pmatrix}
b^j_1 \\
\vdots \\
b^j_M
\end{pmatrix}=-\begin{pmatrix}
\boldsymbol{z}^j_{L+1} \\
\vdots \\
\boldsymbol{z}^j_{L+M}.
\end{pmatrix}. 
\end{align}
The linear system is generally overdetermined as $p>>L+M$, therefore, we seek the solution of $b^j$, using the least squares method. Using the obtained $b^j$, the vector coefficients $\boldsymbol{a}^j$ of the Pad\'e approximation are determined by 
\begin{align}
\boldsymbol{a}^j_1 &= \boldsymbol{z}^j_1, \\
\boldsymbol{a}^j_\nu &=  \boldsymbol{z}^j_{\nu} + \sum_{\mu = 1}^{\nu-1} b^j_{\mu} \boldsymbol{z}^j_{\nu-\mu}, \quad  \quad 1<\nu \leq L.
\end{align}
These formulas are derived from collecting terms of order $o(\Delta^{L})$ and smaller. If the true GSS $\mathbf{z}^*(t)$ is a rational function with an analytic numerator, then Theorem 8.4.9 from \citet{Baker_Graves-Morris_1996} provides convergence results for its vector Pad\'e approximation as $L=M$ and $M \to \infty$. Motivated by this result, we will always compute a vector Pad\'e approximation for the case $L=M$ for large values of $M$. 

\section{Examples}
\label{sec:4}
In this section, we illustrate on examples of increasing complexity the predictive power and accuracy of the formulas and numerical procedures we have discussed for computing GSSs. In our examples, the forcing function $\mathbf{g}$ will have pure time dependence without state dependence. 

Some of our examples will be fully equation-driven, i.e., will utilize full knowledge of the equations of motion of the mechanical system. Other examples will be partially equation driven in that they will only assume knowledge of the linear part of the system and infer its nonlinear part from data. 

We will approximate GSSs by computing $O(\mathrm{N})$ Taylor expansions (see eq. (\ref{eq:anchor})). If the $O(\mathrm{N})$ Taylor expansions turns out to be divergent for a given forcing level, we compute the $\text{Pad\'e}[\mathrm{N},\mathrm{N}]$ GSS expansion using eq. (\ref{eq:pade}) with $O(2\mathrm{N})$ GSS Taylor approximations. All our examples and methods discussed here are available in the open source MATLAB code \textit{GSSTool} (see \citet{kaundinya_gsstool}).

\subsection{Example 1: Axially moving beam under chirp excitation}
\label{subsec:e1}
We consider a simply supported slender viscoelastic beam that moves with a constant velocity $v$ in the axial direction. This example emulates the dynamics of underwater cables and power transmission lines, where axial motion comes from the constant towing speed of underwater cables or waves moving along overhead power lines (see \citet{Pellicano_1999}). The constant axial movement of the beam creates a gyroscopic damping force whose linear part does not satisfy the structural damping hypothesis (see \citet{Pellicano_1999} and \citet{li22a}). 

The beam is made of steel with density $\rho = 7680 \text{ } [\mathrm{kg/m^3}]$, Young's modulus $E = 30 \times 10^{9} \text{ } [\mathrm{Pa}]$, and viscosity $\nu = 30 \times 10^{5} \text{ } [\mathrm{Pa}]$. The beam's length is $ l = 1 \text{ } [\mathrm{m}]$ and its cross-sectional area $A = 1.2 \times 10^{-3} \text{ } [\mathrm{m^2}]$. The axial velocity of the beam is $v = 43.82 \text{ [m/s]}$. The governing partial differential equations (PDEs) of the beam are discretized using the Galerkin approach with Fourier basis functions to obtain a second order system of the form (\ref{eq:eom}). The exact implementation of the discretized equations of motion of the beam can be found in the GitHub repository \textit{SSMTool} (see \citet{jain23}). The discretized axial beam has $n =10$ degrees of freedom and its transverse coordinates are given by
\begin{equation}
w(x,t) = \sum_{j=1}^{n} \sin(\pi jx)\mathbf{u}_j(t),
\end{equation}
where $\mathbf{u}_j$ denote the Galerkin modes. The geometry of the moving beam is depicted in Fig.\ref{fig:setup_AB}.  We transversely excite the structure at the base with a chirp forcing signal whose instantaneous frequency varies in time as
\begin{equation}
\label{eq:chirp_sig}
\Omega(t) = \sin(2\pi(t^2/2+\omega t)),
\end{equation}
where $\omega \approx 3.2 \text{ } [\mathrm{Hz}]$ is the natural frequency of the slowest mode. Formula (\ref{eq:chirp_sig}) defines a linear chirp signal with a starting frequency of $\omega$ which we apply to the beam over a duration of $10 \text{ [s}]$ (see Fig.\ref{fig:chirp_force}). The maximal amplitude of the forcing $\Delta$ is $0.63 \text{ [N]}$. This forcing level is approximately $600$ times larger than the value used to periodically excite the beam by \citet{li22a}.

\begin{figure}[H]
\centering
\begin{subfigure}{0.45\textwidth}
    \raisebox{2cm}{\includegraphics[width=\textwidth]{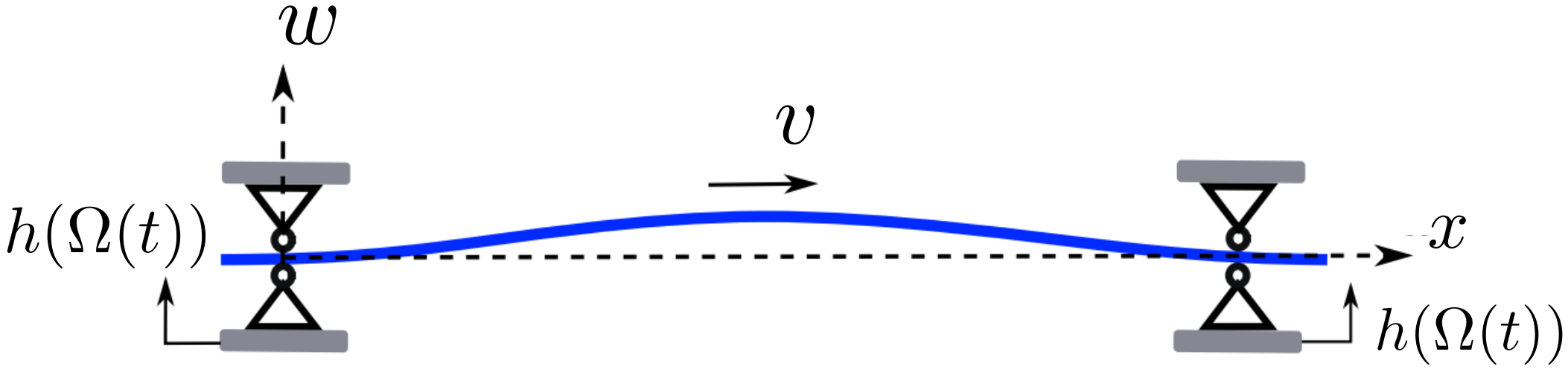}}
    \caption{}
    \label{fig:setup_AB}
\end{subfigure}
\hfill
\begin{subfigure}{0.45\textwidth}
    \includegraphics[width=\textwidth]{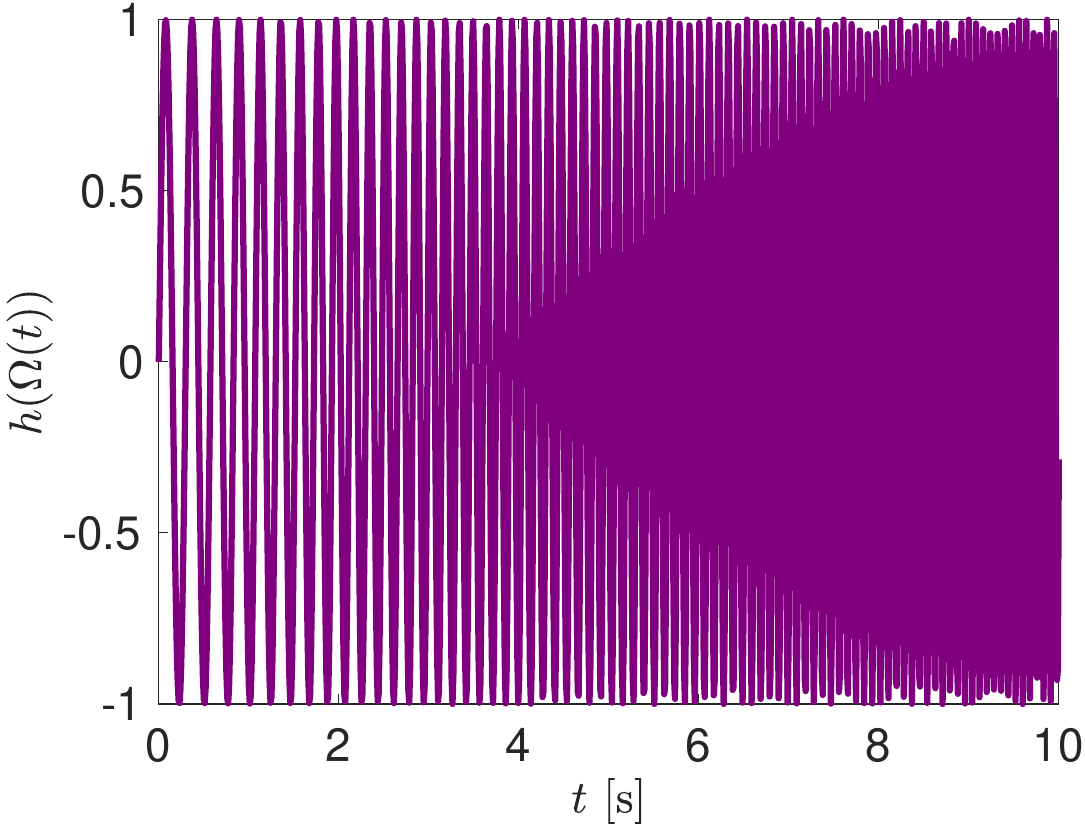}
    \caption{}
    \label{fig:chirp_force}
\end{subfigure}
\caption{(a) The axially moving beam setup adopted from \citet{li22a} with simply supported boundary conditions. (b) Linear chirp forcing signal profile.}
\label{fig:AB_chirp}
\end{figure}

We compute an $O(\mathrm{10})$ GSS Taylor approximation for the beam's Galerkin discretized model using the formulas appearing in Section \ref{subsec:general_damping} for the case of general damping. Since the forcing is zero for $t<0$, we do not need to pad the integral formulas eq. (\ref{eq:Integral_Fredholm}). We also compute the $O(1)$ GSS Taylor approximation, which gives the linear response of the forced system (\ref{eq:eom}).  We then compare these GSS Taylor approximations with the full model. The time integration of the full model is performed using the implicit Newmark method (see \citet{newmark59}).

\begin{figure}[H]
\centering
\begin{subfigure}{0.45\textwidth}
    \includegraphics[width=\textwidth]{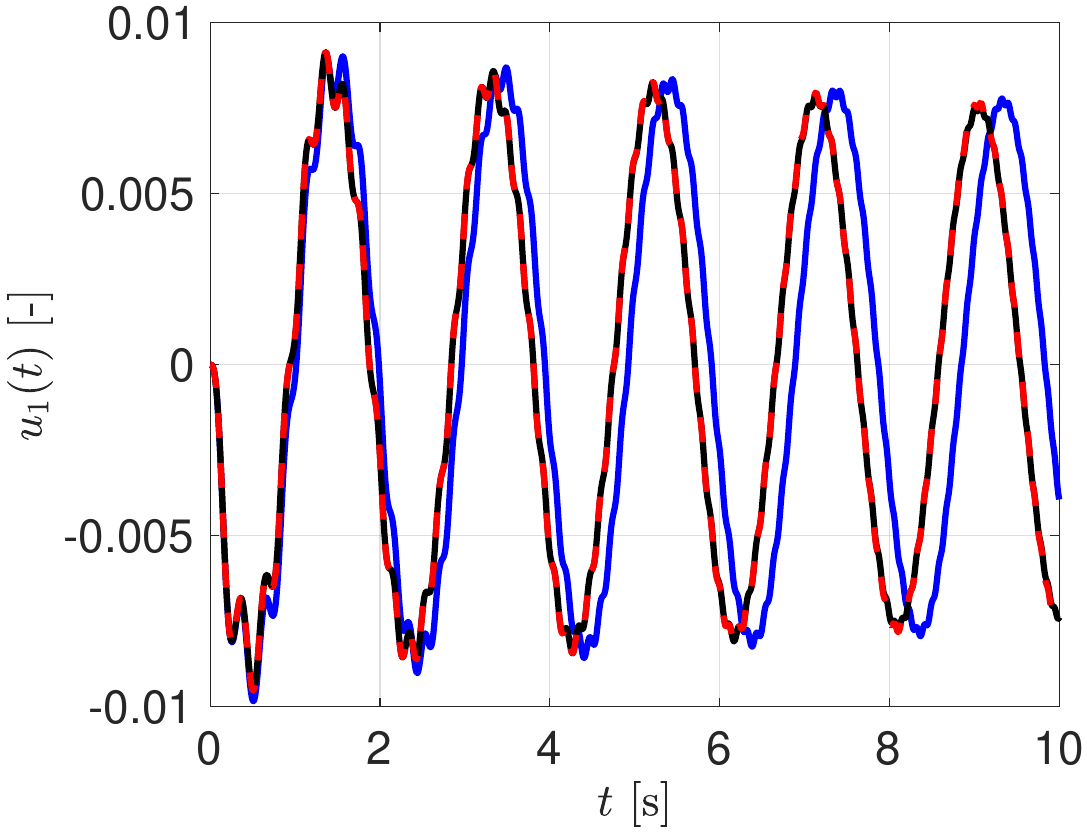}
    \caption{}
    \label{fig:u1}
\end{subfigure}
\hfill
\begin{subfigure}{0.45\textwidth}
    \includegraphics[width=\textwidth]{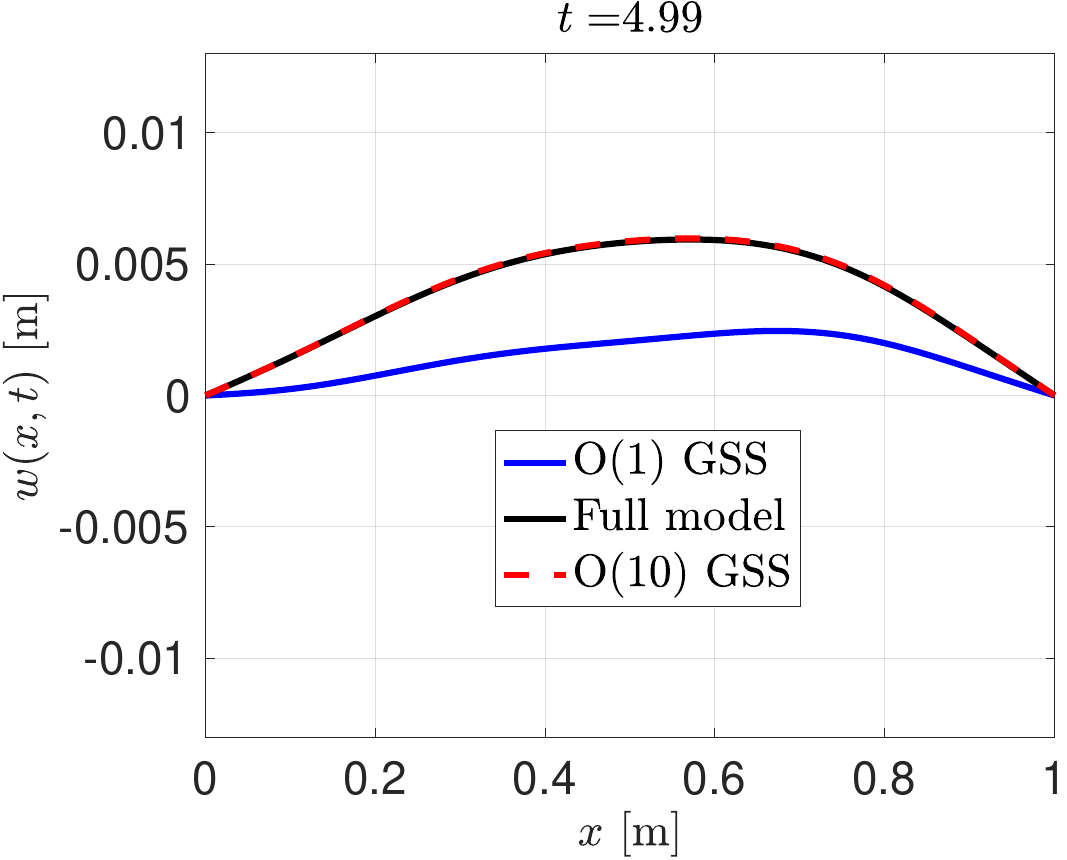}
    \caption{}
    \label{fig:snap1_AB}
\end{subfigure}
\hfill 
\begin{subfigure}{0.45\textwidth}
    \includegraphics[width=\textwidth]{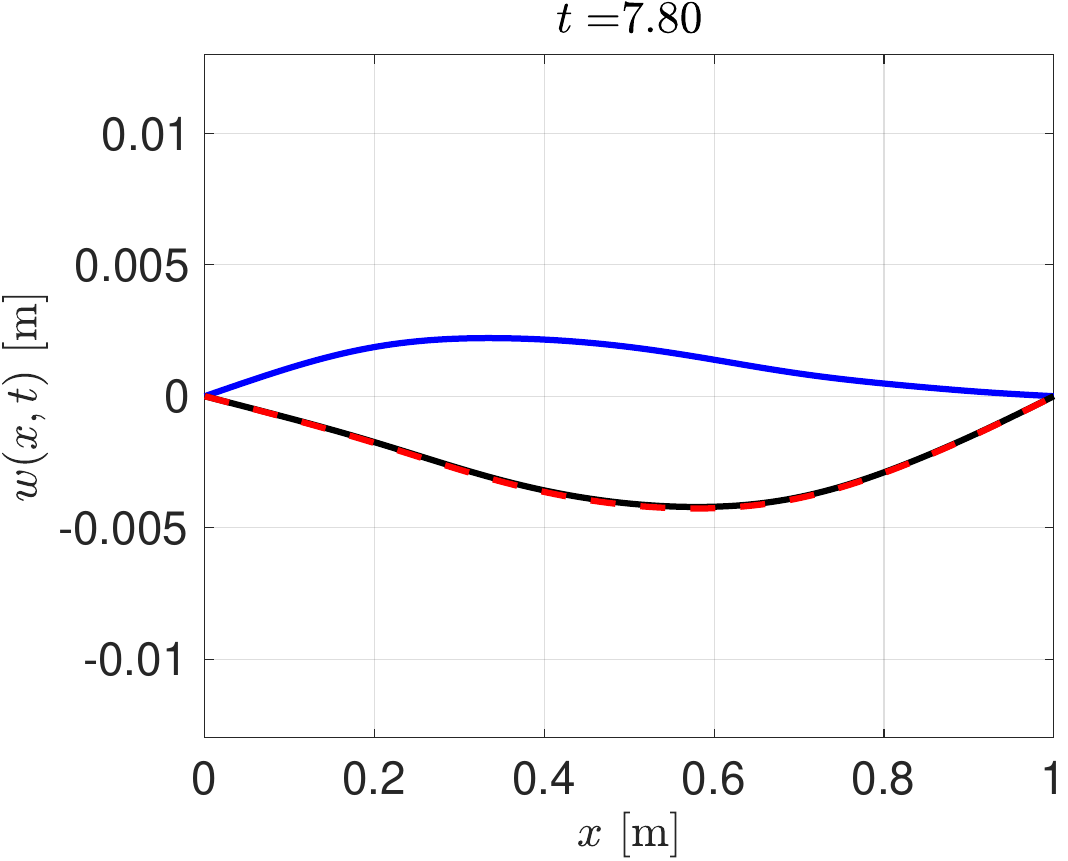}
    \caption{}
    \label{fig:snap2_AB}
\end{subfigure}
\caption{(a) Time evolution of the first Galerkin mode $u_1(t)$ under chirp base excitation for the full model (black), the $O(1)$ linear GSS Taylor approximation (blue) and the $O(10)$ GSS approximation (red). (b) Snapshot of the beam profile at $t \approx 4.99 \text{ [s]}$. (c) Snapshot of the beam profile at $t = 7.80 \text{ [s]}$.  }
\label{fig:AB_chirp}
\end{figure}

An $O(10)$ GSS approximation already suffices to capture the full model response accurately. Figure. \ref{fig:u1} compares this approximation with the $O(1)$ GSS for the $u_1$ Galerkin mode. We also plot snapshots of the full beam profile for $t=4.99 \text{ [s]}$ and $t=7.80 \text{ [s]}$ in Figs. \ref{fig:snap1_AB} and \ref{fig:snap2_AB}. The same figures highlight that higher-order correction terms to the GSS approximation become relevant for longer times when the nonlinear internal forces of the system become more active.

\begin{table}[h]
\centering
\caption{Run time comparisons for axially moving beam}
\label{tab:runtime_AB}
\begin{tabular}{lcc}
\toprule
 & Full model & $O(\mathrm{10})$ GSS \\
\midrule
Run times  & 4.8 s  & \textbf{1.5 s} \\
\bottomrule
\multicolumn{3}{c}{All the simulations are run on a M4 Macbook Pro 24 GB RAM laptop.} \\
\bottomrule
\end{tabular}
\end{table}
In this example, we have computed the GSS using the formulas listed in Section \ref{subsec:general_damping} without the modal truncation step. We list computational times in Table \ref{tab:runtime_AB} for the $O(10)$ GSS Taylor approximation using \textit{GSSTool} and for the time integration of the full model. We note that \textit{GSSTool} is more than three times faster than the full model simulation.

\subsection{Example 2: Randomly excited oscillator chain}
\label{subsec:e2}
We now consider an oscillator chain comprising $n = 20$ masses connected via nonlinear springs and viscous dashpots (see Fig. \ref{fig:oscillator_chain_setup}a). All masses in the chain are set to $0.1 \text{ } [\mathrm{kg}]$. The springs have a linear stiffness of $100 \text{ [N/m]}$ and a cubic nonlinear stiffness of $2,500 [\text{ N/m$^3$}]$. The damping of the viscous dashpots is set to $0.1 \text{ [Ns/m]}$. 
The start and end of the chain are forced randomly with a 100 second long signal sampled from a Gaussian distribution with standard deviation $\sigma = 1.54$ and mean $0$ (see Fig. \ref{fig:oscillator_chain_setup}b). Further, the Gaussian forcing is filtered by removing all frequency content that is larger than $7.5 \text{ [Hz]}$. This frequency value corresponds to the slowest mode of vibration of the unforced chain.  The maximum magnitude of the forcing realizations used in this example is $\Delta = 2.8 \text{ [N]}$.   
\begin{figure}[H]
    \begin{centering}
    \includegraphics[width=1\textwidth]{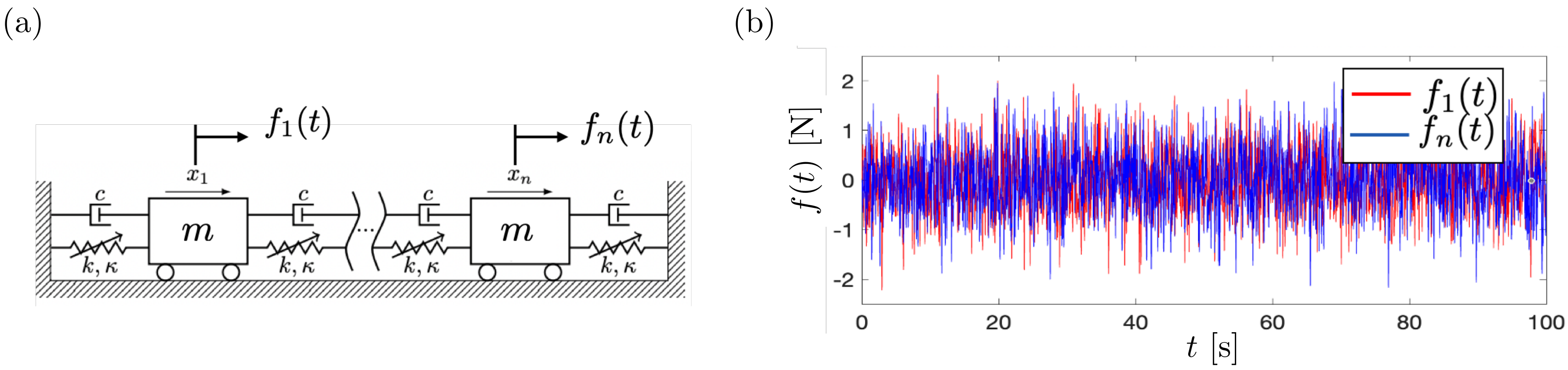}
    \par\end{centering}
    \caption{(a) Oscillator chain setup. (b) Gaussian random forcing signals with zero mean and standard deviation $\sigma = 1.54$}
    \label{fig:oscillator_chain_setup}
    \end{figure}

For exactly the same configurations and forcing, \citet{simpson2021} fit a $9$-dimensional autoencoder (AE) mapping and a Long Short-Term Memory (LSTM) network to model the chain dynamics using forced data. The input data for their fitting procedure is the full positional data of the first $50 \%$ of the forced response of the system obtained from applying a single random realization of the Gaussian forcing to the chain. They proceed by testing their model on the remaining half of the forced response. We detail their methodology and our implementation of it in Appendix \ref{app:lstm}. 

In Fig.\ref{fig:oc_test_train}a, we plot the $O(10)$ GSS approximation and the AE+LSTM prediction in the 80-90 second time window for the same realization of the forcing. The GSS is computed on the full time window of the forcing realization using integral formulas for the structural damping case presented in Section \ref{subsec:structural_damping}. The $O(10)$ GSS accurately matches the full model. The AE+LSTM model correctly predicts the trends in the response but does not perfectly match the full model. This is expected as the GSS is computed using the full model equations, whereas the AE+LSTM is a purely data-driven model. 

\begin{table}[h]
\centering
\caption{Run time comparisons for the oscillator chain model}
\label{tab:runtime_oschain}
\begin{tabular}{lccc}
\toprule
 & Full model & $O(\mathrm{10})$ GSS & Autoencoder+LSTM (9D) \\
\midrule
\multirow{2}{*}{Run times} & \multirow{2}{*}{7.1 s} & \multirow{2}{*}{\textbf{2.5 s}} & Training: 20 min \\
                           & & & Testing: 3.4 min \\
\bottomrule
\multicolumn{4}{c}{All the simulations are run on a M4 Macbook Pro 24 GB RAM laptop.} \\
\bottomrule
\end{tabular}
\end{table}

\begin{figure}[H]
\centering
\begin{subfigure}{0.48\textwidth}
    \includegraphics[width=\textwidth]{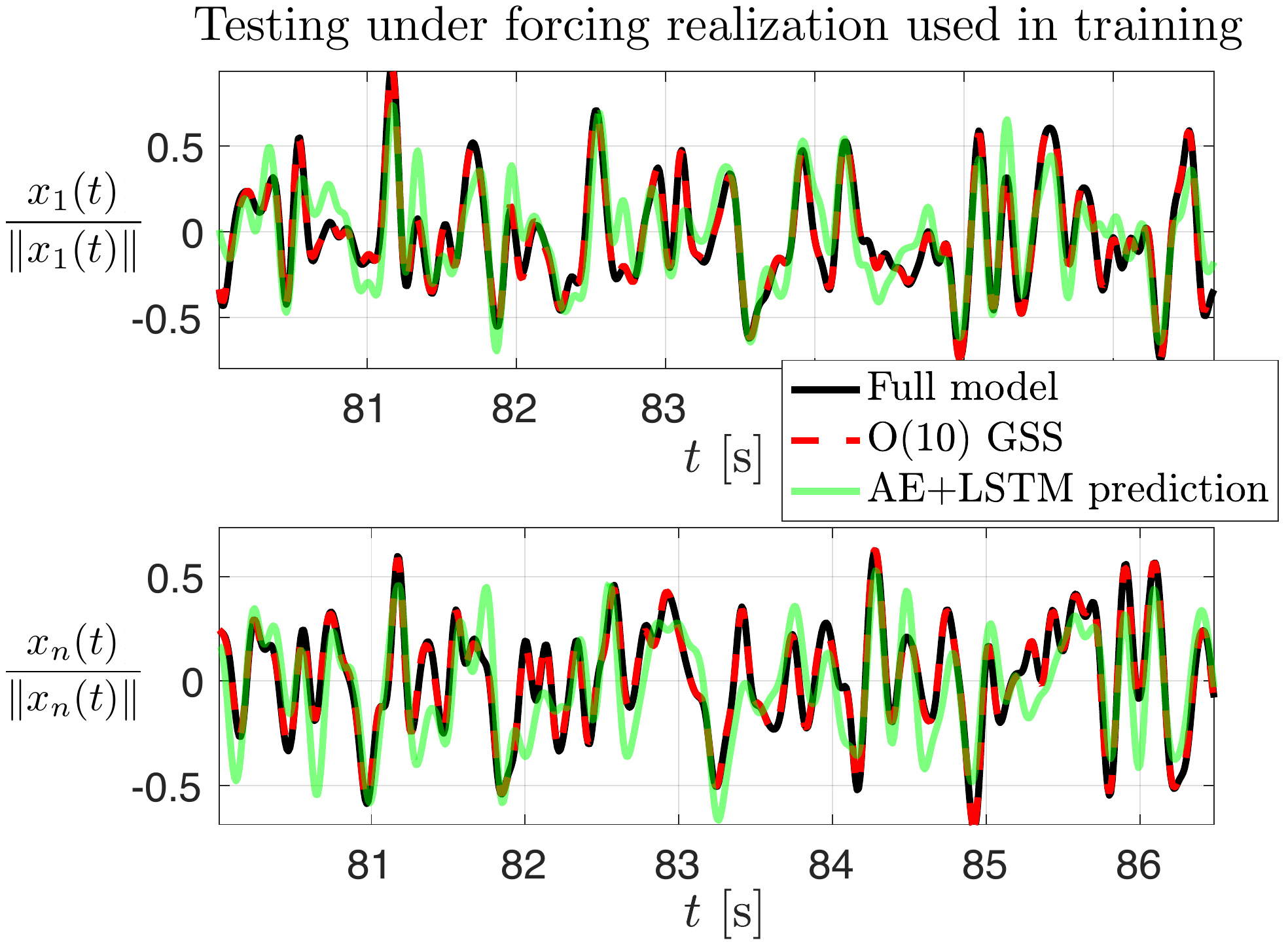}
    \caption{}
    \label{fig:training_result}
\end{subfigure}
\hfill
\begin{subfigure}{0.48\textwidth}
    \includegraphics[width=\textwidth]{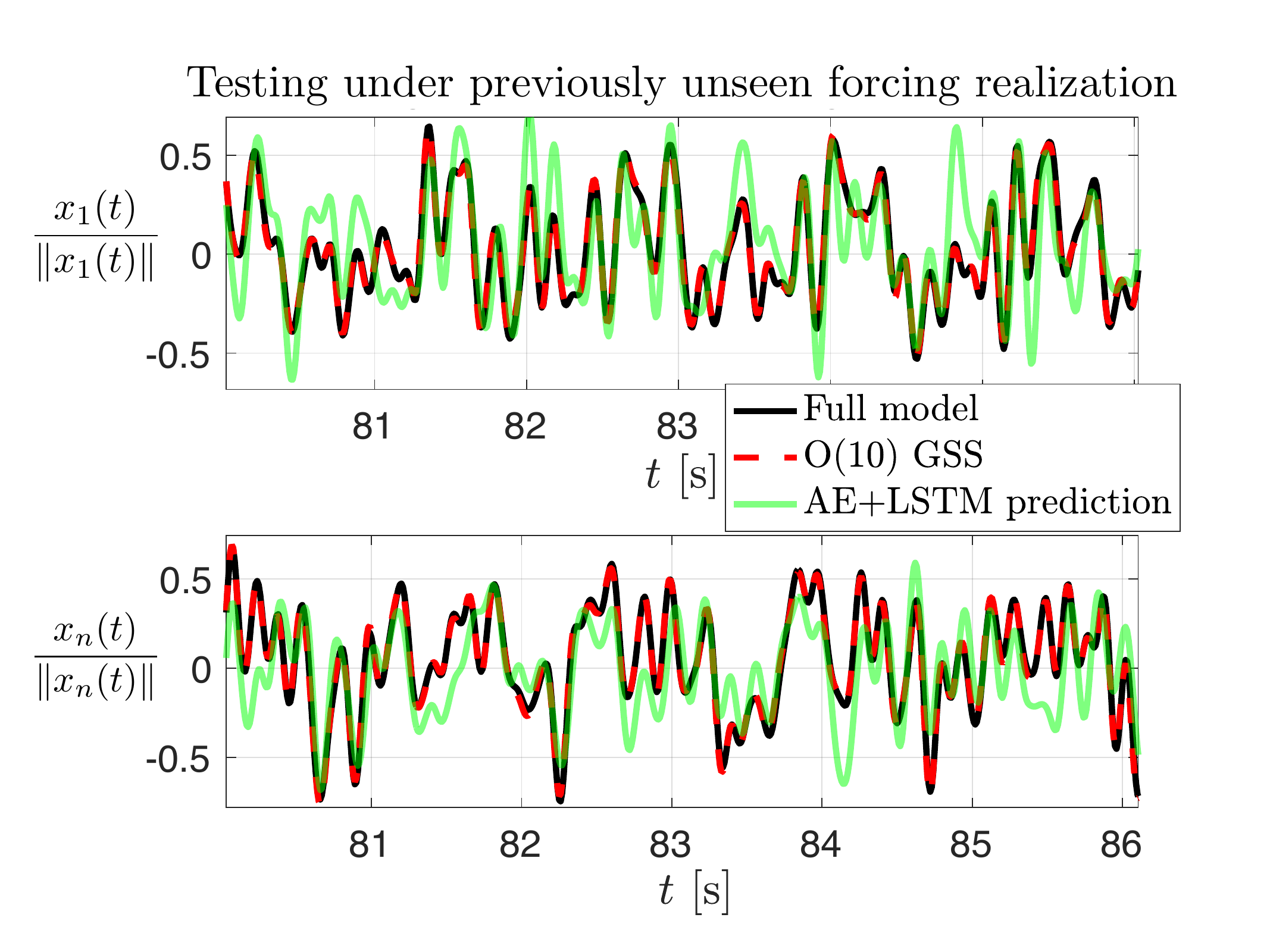}
    \caption{}
    \label{fig:testing_result}
\end{subfigure}
\caption{(a) Normalized displacement trajectories of the first (top) and last (bottom) masses in the chain for the full model (black), the $O(10)$ GSS (red) and the autoeconder+LTSM model (green). (b) Same as in (a) but for a previously unseen random forcing realization generated from the same distribution used in the training of the autoencoder+LSTM model.}
\label{fig:oc_test_train}
\end{figure}

Next, we test the predictive power of the data-driven AE+LSTM model and our GSS approximations for a previously unseen forcing realization from the same Gaussian distribution used in training. From Fig. \ref{fig:oc_test_train}b, we observe that the GSS still matches the response of the full model with high accuracy. In contrast, the AE+LSTM model predicts overall trends in phase but fails at most time instances in predicting the amplitudes of the true response. 

This loss in accuracy by the AE+LSTM model is not compensated by computational run times. Table \ref{tab:runtime_oschain} highlights that training an 9D autoencoder + LSTM model requires 20 minutes and the learned model takes 3.4 minutes to run on an unseen forcing data. This is to be contrasted with the $O(10)$ GSS approximation, which requires only 2.5 seconds to run to provide an accurate prediction for the forced response.

\begin{figure}[H]
\centering
\begin{subfigure}{0.48\textwidth}
    \includegraphics[width=\textwidth]{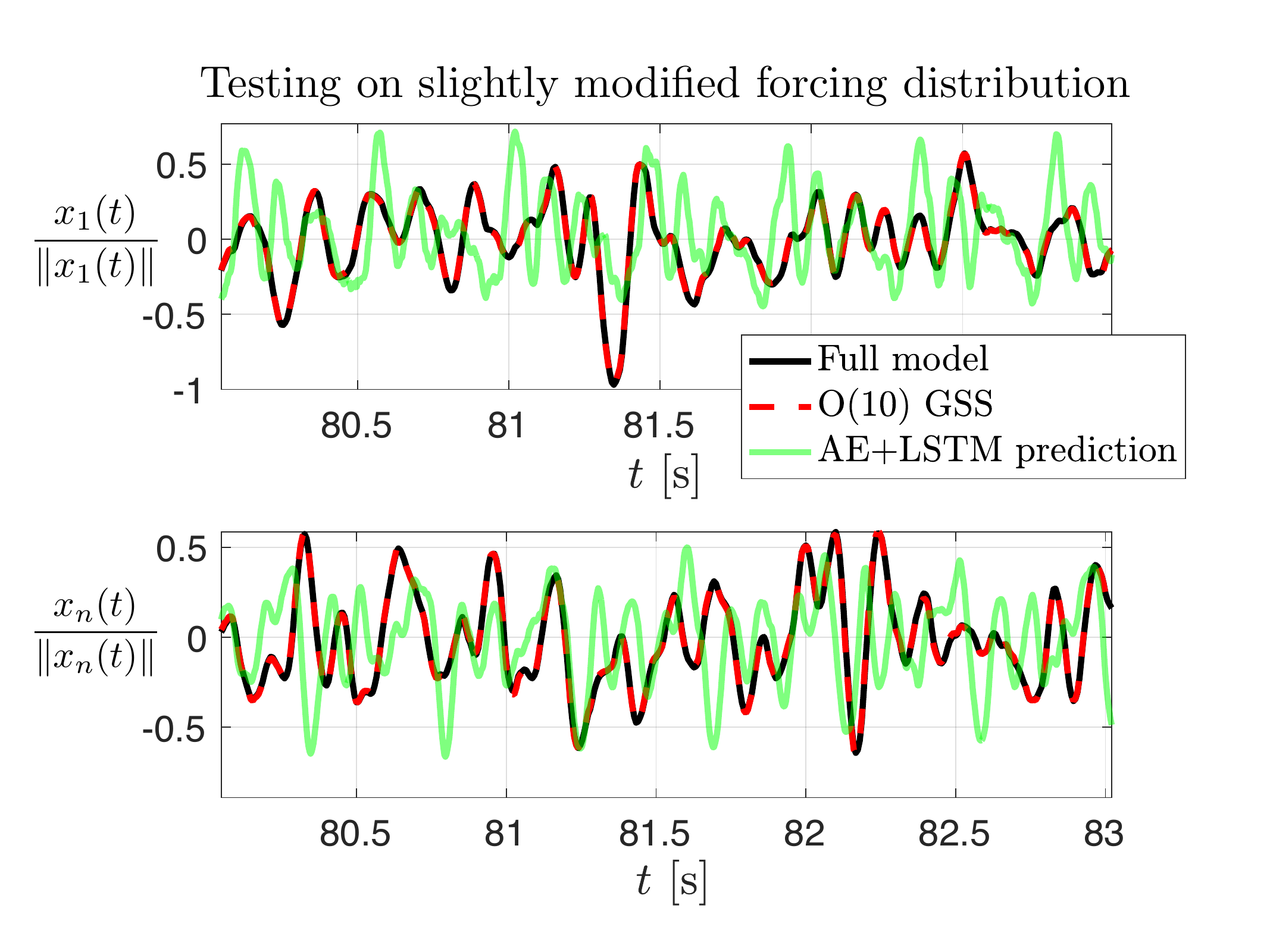}
    \caption{}
    \label{fig:unseen_new}
\end{subfigure}
\hfill
\begin{subfigure}{0.48\textwidth}
    \raisebox{1.5cm}{\includegraphics[width=\textwidth]{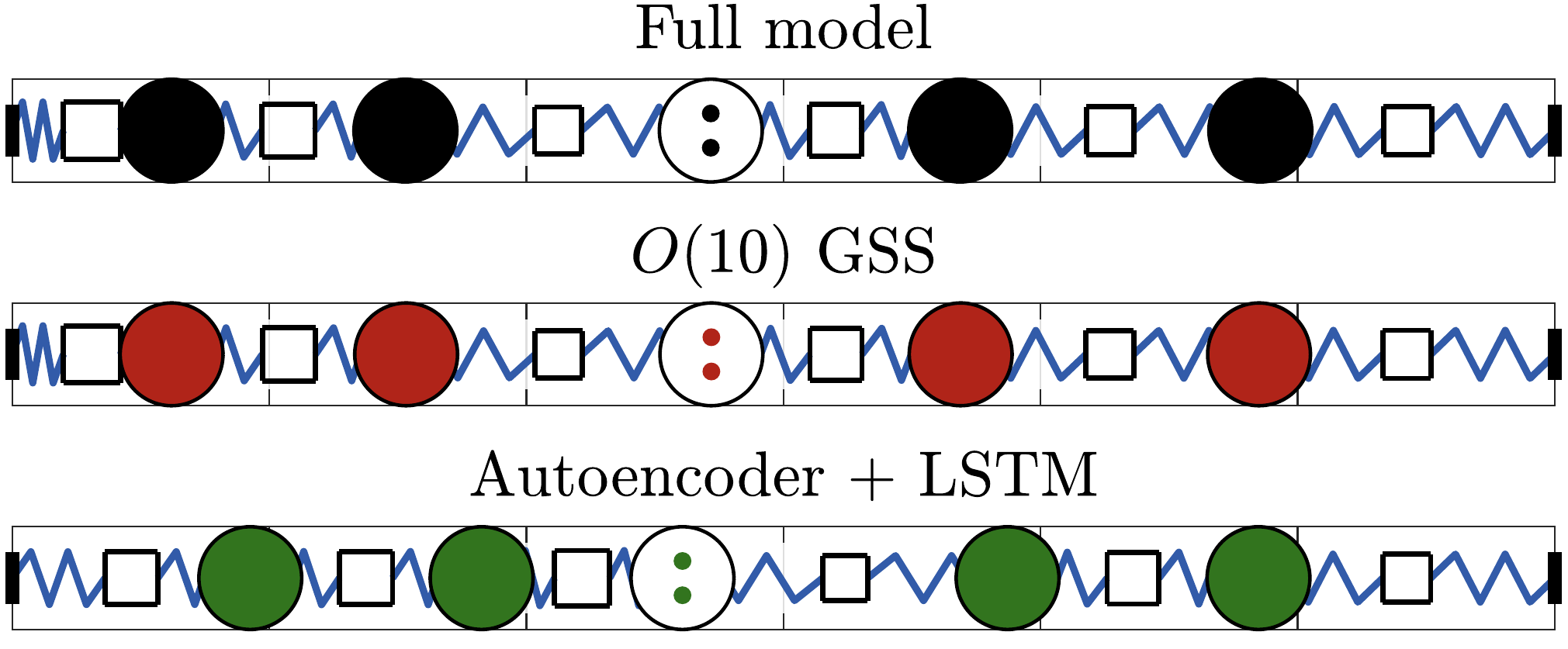}}
    \caption{}
    \label{fig:setup_snap_OC}
\end{subfigure}
\caption{(a) Normalized displacement trajectories of the first (top) and last (bottom) masses in the chain under a modified Gaussian distribution with zero mean and standard deviation $\sigma = 1.8$, for the full model (black), the $O(10)$ GSS (red) and the autoeconder + LTSM model (green). (b) Snapshot of the oscillator chain at $t = 80 \text{ }\mathrm{[s]}$. The middle mass block represents the center of mass position of the rest of the $16$ masses in the chain.}
\label{fig:oc_new_unseen}
\end{figure} 

% Note, in all our comparisons with the AE+LSTM model, we only present the plots in the normalized coordinates of the full model as the AE+LSTM model is trained on normalized trajectory data. Hence does not retain or provide a scaling variable to rescale the predictions to the true values. The unseen forcing realization, will demand a new normalization transformation, this is apriori unknown to us, unless you have the full response. This is counter productive to what the AE+LSTM model set out to do. This also explains the reason for now observing dip in amplitude accuracy prediction. The trained AE+LSTM model is biased to the magnitudes of the forced response of a single realization of forcing. 

We close this section by comparing the GSS and the AE+LSTM model predictions (see Fig. \ref{fig:oc_new_unseen}) for forcing generated from a modified zero-mean Gaussian distribution with standard deviation $\sigma =1.8$. This slightly modified distribution is filtered to contain frequencies up to $40 \text{ [Hz]}$. We observe that the trained AE+LSTM model is not robust with respect to this slight statistical change of the forcing distribution. In contrast, the $O(10)$ GSS approximation still accurately predicts the full model response accurately with minimal computational effort. 

We stress that the AE+LSTM model is fully data-driven, whereas the GSS computation presented here utilizes knowledge of the terms in the equations of motion.

\subsection{Example 3: Partially data-driven computation of the GSSs for the nonlinear oscillator chain}
\label{subsec:e3}

We now revisit the nonlinear oscillator chain in Example 2 and assume that only the linear part of the system is known explicitly but trajectory data is available for the full nonlinear system in the absence of forcing. This is a realistic setting for structural vibration problems whose linear behavior can be accurately modeled, but their nonlinear characteristics have to be inferred from experimental data. 

We generate numerically two unforced decaying trajectories launched using random initial conditions in the phase space of the system shown in Fig. \ref{fig:ssm_osci_chain}. From one of these trajectories, we extract a two-dimensional slow SSM using \textit{SSMLearn} in the phase space of the oscillator chain. We find an $O(11)$ polynomial SSM approximation to accurately predict the second decaying trajectory that was not used in the SSM construction (see Figs. \ref{fig:ssm_osci_chain},b). The SSM parametrization and reduced dynamics for the unforced system take the form given in eq. (\ref{eq:reduced_SSM}).

At leading order, the contribution to the SSM dynamics can be obtained by projecting the forcing to the tangent space of the unforced SSM at the origin. Under forcing, the parametrized SSM deforms in time, but remains anchored to the GSS. These leading order terms in eq. (\ref{eq:reduced_SSM}) involve the spectral projections of the external forcing vector. We already have these spectral projections at our disposal from the linear analysis of the unforced system. We calculate GSS Taylor approximations with the forced SSM-reduced model and lift the reduced GSS solution found within the SSM to the full phase space using the learned SSM parametrization map.

\begin{figure}[H]
\centering
% Top row
\begin{subfigure}{0.45\textwidth}
    \centering
    \includegraphics[width=\textwidth]{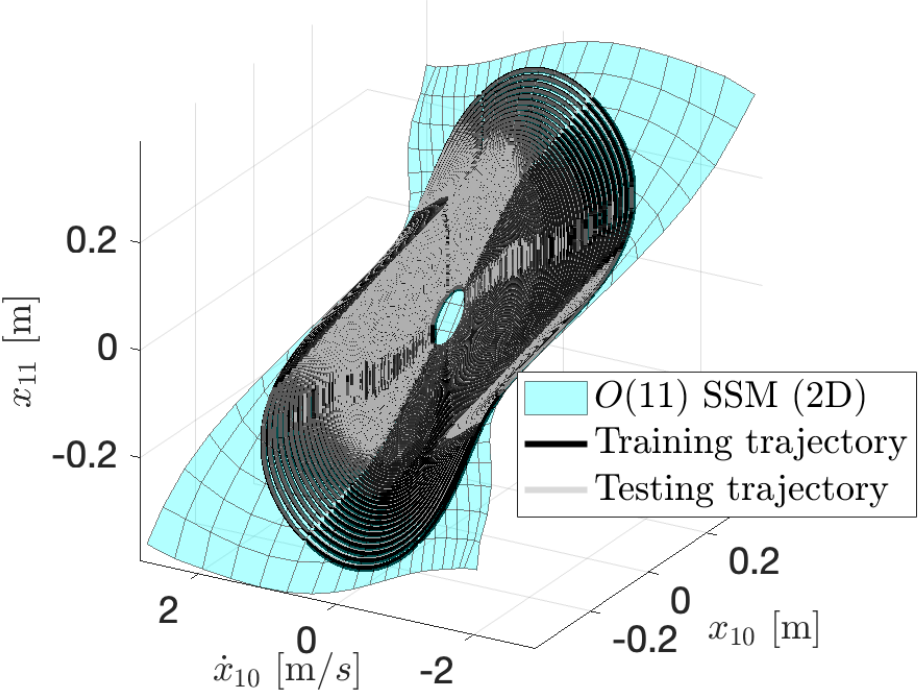}
    \caption{}
    \label{fig:ssm_osci_chain}
\end{subfigure}
\hfill
\begin{subfigure}{0.5\textwidth}
    \centering
    \includegraphics[width=\textwidth]{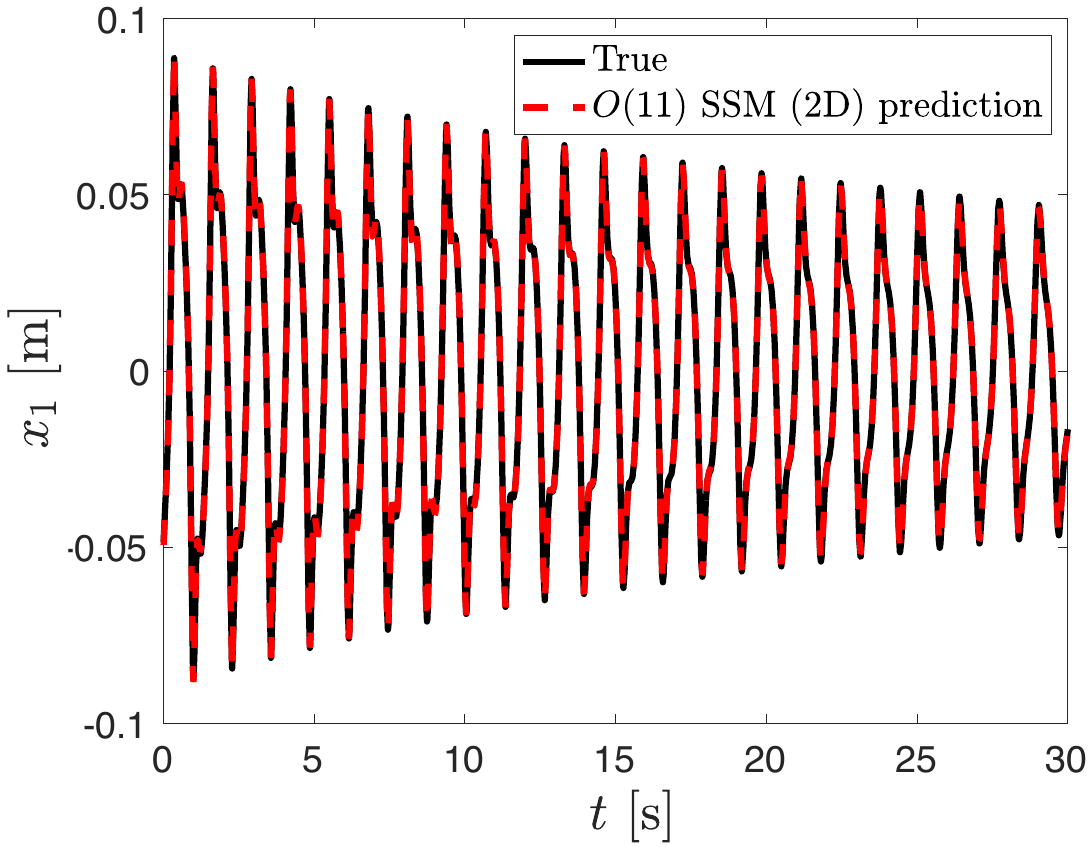}
    \caption{}
    \label{fig:period3_qp_shell}
\end{subfigure}
\hfill
\begin{subfigure}{0.5\textwidth}
    \centering
    \includegraphics[width=\textwidth]{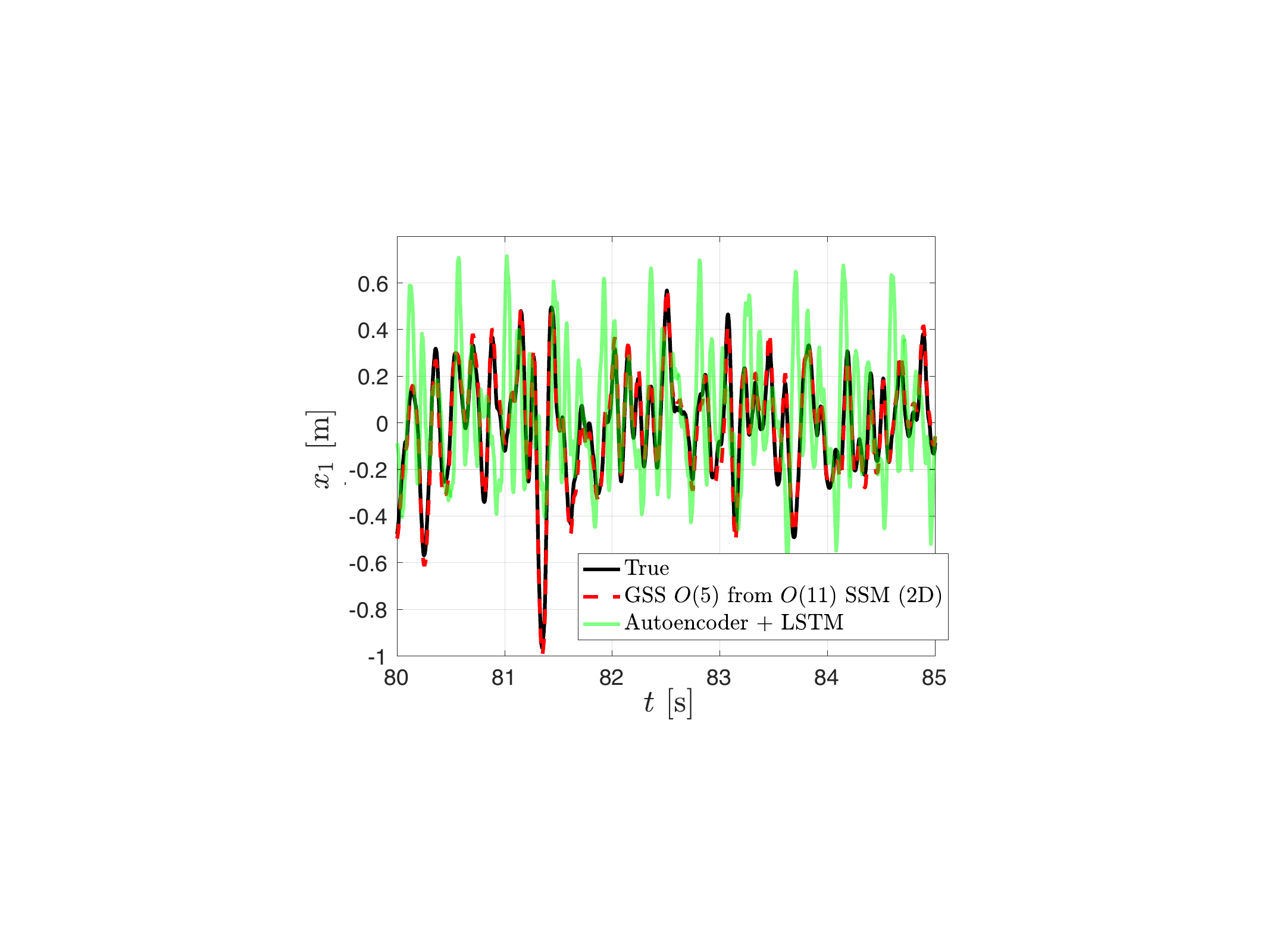}
    \caption{}
    \label{fig:x1_anchor_response}
\end{subfigure}
\caption{(a) Unforced 2D SSM of oscillator chain learned from data. (b) Reduced prediction using the SSM-reduced model (red) compared with the full model (black). (c) Normalized $x_1$ response for an $O(5)$ GSS of the forced SSM-reduced model (red), Autoencoder+LSTM model (green) and the full model (black). The stochastic forcing used in this example is the same as the one used to generate Fig. 5.}
\label{fig:SSM_OC_main}
\end{figure}

In Fig. \ref{fig:x1_anchor_response}, we plot the GSS predicted from our partially data-driven forced SSM-reduced model and the response of the fully data-driven autoencoder+LSTM model for the forcing sampled from unseen Gaussian distribution discussed at the end of Section 4.2. We observe that the GSS prediction closely follows the full response of the system. 

In Fig. \ref{fig:SSM_OC_snaps}, we plot snapshots of the predicted GSS, the aperiodic SSM predictions, and the system's true response in the phase space. Figure \ref{fig:ssm_snap3} is the snapshot captured in the coordinates that represent the middle masses of the oscillator chain. In this frame, the SSM does not move much, but the forced prediction moves on the SSM. On the other hand, for the snapshots in Figs. \ref{fig:ssm_snap1} and \ref{fig:ssm_snap2}, the SSM parametrization is anchored to the GSS and translates along it.  

The fully data-driven autoencoder+LSTM model only predicts positions, so velocities were recovered by finite differencing. The predictions of the autoencoder+LSTM model are unaware of the SSM geometry and fail to capture the dynamics despite using an 9-dimensional latent space. Moreover, the model is expensive to simulate due to the dense network architecture of the LSTM. In contrast, our SSM-reduced model is only two-dimensional and comprises simple smooth polynomials. As shown in Table \ref{tab:smm_oc_runtimes}, it is 40 times faster to test and 150 times faster to train than the autoencoder+LSTM model.

% Bottom row
\begin{figure}[H]
\centering
\begin{subfigure}{0.45\textwidth}
    \centering
    \includegraphics[width=\textwidth]{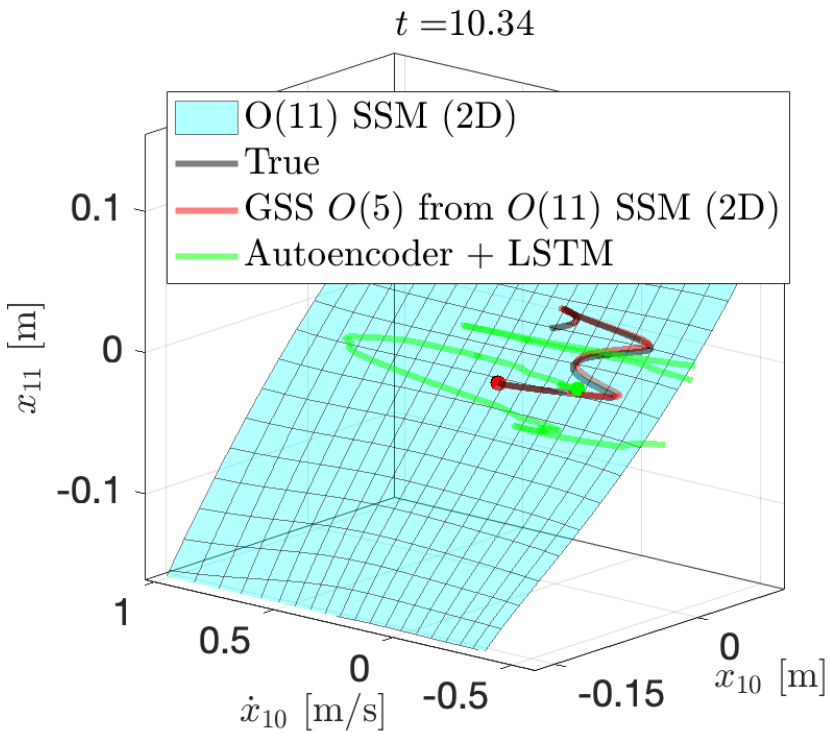}
    \caption{}
    \label{fig:ssm_snap3}
\end{subfigure}
\hfill
\begin{subfigure}{0.5\textwidth}
    \centering
    \raisebox{0.8cm}{\includegraphics[width=\textwidth]{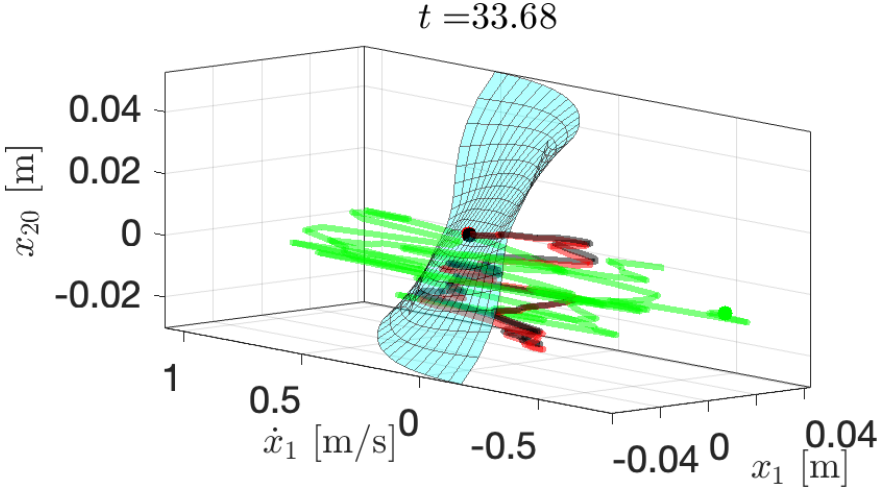}}
    \caption{}
    \label{fig:ssm_snap1}
\end{subfigure}
\hfill
\begin{subfigure}{0.5\textwidth}
    \centering
    \raisebox{0.8cm}{\includegraphics[width=\textwidth]{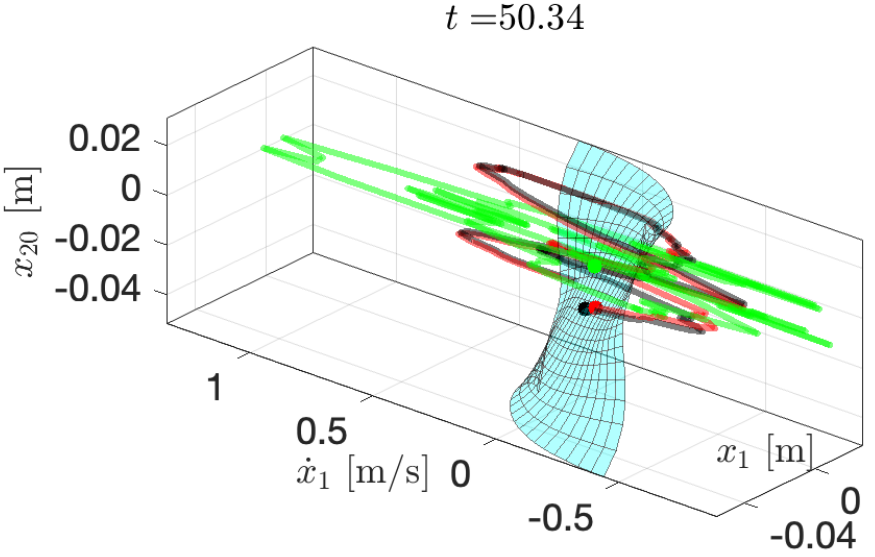}}
    \caption{}
    \label{fig:ssm_snap2}
\end{subfigure}

\caption{Aperiodically forced SSM (cyan) snapshots in the phase space with trajectory plots of the GSS prediction (red), the autoencoder+LSTM prediction (green), and the true response (black). (a) Snapshot at $t=10.34 \text{ }[\mathrm{s}]$  in the $x_{10}$, $\dot{x}_{10}$ and $x_{11}$ coordinates. Snapshots at  (b) $t=33.68 \text{ }[\mathrm{s}]$ and  (c)  $t=50.34 \text{ }[\mathrm{s}]$ in the $x_{1}$, $\dot{x}_{1}$ and $x_{20}$ coordinates.}
\label{fig:SSM_OC_snaps}
\end{figure}

\begin{table}[h]
\centering
\caption{Run time comparisons for the SSM-reduced model for the oscillator chain}
\label{tab:smm_oc_runtimes}
\begin{tabular}{lccc}
\toprule
 & Full model & SSM (2D) + $O(\mathrm{5})$ GSS & Autoencoder + LSTM (9D) \\
\midrule
\multirow{2}{*}{Run times} & \multirow{2}{*}{7.1 s} & Training: 8 s & Training: 20 min \\
                           & & Testing: \textbf{5 s}   & Testing: 3.4 min \\
\bottomrule
\multicolumn{4}{c}{All the simulations are run on a M4 Macbook Pro 24 GB RAM laptop.} \\
\bottomrule
\end{tabular}
\end{table}

\subsection{Example 4: Earthquake response of a cantilevered von K\'arm\'an beam}

Our next example is a finite-element model of a beam satisfying the von K\'arm\'an strain hypothesis. The model involves an upright cantilevered beam mounted to the ground, as shown in Fig. \ref{fig:setup_beam}. The beam is made of steel of density $7850 \text{ [kg/m$^3$]}$, width $0.01 \text{ [m]}$, breadth $ 0.05 \text{ [m]}$, height $2 \text{ [m]}$, Young's modulus  $190\times 10^9 \text{ [Pa]}$, Poisson's ratio $0.3$, and viscous damping parameter $10^7 \text{ [Ns/m]}$. A detailed description of the finite-element model can be found in \citet{jain23}. Each node in the finite-element model comprises three degrees of freedom, the axial displacements denoted by $z$, the transverse displacements denoted by $x$, and the cross-sectional angle denoted by $\alpha$. 

The beam is subjected to a realistic earthquake excitation by ground motion. For which we use historical ground acceleration data from the PEERs database hosted by the University of Berkeley \citet{PEER_NGA_West2}. Specifically, we consider a recorded acceleration data for a magnitude 6 earthquake event in Parkfield, California, 1966. This is a moderately large earthquake that is expected to cause significant vibrations for steel structural elements. The ground is shaken in the vertical and horizontal directions, thus impacting all the degrees of freedom of the beam (see Fig. \ref{fig:ground_accel}).

\begin{figure}[H]
\centering
\begin{subfigure}{0.45\textwidth}
    \includegraphics[width=\textwidth]{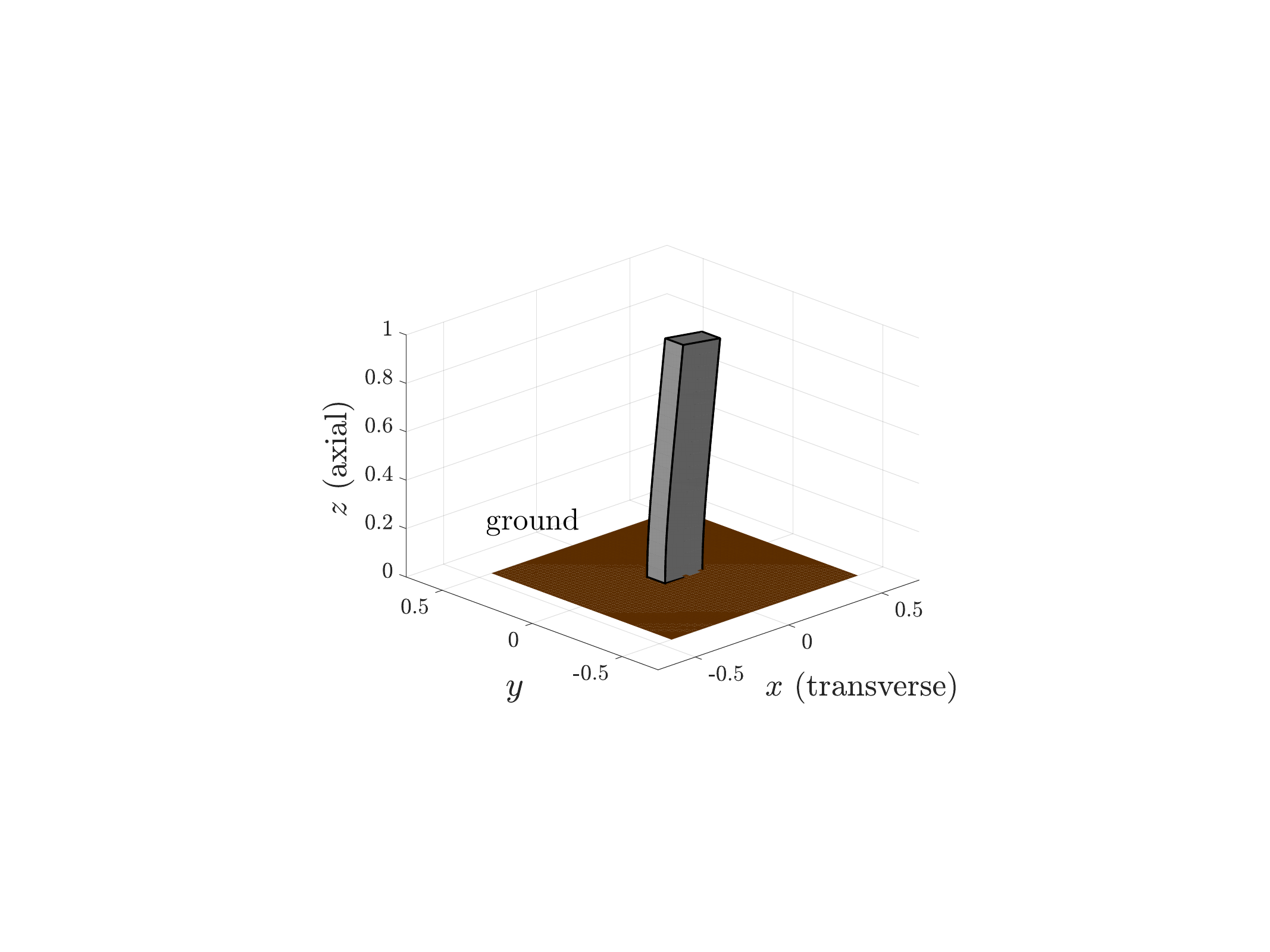}
    \caption{}
    \label{fig:setup_beam}
\end{subfigure}
\hfill
\begin{subfigure}{0.45\textwidth}
    \includegraphics[width=\textwidth]{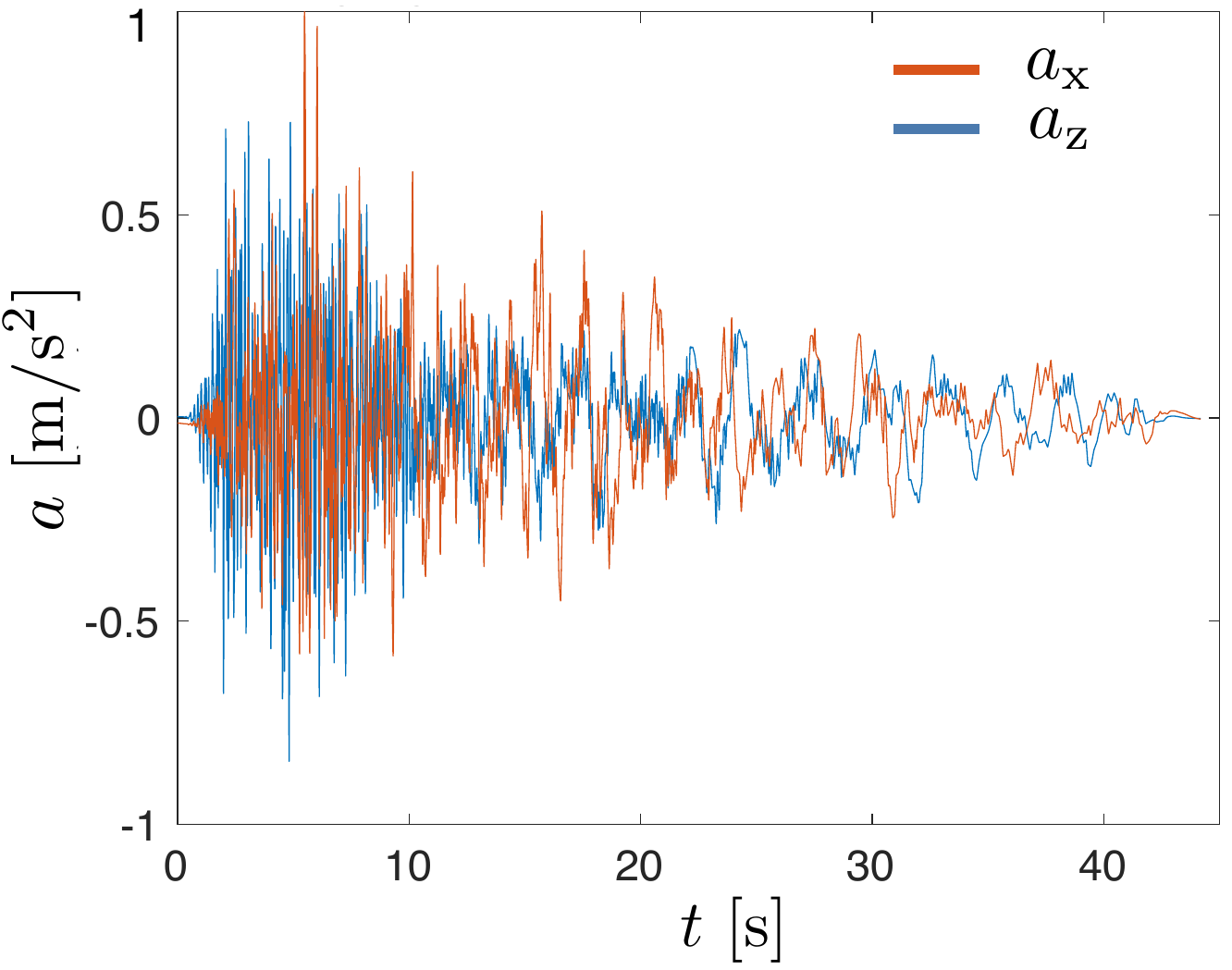}
    \caption{}
    \label{fig:ground_accel}
\end{subfigure}
\caption{(a) Cantilevered von K\'arm\'an beam mounted to the ground. (b) Historic horizontal $a_x$ and vertical $a_z$ ground acceleration data from the PEERs database for a 1966 earthquake in Parkfield, CA.}
\label{fig:beam_earthquake}
\end{figure}

 We compute an $O(5)$  GSS approximation for the beam model for an increasing number of degrees of freedom. We record the computational run times for the evaluation of the GSS using the integral formulas (\ref{eq:fund_integral}) and for the time integration of the full model. For comparison, we also include run times for computing the GSS approximation using numerical time integration. 

Figure \ref{fig:run_times_beam} shows the computational run times for the GSS method using the integral approach, for the GSS method using numerical time integration, and for the full model's response using  numerical time integration. The GSS method consistently achieves run times shorter than the direct numerical solution of the full nonlinear model. As the degrees of freedom model is increased, the run times for all methods gradually increase, but the GSS computations remain approximately three times faster than the full model.

% Bottom row
\begin{figure}[H]
\centering
\begin{subfigure}{0.47\textwidth}
    \centering
    \includegraphics[width=\textwidth]{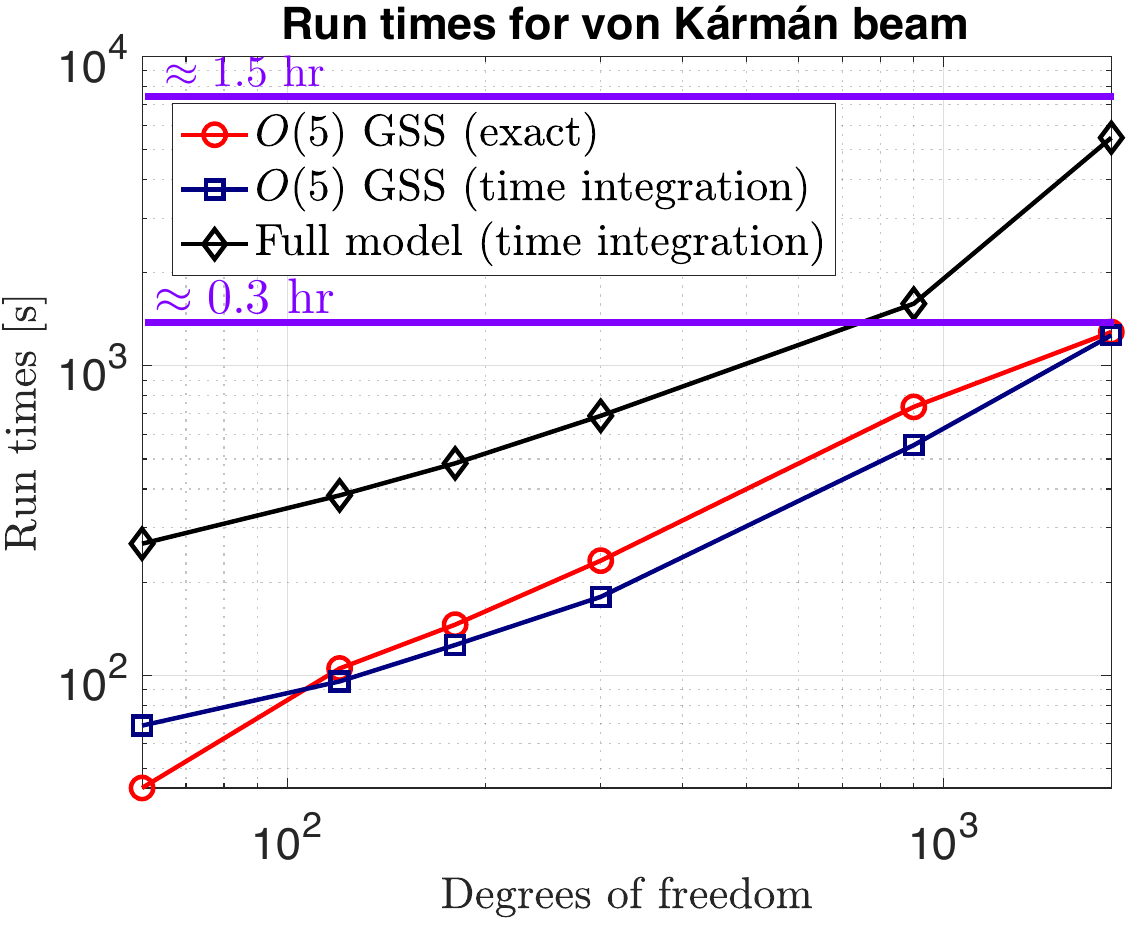}
    \caption{}
    \label{fig:run_times_beam}
\end{subfigure}
\hfill
\begin{subfigure}{0.5\textwidth}
    \centering
    \includegraphics[width=\textwidth]{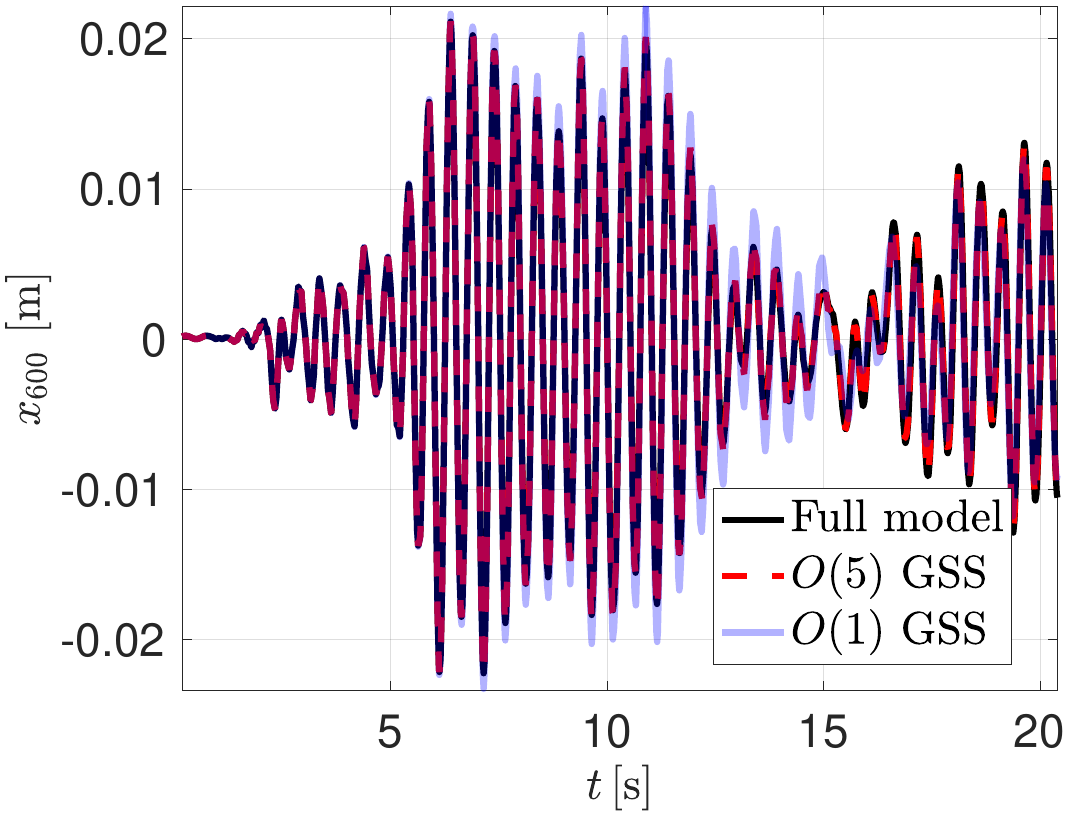}
    \caption{}
    \label{fig:transverse_disp_EQ}
\end{subfigure}
\begin{subfigure}{0.5\textwidth}
    \centering
    \includegraphics[width=\textwidth]{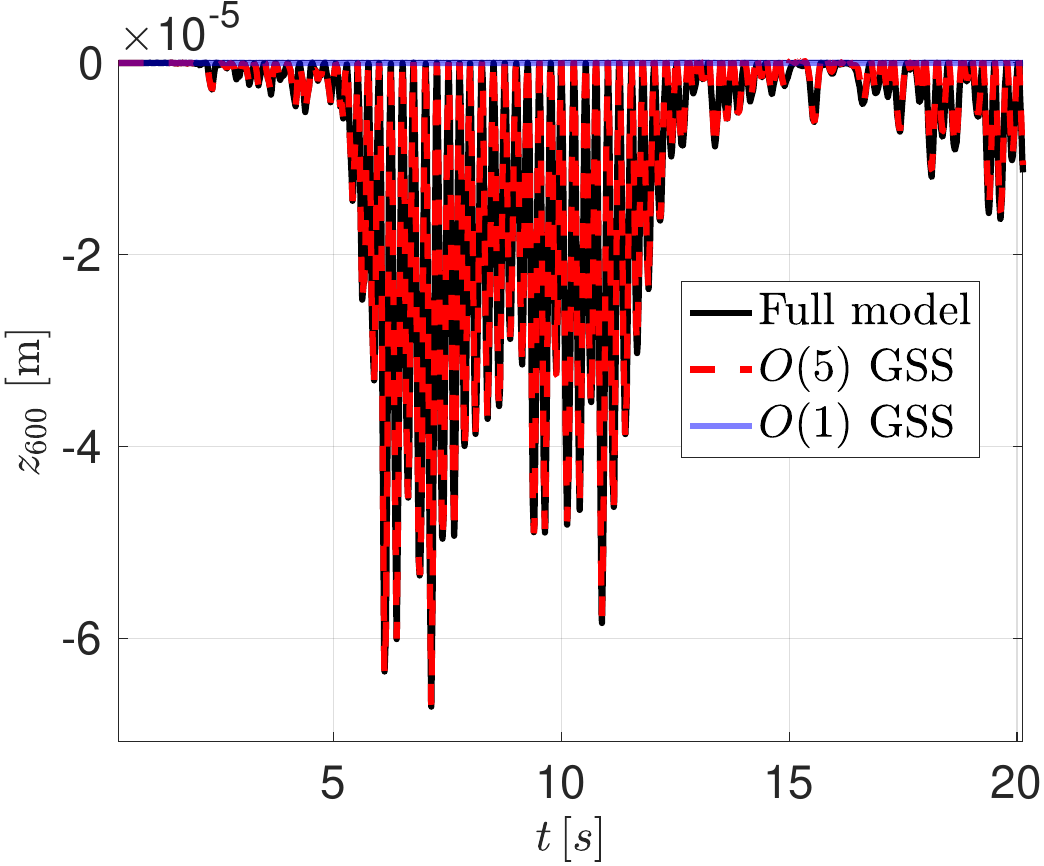}
    \caption{}
    \label{fig:axial_disp_EQ}
\end{subfigure}

\caption{(a) Run time plots with the vertical axis representing CPU runtimes on a M4 Macbook Pro 24 GB RAM laptop and the horizontal axis representing the degrees of freedom of the beam for the $O(5)$ GSS expansions calculated analytically (green) or using the Newmark Beta-method (blue) and the full model (black). (b) Transverse displacement of the free end and (c) axial displacement of the free end under earthquake response for a 1,800 DOF von K\'arm\'an beam for the full model (black), the $O(5)$ GSS (red) and the linear GSS (blue).}
\label{fig:Beam_EQ}
\end{figure}

In Figs. \ref{fig:transverse_disp_EQ} and \ref{fig:axial_disp_EQ}, for a 1,800 DOF finite element model of the beam, the $O(5)$ GSS approximation is compared with the full model's response. For completeness, we also include the linear GSS approximation results. Our approach requires 0.3 hours to compute compared to $1.5$ hours for the full model. The $O(5)$ GSS approximation is sufficient to match the response of the full model. The need for nonlinear corrections to the GSS becomes evident in the axial coordinates, due to nonlinear coupling with the transverse coordinates.

\subsection{Example 5: Stochastic response on a slow SSM-reduced model for the von K\'arm\'an beam}
\label{subsec:e5}
We adopt a similar strategy outlined in Section \ref{subsec:e3} by computing an SSM-reduced model using \textit{SSMLearn} for the cantilevered von K\'arm\'an beam. Specifically, we learn an $O(7)$ SSM-reduced model using a single unforced decaying trajectory of the system. The beam undergoes base excitation with a bounded Gaussian noise signal similar to the one used in \citet{Xu25}. 

Using the forced response of the beam under one realization of the forcing, we learn an $6\mathrm{D}$ autoencoder+LSTM model. We followed a similar learning algorithm as described in \citet{simpson2021}. Specific details on the training datasets and validation errors for this example can be found in the Appendix \ref{app:lstm}.   

We focus on testing our model and the autoencoder+LSTM model on an unseen realization of the forcing from the same noise model. We compute an $O(20)$ GSS approximation of the beam on the forced SSM-reduced model, which we later re-sum as a vector $\text{Pad\'e}[10,10]$ GSS approximation (see Section \ref{subsec:vector_pade}). For more detail on the accuracy comparisons of the Pad\'e and Taylor approximations of the GSS, see Appendix \ref{app:pade}.

\begin{figure}[H]
    \begin{centering}
    \includegraphics[width=1\textwidth]{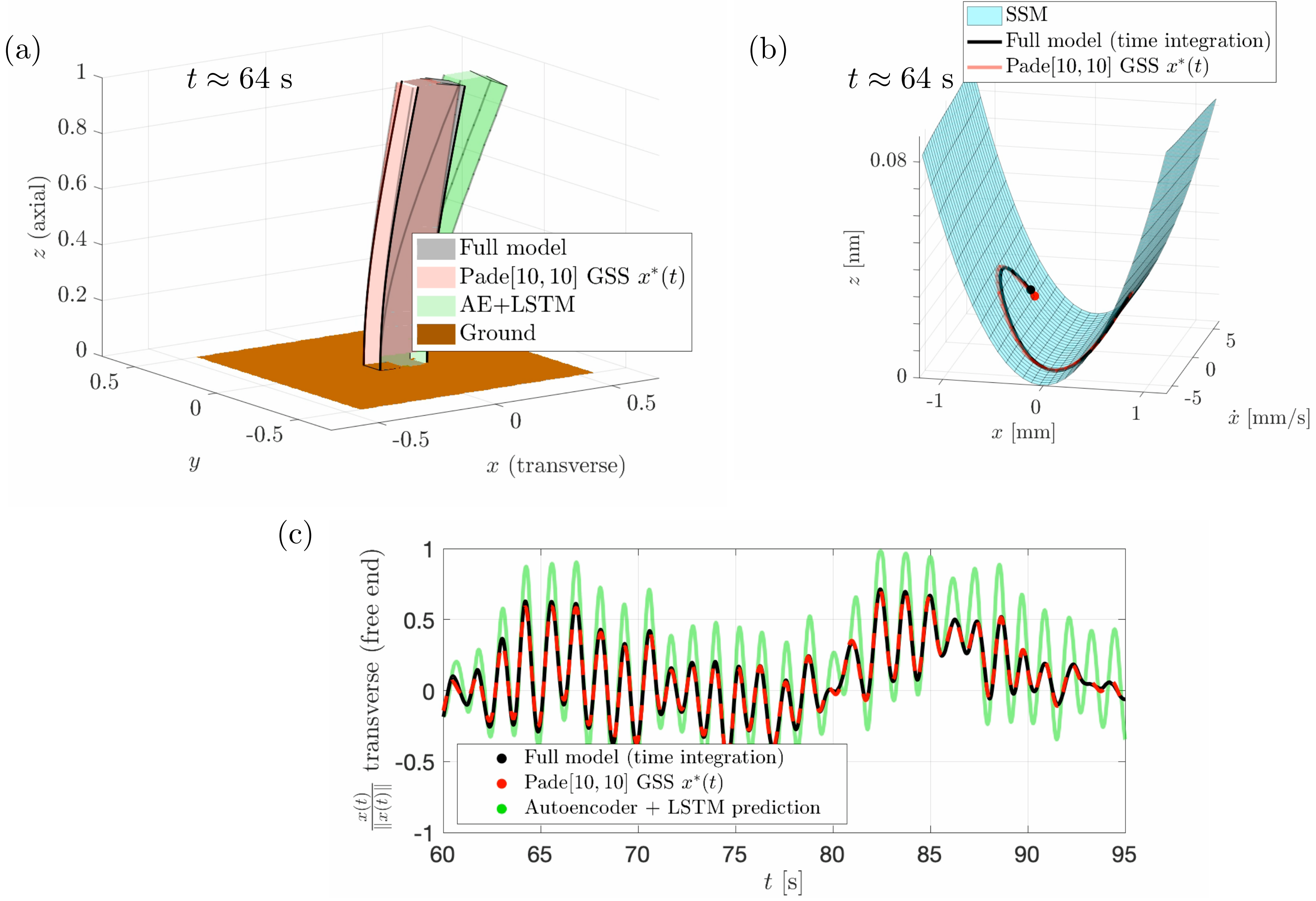}
    \par\end{centering}
    \caption{(a) Snapshot of the beam profiles for various models at $t \approx 64 \text{ }\mathrm{[s]}$. (b) SSM (cyan) snapshot at the same time with trajectory plots of the Pad\'e approximated GSS (red) and the full model (black). (c) Normalized transverse displacement response of the free end. The autoencoder (AE) + LSTM prediction is depicted in green in all the plots. The ground was forced horizontally in the $x$ direction with a random signal drawn from a filtered Gaussian noise model.}
    \label{fig:beam_SSM}
    \end{figure}

\begin{table}[h]
\centering
\caption{Run time comparisons for the SSM-reduced model for the von-K\`arm\`an beam}
\label{tab:pade_beam}
\begin{tabular}{lccc}
\toprule
 & Full model & SSM (2D) + $\text{Pad\'e}[10,10]$ GSS & Autoencoder + LSTM (6D) \\
\midrule
\multirow{2}{*}{Run times} & \multirow{2}{*}{3.14 min} & Training: 1 min & Training: 25 min \\
                           & & Testing: \textbf{3.3 s}   & Testing: 2.8 min \\
\bottomrule
\multicolumn{4}{c}{All the simulations are run on a M4 Macbook Pro 24 GB RAM laptop.} \\
\bottomrule
\end{tabular}
\end{table}

Figures \ref{fig:beam_SSM} a, \ref{fig:beam_SSM} b, plot temporal snapshots of the beam and the forced SSM-reduced predictions. We observe the Pad\'e GSS approximation accurately matches the full model response, while the autoencoder+LSTM predicts a larger response (see Fig.\ref{fig:beam_SSM}c). Added to the accuracy in prediction, the run-times in table \ref{tab:pade_beam} indicate that the SSM-based GSS approximation is 25 times faster in training and 50 times faster in testing compared to the autoencoder+LSTM method. The run times also validate the computational speed-ups obtained via the SSM-reduced model compared to the full model simulation. 

\subsection{Example 6: Curved von K\'arm\'an shell subject to pressure fluctuations}

Our final example involves a curved shell structure satisfying the von K\'arm\'an strain theory. SSM-reduction methods have already been used to predict the forced response of this shell under periodic forcing (see \citet{jain2022} and \citet{li22a}). The shell has a rectangular geometry with length $ L = 1 \text{ [m]}$ and breadth $ H = 0.5 \text{ [m]}$ and is simply supported along the breadth edges. The finite element model describing the shell has 1,386 degrees of freedom. Each element of the shell has 6 degrees of freedom.  We denote the in-plane positions by $x,y$, the out-off plane position by $z$ and the angles by $\alpha, \beta, \text{ and } \gamma$.  The shell is made of aluminum of density  $2700 \text{ [kg/m$^3$]}$, has a thickness of $ 0.01 \text{ [m]}$, Young's modulus $70 \times 10^9 \text{ [Pa]}$, Poisson's ratio $0.33$ and viscous damping $10^{10} \text{ [Ns/m]}$. The curvature of the shell is set by the ensuring that the shell covers a part of a spherical surface with radius $R = 0.01 \text{ [m]}$. The schematic of the shell is sketched in Fig.\ref{fig:shell_setup}a. We will plot predictions for the degrees of freedom associated with elements near points M $\left(\frac{L}{2},\frac{H}{2}\right)$ and A $\left(\frac{3L}{4},\frac{H}{4}\right)$. 

The external forcing is a uniform pressure field acting normal to the shell's surface. This forcing configuration is of practical interest in fluid-structure interaction problems, as it models the effect of pressure fluctuations, when the shell is submerged in a possibly turbulent fluid flow. To test our method across a range of forcing profiles, we consider cases wherein this pressure field is either quasi-periodic (see Fig. \ref{fig:shell_setup}b) of the form
\begin{equation}
g(t) = \sin(\Omega_1 t) + \sin(\Omega_2t),
\end{equation}
or chaotic of the form 
\begin{align}
\label{eq:rossler}
\dot{x} &= -y - z,  \\ \nonumber
\dot{y} &= x + a y,  \\
\nonumber
\dot{z} &= b + z(x - c),  \\ 
\nonumber
g(t) &= x(t)/|x|,
\end{align}
where the chaotic signal is the extracted from the R\"ossler system (see \citet{ROSSLER1976}) with parameters $a=0.2$, $b=0.2$, and $c=5.7$. This forcing is suitably normalized and rescaled to have a maximum forcing magnitude $\Delta = 121.8 \text{ [N]}$. For the quasi-periodic case, we will also ensure that the forcing signal is padded with $T_p = 100$ (see eq. (34)), since the forcing is non-zero for all times.

\begin{figure}[H]
    \begin{centering}
    \includegraphics[width=1\textwidth]{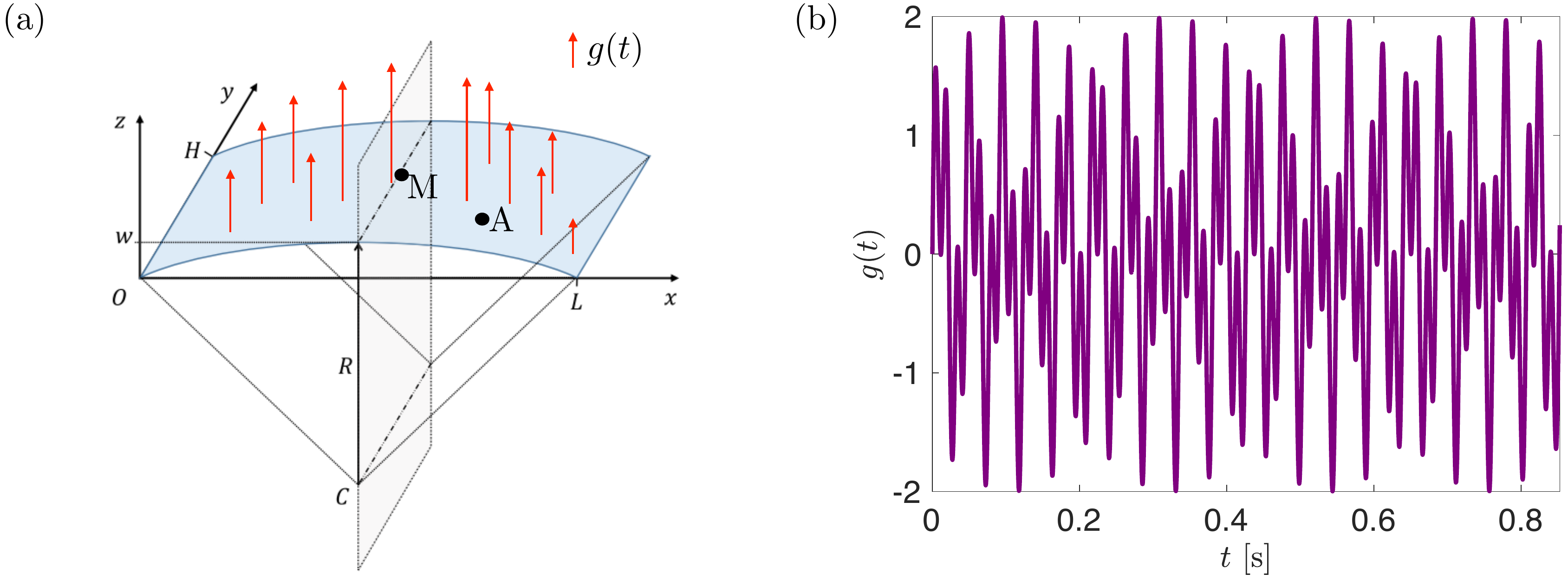}
    \par\end{centering}
    \caption{(a) Schematic of the von K\'arm\'an shell taken from \citet{jain2022}. (b) Quasi-periodic forcing signal applied uniformly across the shell.}
    \label{fig:shell_setup}
    \end{figure}

Due to the high dimensionality of the shell, we apply the mode selection criterion outlined in Section \ref{subsec:mode_select} to compute the GSS approximations. For quasi-periodic forcing with time spacing $\delta t = 10^{-4}$, the criterion retains the 200 slowest modes to ensure accurate evaluation of the integral formulas in the GSS (see Fig. \ref{fig:mode_select}). For the case of the chaotic forcing with time spacing $\delta t = 10^{-3}$, 150 slowest modes prove sufficient to maintain accuracy. 
   
\begin{figure}[H]
\centering
\begin{subfigure}{0.45\textwidth}
    \includegraphics[width=\textwidth]{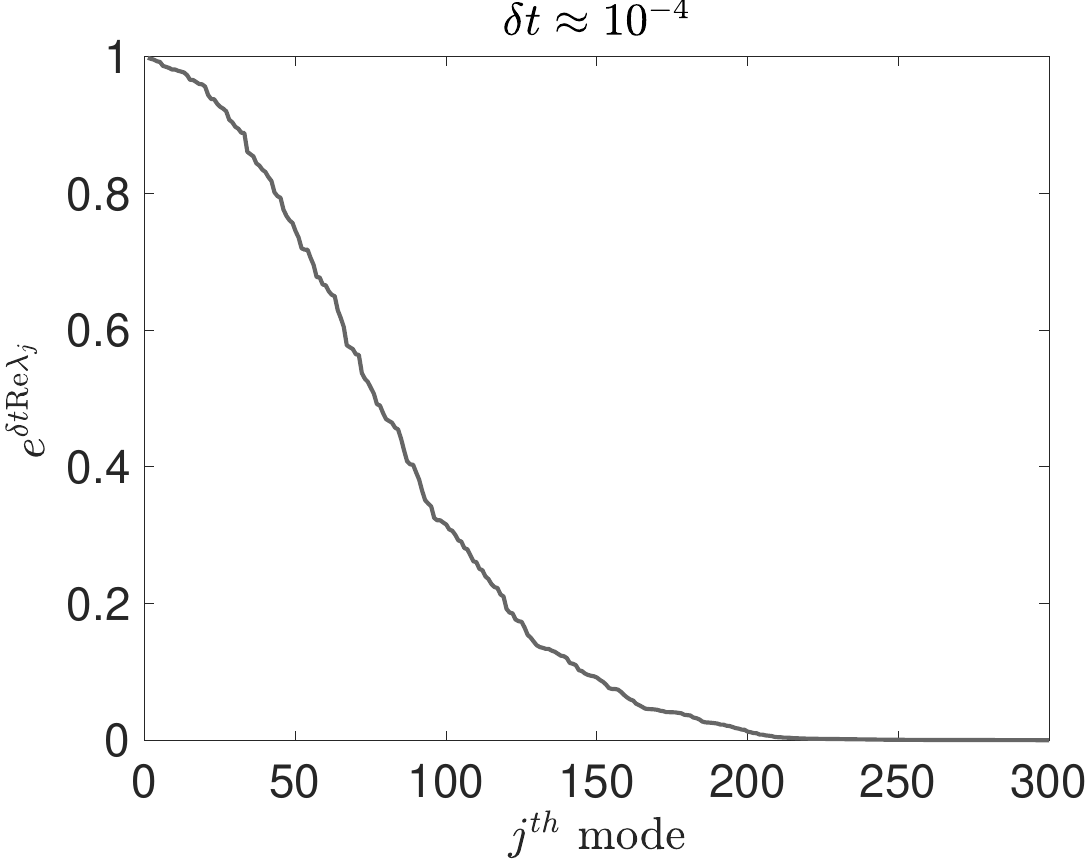}
    \caption{}
    \label{fig:mode_select}
\end{subfigure}
\hfill
\begin{subfigure}{0.5\textwidth}
    \includegraphics[width=\textwidth]{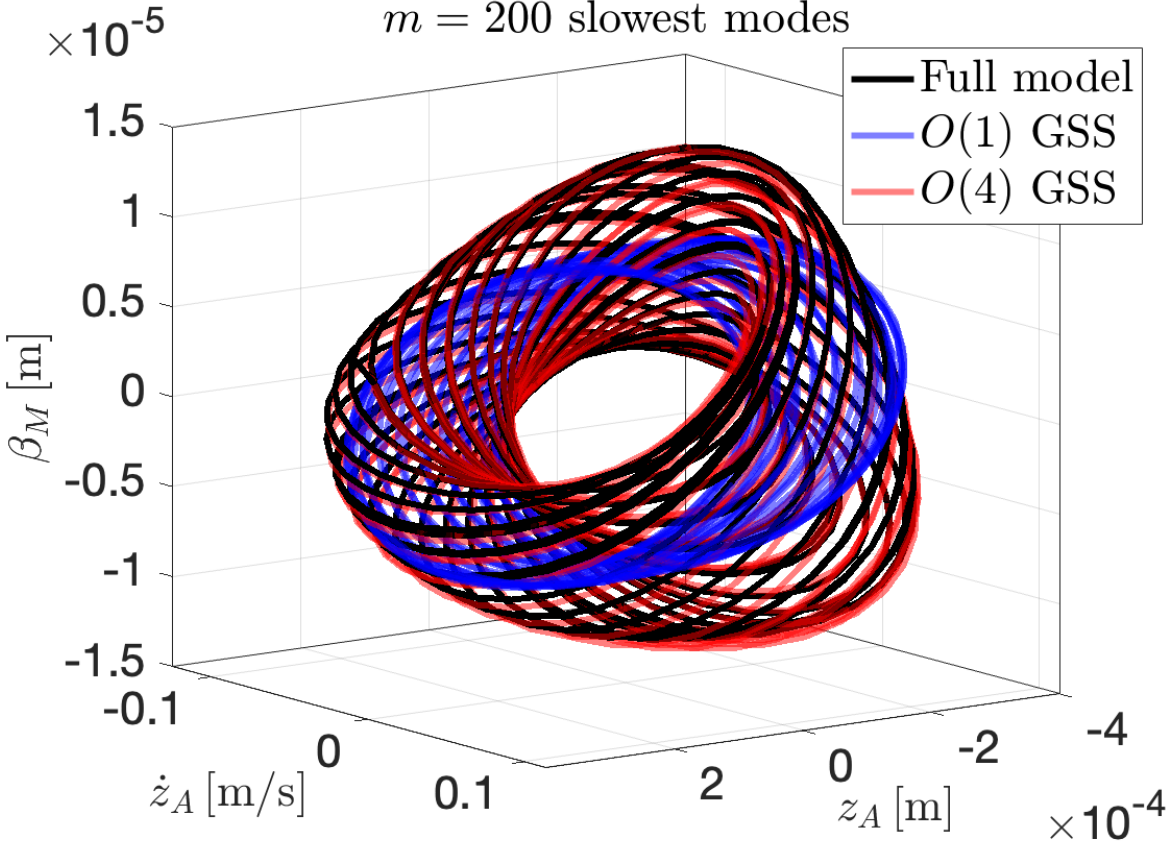}
    \caption{}
    \label{fig:qp_shell}
\end{subfigure}
\caption{(a) Plot of $e^{\delta t\mathrm{Re}\lambda_j}$ vs the $j^{th}$ mode, with modes ordered from slowest to fastest. (b) The forced response of the shell under uniform quasi-periodic forcing for the full model (black), the $O(1)$ GSS (blue) and the $O(4)$ GSS (red). The quasi-periodic signal used in this simulation has frequencies $\Omega_1 = \omega_1$ and $\Omega_2 = \pi \omega_1$.}
\label{fig:shell_qp_main}
\end{figure}

We calculate an $O(4)$ GSS approximation for the quasi-periodic external forcing.  Figure \ref{fig:qp_shell} shows the quasi-periodic GSS when $\Omega_1 = \omega_1$ and $\Omega_2 =\pi \omega_1$, where $\omega_1$ is the frequency corresponding to the slowest mode of the shell. In Fig.\ref{fig:omega12_qp_shell}, we plot the GSS for the case $\Omega_2 = \omega_2$ where $\omega_2$ is the frequency corresponding to the second slowest mode of the shell. In Fig. \ref{fig:period3_qp_shell}, we plot a period-3 GSS  for the case $\Omega_2 = 3\omega_1 $. Finally, in Fig.\ref{fig:rossler_shell}, we plot an $O(6)$ GSS approximation when the uniform pressure field is a chaotic signal taken from  eq. (\ref{eq:rossler}). 

In all the cases, we demonstrate that a nonlinear approximation of the GSS is needed to accurately predict the forced response of the full model. Table \ref{tab:runtime_shell} lists the computational run-times. We observe that our mode truncation step results in run-times that are 6 times faster than the full model simulation.

    \begin{figure}[H]
\centering
\begin{subfigure}{0.45\textwidth}
    \includegraphics[width=\textwidth]{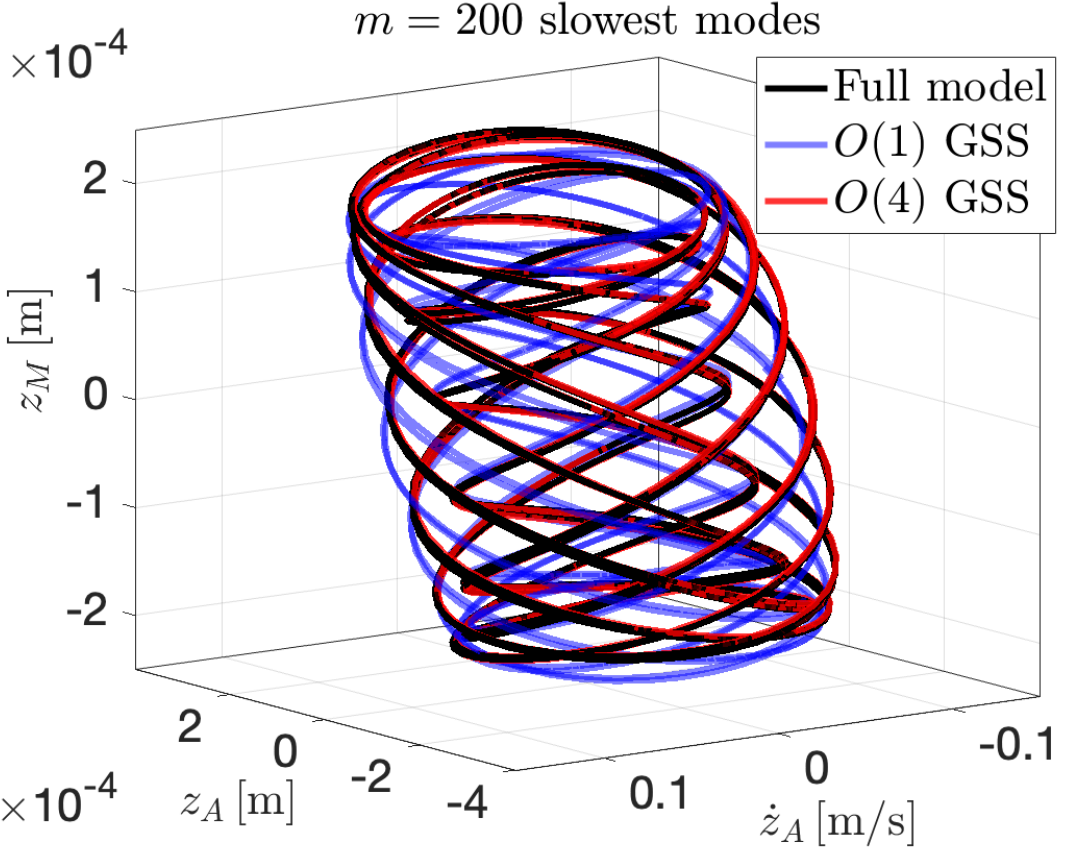}
    \caption{}
    \label{fig:omega12_qp_shell}
\end{subfigure}
\hfill
\begin{subfigure}{0.45\textwidth}
    \includegraphics[width=\textwidth]{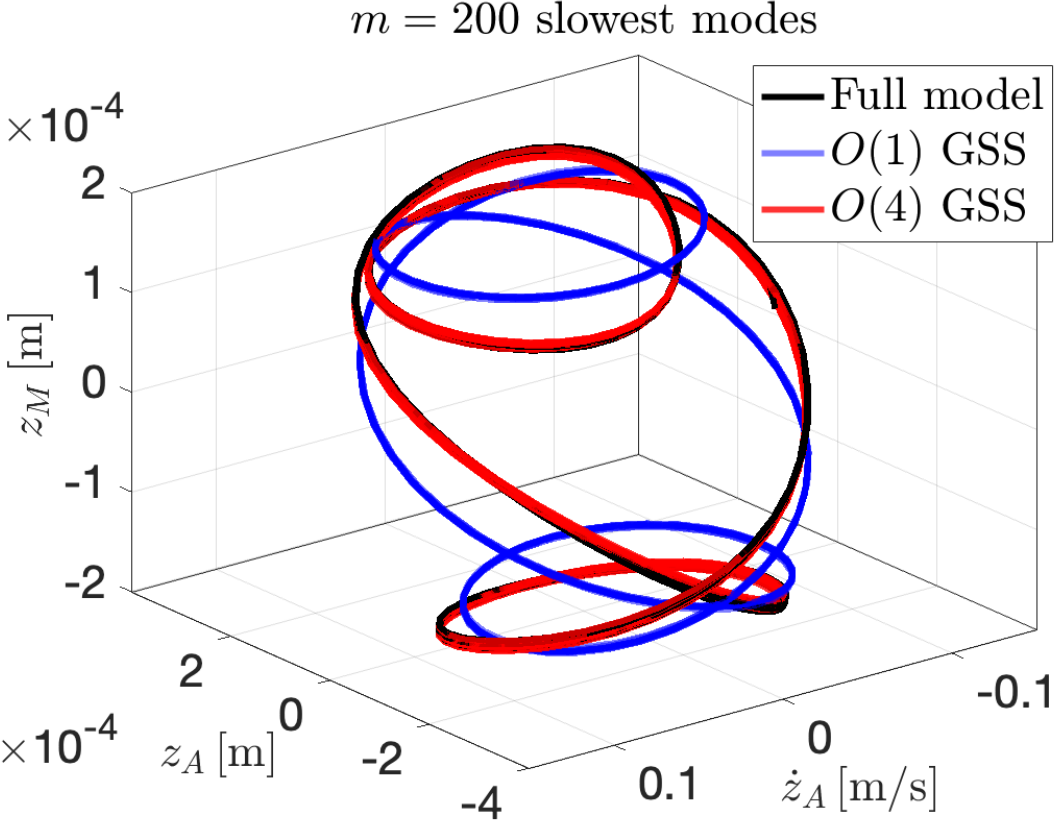}
    \caption{}
    \label{fig:period3_qp_shell}
\end{subfigure}
\hfill 
\begin{subfigure}{0.45\textwidth}
    \includegraphics[width=\textwidth]{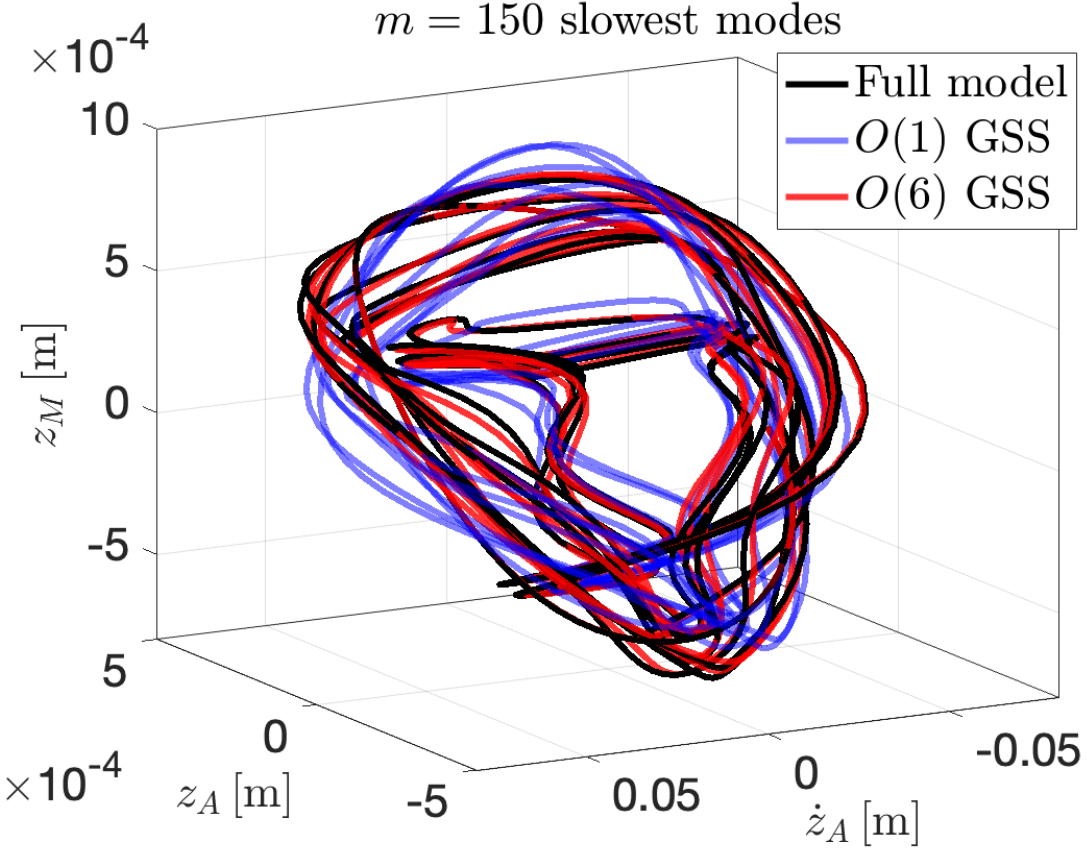}
    \caption{}
    \label{fig:rossler_shell}
\end{subfigure}
\caption{(a) The forced response of the shell under uniform quasi-periodic forcing with $\Omega_1 = \omega_1$ and $\Omega_2 = \omega_2$. (b) Period-3 shell response when  $\Omega_1 = \omega_1$ and $\Omega_2 = 3\omega_1$. (c) Chaotic shell response when the uniform pressure field is modeled using a chaotic R\"ossler system (\ref{eq:rossler}).}
\label{fig:shell_forcing}
\end{figure}

\begin{table}[h]
\centering
\caption{Run time comparisons for curved von K\'arm\'an shell}
\label{tab:runtime_shell}
\begin{tabular}{lcc}
\toprule
 & Quasiperiodic forcing & Chaotic forcing \\
\midrule
Full model & 3.35 min & 1.5 min \\
GSS & \textbf{31.8 s} $O(4)$ & \textbf{32.7 s} $O(6)$ \\
\bottomrule
\multicolumn{3}{c}{All the simulations are run on a M4 Macbook Pro 24 GB RAM laptop.} \\
\bottomrule
\end{tabular}
\end{table}

\section{Conclusion}
In this paper, we have used the recent theoretical results on generalized steady states (GSSs) of non-autonomous dynamical systems to develop a computational algorithm to locate GSSs in nonlinear mechanical systems under external forcing with arbitrary time dependence. We have focused on extracting GSSs via an asymptotic Taylor or Pad\'e expansion, assuming that the general forcing action on the system is specified as a temporal data matrix. We have also assumed that at least the linear part of the equations of motion is known.

Our approach has yielded a numerical toolkit, \textit{GSSTool} (see \citet{kaundinya_gsstool}), for practical settings in which unforced equilibria are stable and the forcing is purely external. The basic algorithm behind \textit{GSSTool} involves the evaluation of exponential kernel integrals, where the integrand is recursively defined by lower-order GSS Taylor expansion coefficients. Inspired by existing recursions offered by \textit{SSMTool}, we extended them to handle temporal GSS Taylor expansion coefficients to store minimally required terms for rapid recursive calculations. Equipped with these recursively computed integrands, we have also devised numerical integral schemes to evaluate exponential kernel integrals to the same accuracy as well-established implicit integrators such as the Newmark method.  

For larger amplitude forcing cases, when the GSS Taylor approximations diverge, we have provided a systematic re-summation technique to obtain convergent vector Pad\'e GSS approximations. In addition, we enhance the algorithm to handle structurally damped systems efficiently. We have also provided a numerically justified modal truncation criterion, enabling computational speed-ups while preserving accuracy.   

The high accuracy and speed of our method have been demonstrated across six examples that differ in forcing profiles, damping types, nonlinear complexity, and dimensionality. We also used \textit{GSSTool} on justified data-driven reduced models obtained via \textit{SSMLearn} for aperiodically forced mechanical systems. The examples in Sections 3 and 5 illustrate the success of this hybrid, partially data-driven approach, in which a 2-dimensional aperiodic SSM sufficed to predict the aperiodic GSS accurately. By contrast, state-of-the-art reduced-order models learned using an autoencoder and a long short-term memory (LSTM) neural network trained on forced data have turned out to require a 6- or 9-dimensional latent space. Yet, they still fail to approximate the aperiodic GSS with acceptable accuracy.

On the algorithmic side, further work is needed to (a) deploy \textit{GSSTool} for aperiodic forced response prediction on data-assisted SSM-reduced models for physical systems with degrees of freedom in the order of a million and (b) enable computation of GSSs for aperiodic parametric forcing. We believe that \textit{GSSTool}, coupled with robust methods like \textit{SSMTool} and \textit{SSMLearn}, can be successfully employed in commercial finite-element packages to compute aperiodic GSSs rapidly. 

\section*{Acknowledgements}
This work has been partially supported by grant $200021\_214908$ from the Swiss National Science Foundation (SNF). We are grateful to Thomas Simpson, Nikolaos Dervilis, Vlachas Konstantinos, and Eleni Chatzi for access to their implementation of the autoeconder+LSTM code. We also thank Eleni Chatzi for valuable suggestions on improving the training of large autoencoder+LSTM models.

\begin{appendices}
\section{Theorem and proof for the existence and uniqueness of a GSS}
\label{app:thm2}
\citet{haller24_wa} for a first-order system (\ref{eq:first_order_system}) and for forcing condition eq. (\ref{eq:uniform_boundedness}), prove hyperbolic fixed points of the unforced system survive as a unique uniformly bounded trajectories, retaining the qualitative hyperbolic behavior of the unforced fixed points. We restate the theorems for the setting when the hyperbolic fixed point is asymptotically stable and provide refined conditions that explicitly discuss convergence to the unique uniformly bounded trajectory.

\begin{thm}
[\textbf{Unique generalized steady state (GSS)}]
\label{thm:unique_GSS}
Consider a $\delta$-size ball $B_{\delta} \subset V$ centered about $\mathbf{z} = 0$, within this open-ball and for all $t \in \mathbb{R}$ assume the following: 
\begin{itemize}
    \item[] (A1) $\mathbf{F}$ and $\mathbf{G}$ are Lipschitz continuous with Lipschitz constants $L_\delta^F$ and $L_\delta^G$. 
    \item[] (A2) The conditions 
    \begin{align}
\label{eq:conditions}
\Delta \leq \delta \left( \frac{1}{\|\mathbf{V}\| \|\mathbf{V}^*\| \Gamma(\mathbf{\Lambda})} - 2(L^F_\delta + L^G_\delta)\right) - \|\mathbf{F}(\mathbf{z})\|, \quad  \quad  2\|\mathbf{V}\| \|\mathbf{V}^*\| \Gamma(\mathbf{\Lambda}) (L^F_\delta + L^F_\delta) \leq a <1,
\end{align}
are satisfied for some $a \in (0,1)$ and $\delta >0$. 
\end{itemize}

Then there exists a unique uniformly bounded generalized steady state (GSS) $\mathbf{z}^*(t)$ in $B_{\delta}$. The GSS is asymptotically stable and is as smooth in any parameter as system (\ref{eq:first_order_system}).  
\end{thm}
\begin{proof}
We define a class of uniformly bounded $C^0$ trajectories in $B_{\delta}$:
\begin{equation}
C_{\delta}[\mathbb{R}]:=\{ \mathbf{z}:\mathbb{R} \to \mathbb{R}^{2n} |\quad  \|\mathbf{z}(t)\| \leq \delta\}.
\end{equation}
For any $\mathbf{z}(t) \in C_{\delta}$, we can treat the term $\mathbf{F}(\mathbf{z}(t)) + \mathbf{G}(\mathbf{z}(t),t)$  as an uniformly bounded aperiodic forcing applied to the linear system $\mathbf{B}\dot{\mathbf{z}} = \mathbf{A}\mathbf{z}$. Under this setting and using the transformation matrix $\mathbf{V}$, the general solution of system (\ref{eq:first_order_system}) for an arbitrary initial condition $\mathbf{z}_0$ is given by (see \citet{burd07}),
\begin{equation}
\label{eq:integral_equation_1}
\mathbf{z}(t) = \lim_{{t_0 \to -\infty}} \mathbf{V}e^{\mathbf{\Lambda}(t-t_0)}\mathbf{V}^* \mathbf{z}_0 + \int_{-\infty}^{t} \mathbf{V}e^{\mathbf{\Lambda}(t-s)}\mathbf{V}^*  \left[\mathbf{F}(\mathbf{z}(s)) + \mathbf{G}(\mathbf{z}(s),s) \right] ds.
\end{equation}
The first term in eq. (\ref{eq:integral_equation_1}) will vanish due to $\text{Re[Spec}(\mathbf{\Lambda})] < 0$. This results in an integral equation for the GSS in the form
\begin{equation}
\label{eq:integral_equations_GSS}
\mathbf{z}(t) =  \int_{-\infty}^{t} \mathbf{V}e^{\mathbf{\Lambda}(t-s)}\mathbf{V}^* \left[\mathbf{F}(\mathbf{z}(s)) + \mathbf{G}(\mathbf{z}(s),s) \right] ds.
\end{equation}
The simplest scheme to solve eq. (\ref{eq:dichotmy}), is to employ Picard iteration to the map
\begin{equation}
\label{eq:picard_map}
\mathbf{z}_{l+1}(t) = \mathcal{G}_P(\mathbf{z}_l)= \int_{-\infty}^{t} \mathbf{V}e^{\mathbf{\Lambda}(t-s)}\mathbf{V}^* \left[\mathbf{F}(\mathbf{z}_l(s)) + \mathbf{G}(\mathbf{z}_l(s),s) \right] ds,
\end{equation}
using random initial guess trajectories in $C_\delta$. The method is guaranteed to converge to a unique GSS if the mapping $\mathcal{G}_P$ is a contraction mapping on $C_\delta$. Next, we provide conditions for the convergence of iterations of $\mathcal{G}_P$. 

We first need to prove the mapping $\mathcal{G}_p$ maps all uniformly bounded functions in the open ball $B_\delta$ into $B_\delta$. Take $\mathbf{z}(t), \text{ } \tilde{\mathbf{z}}(t) \in C_\delta$, we can always write 
\begin{equation}
\mathcal{G}_p(\mathbf{\tilde z}(t)) = \mathcal{G}_p(\mathbf{z}(t)) + \int_{-\infty}^t \mathbf{V} e^{\boldsymbol \Lambda (t-s)} \mathbf{V}^* \left[\mathbf{F}(\mathbf{\tilde z}(s)) -\mathbf{F}(\mathbf{ z}(s)) + \mathbf{G}(\mathbf{\tilde z}(s),s) - \mathbf{G}(\mathbf{z}(s),s) \right] ds. 
\end{equation}
Taking the sup norm on both sides, we obtain 
\begin{align}
\|\mathcal{G}_p(\mathbf{\tilde z}(t))\| \leq \|\mathcal{G}_p(\mathbf{z}(t))\| + \Bigg\|\int_{-\infty}^t \mathbf{V} e^{\boldsymbol \Lambda (t-s)} \mathbf{V}^* \Big[\mathbf{F}(\mathbf{\tilde z}(s)) -&\mathbf{F}(\mathbf{ z}(s))  \quad + \nonumber \\
&\mathbf{G}(\mathbf{\tilde z}(s),s) - \mathbf{G}(\mathbf{z}(s),s) \Big] ds \Bigg\|, \nonumber
\end{align}
\begin{align}
\|\mathcal{G}_p(\mathbf{\tilde z}(t))\| \leq \|\mathcal{G}_p(\mathbf{z}(t))\| + \|\mathbf{V}\|\|\mathbf{V}^*\| \text{sup}_t \Bigg(\int_{-\infty}^t  |e^{\boldsymbol \Lambda (t-s)} |  \Big[ |\mathbf{F}(\mathbf{\tilde z}(s)) -&\mathbf{F}(\mathbf{ z}(s))| \quad + \nonumber \\
&|\mathbf{G}(\mathbf{\tilde z}(s),s) - \mathbf{G}(\mathbf{z}(s),s)| \Big] ds  ]\Bigg) \nonumber,
\end{align}
\begin{equation}
\|\mathcal{G}_p(\mathbf{\tilde z}(t))\| \leq \|\mathcal{G}_p(\mathbf{z}(t))\| + \|\mathbf{V}\|\|\mathbf{V}^*\| \mathrm{sup}_t\left( \int_{-\infty}^t  |e^{\boldsymbol \Lambda (t-s)} |  \left[ L^F_{\delta} |\mathbf{z}(s)-\mathbf{\tilde z}(s)| + L^G_{\delta} |\mathbf{z}(s)-\mathbf{\tilde z}(s)| \right] ds \right) .\nonumber
\end{equation}
From the definition of $C_\delta$, we have $|\mathbf{z}(s)-\mathbf{\tilde z}(s)|\leq \|\mathbf{z}(s)-\mathbf{\tilde z}(s)\|\leq 2 \delta$, which gives us 
\begin{equation}
\label{eq:main_inequality}
\|\mathcal{G}_p(\mathbf{\tilde z}(t))\| \leq \|\mathcal{G}_p(\mathbf{z}(t))\| + \|\mathbf{V}|\|\mathbf{V}^*\| 2 \delta (L^F_{\delta} + L^G_{\delta}) \text{sup}_t\left(\int_{-\infty}^t  |e^{\boldsymbol \Lambda (t-s)} | ds \right). 
\end{equation}
We can similarly estimate 
\begin{equation}
\|\mathcal{G}_p(\mathbf{ z}(t))\| \leq \|\mathbf{V}\| \|\mathbf{V}^*\| \Gamma(\mathbf{\Lambda}) \Big( \Delta + \|\mathbf{F}(\mathbf{z}(s))\| \Big)  , \nonumber
\end{equation}
with 
\begin{equation}
 \Gamma(\mathbf{\Lambda}) = \text{sup}_t\left(\int_{-\infty}^t  |e^{\boldsymbol \Lambda (t-s)} | ds \right) = \mathop{\text{max}}_{1\leq j \leq n} \left[\frac{1}{|\mathrm{Re}[\lambda _j]|}\right]  . \nonumber
\end{equation}
Plugging all this back into eq.(\ref{eq:main_inequality}) gives
\begin{equation}
\|\mathcal{G}_p(\mathbf{\tilde z}(t))\| \leq \|\mathbf{V}\| \|\mathbf{V}^*\| \Gamma(\mathbf{\Lambda}) \Big( \Delta + \|\mathbf{F}(\mathbf{z}(s))\| \Big) + \|\mathbf{V}|\|\mathbf{V}^*\| 2 \delta (L^F_{\delta} + L^G_{\delta}) \Gamma(\mathbf{\Lambda}),
\end{equation}
by requiring $\|\mathcal{G}_p(\mathbf{\tilde z}(t))\| \leq \delta$ we get the first condition in eq.(\ref{eq:conditions}). Next step is to prove the mapping $\mathcal{G}_p : C_\delta \to C_{\delta}$ is a contraction. We start by estimating 
\begin{equation}
\|\mathcal{G}_p(\mathbf{z}(s))-\mathcal{G}_p(\mathbf{\tilde z }(s))\| \leq \|\mathbf{V}|\|\mathbf{V}^*\|  (L^F_{\delta} + L^G_{\delta}) \text{sup}_t\left(\int_{-\infty}^t  |e^{\boldsymbol \Lambda (t-s)} | |\mathbf{z}(s)-\mathbf{\tilde z}(s)| ds \right), \nonumber 
\end{equation}
\begin{equation}
\label{eq:final_ineq}
\|\mathcal{G}_p(\mathbf{z}(s))-\mathcal{G}_p(\mathbf{\tilde z }(s))\| \leq \|\mathbf{V}|\|\mathbf{V}^*\|  (L^F_{\delta} + L^G_{\delta}) \text{sup}_t\left(\int_{-\infty}^t  |e^{\boldsymbol \Lambda (t-s)} |  ds \right) \|\mathbf{z}(s)-\mathbf{\tilde z}(s)\|. \nonumber 
\end{equation}
For a mapping to be a contraction we need 
\begin{equation}
\|\mathcal{G}_p(\mathbf{z}(s))-\mathcal{G}_p(\mathbf{\tilde z }(s))\| \leq a \|\mathbf{z}(s)-\mathbf{\tilde z}(s)\|, 
\end{equation}
with $a < 1$. Enforcing this on eq. (\ref{eq:final_ineq}) provides us with the second condition in eq. (\ref{eq:conditions}).  

We also observe that if set $\Gamma(\Lambda) \equiv \frac{1}{\kappa}$ (where $\kappa$ is defined in eq.(\ref{eq:dichotmy})), $\|\mathbf{V}\|\|\mathbf{V}^*\| \equiv K$ and further impose the condition $4(L^F_{\delta}+L^G_{\delta})K \leq \kappa$ (see \citet{palmer73} and \citet{haller24_wa}), conditions (\ref{eq:conditions}) modify to 
\begin{align}
\Delta \leq \delta \left( \frac{\kappa}{K} - 2\frac{\kappa}{4K}\right) - \|\mathbf{F}(\mathbf{z})\|, \quad   2\frac{K}{\kappa} (L^F_\delta + L^F_\delta) \leq a <1.
\end{align}
We recover conditions stated in assumption (A2) in Theorem \ref{thm:unique_GSS} for $a=0.25$. It appears that the conditions from \citet{palmer73} and \citet{haller24_wa} are stricter bounds for the existence of the GSS compared to the derived bounds to demonstrate convergence via Picard iteration. Overall, the conditions listed in assumption (A2) guarantee the mapping is a contraction, hence, proving the existence of a unique uniformly bounded solution to the integral equations (\ref{eq:integral_equation_1}). Specifically, the strength of convergence to the GSS is dictated by the factor 
\begin{equation}
\label{eq:Estimate_convg}
\Gamma(\mathbf{\Lambda}) = \mathop{\text{max}}_{1\leq j \leq n} \left[\frac{1}{|\mathrm{Re}[\lambda _j]|}\right] 
\end{equation}
and the validity of the contraction mapping is determined by maximum magnitude of forcing $\Delta$.
\end{proof}

Theorem \ref{thm:unique_GSS} provides a local existence and uniqueness result for a generalized steady state (GSS) that is confined to a $\delta$-size ball close to the origin. Equation (\ref{eq:integral_equations_GSS}) and Theorem \ref{thm:unique_GSS} are generalizations of the results stated in \citet{jain19} for periodic and quasi-periodic forcing. Assumption (A2) conditions provide conservative estimates on the magnitude of external forcing $\Delta$. A faster convergence to the GSS is only guaranteed if the maximum forcing magnitude is small and the system is highly damped. Most physical setups will experience much larger forcing and suffer from low damping which result in Picard-type iterations of obtaining a unique GSS to be computationally intractable.

\section{Deriving the GSS Taylor coefficients from recursions used in \textit{SSMTool}}
\label{app:radial_derivative_Faa_Di_bruno}

We first rewrite the left-hand side of eq. (\ref{eq:radial_derivative}) as 
\begin{equation}
\left[\sum_{q=1}^{\nu}\left(\sum_{j=1}^{q} \mathbf{z}_{({l}_j)}(s)  \right)^{\boldsymbol{\gamma} }\right]_{O(\Delta^{\nu})} = \left[\sum_{q=1}^{\nu} \prod_{i=1}^{2n} \left( \sum_{j=1}^q \mathbf{z}_{({l}_j)}^i(s)\right)^{\boldsymbol{\gamma}_i}\right]_{O(\Delta^{\nu})}.
\end{equation}
Expanding each of the the multi-monomial functions appearing in the sum-product, we obtain
\begin{equation}
 \left[\sum_{q=1}^{\nu} \prod_{i=1}^{2n} \left( \sum_{j=1}^q \mathbf{z}_{({l}_j)}^i(s)\right)^{\boldsymbol{\gamma}_i}\right]_{O(\Delta^{\nu})} = \left[\sum_{q=1}^{\nu} \prod_{i=1}^{2n} \left(\sum_{\boldsymbol{\gamma}_i = \sum_{j=1}^q \mathbf{k}_{ji}} \boldsymbol{\gamma}_i! \prod_{j=1}^q \frac{\left(\mathbf{z}_{({l}_j)}^i (s)\right)^{\mathbf{k}_{ji}}}{\mathbf{k}_{ji}!}  \right) \right]_{\sum_{j=1}^q\sum_{i=1}^{2n} \mathbf{k}_{ji}l_j = \nu}.
\end{equation}
By grouping $\boldsymbol{\gamma}_i!$ with the product outside the bracket and also moving the first summation to the inside of the bracket, we can rearrange terms to match the factor eq. (\ref{eq:z_factor}) appearing in the multivariate Faa Di Bruno formulas \citet{constantine96}. The multi-monomial expansion brings out the index set eq. (\ref{eq:index_set}) as well, where the vectors $l_j$ by definition of the GSS Taylor expansion are already ordered. Hence, we have justified the use of already existing compositional structure from \textit{SSMTool} to compute the GSS Taylor expansion coefficients. 

\section{Derivation of $\mathbf{Q}_j(\delta t)$ for second-order systems}
\label{app:second_qj}

We diagonalize the matrix $\boldsymbol{\lambda}_j$ in eq. (\ref{eq:block_second_order}), 
\begin{equation}
\label{eq:diag_lambda}
\boldsymbol{\lambda}_j = \frac{1}{-2\omega_j \sqrt{\zeta^2_j-1}}\begin{pmatrix} 1 & 1 \\ \lambda_+ & \lambda_- \end{pmatrix}\begin{pmatrix} \lambda_+ & 0 \\0 & \lambda_-\end{pmatrix} \begin{pmatrix} \lambda_- & -1 \\ \lambda_+ & 1 \end{pmatrix},
\end{equation}
 where we have 
\begin{equation}
\lambda_{\pm}  = \left(-\zeta_j \pm \sqrt{\zeta_j^2-1}\right)\omega_j.
\end{equation}
We expand explicitly the exponential appearing in the integral eq. (\ref{eq:second_order_integral}) using eq. (\ref{eq:diag_lambda}) to obtain
\begin{equation}
\label{eq:integral}
e^{-\boldsymbol{\lambda}_j s} = \frac{1}{-2\omega_j \sqrt{\zeta^2_j-1}}\begin{pmatrix} 1 & 1 \\ \lambda_+ & \lambda_- \end{pmatrix}\begin{pmatrix} e^{-\lambda_+s} & 0 \\0 & e^{-\lambda_-s} \end{pmatrix} \begin{pmatrix} \lambda_- & -1 \\ \lambda_+ & 1 \end{pmatrix}.
\end{equation}
We substitute this expression into the integral eq. (\ref{eq:second_order_integral}) and perform the integration under the  assumption that the forcing $[\boldsymbol\phi(s+t)]_j$ is piecewise linear in the interval $[0,\delta t ]$. The integrals arising from this are 

\begin{equation}
\label{eq:second_integrals}
\frac{1}{-2 \omega_j \sqrt{\zeta_j^2 - 1}} \begin{pmatrix}
\alpha \int_{0}^{\delta t} \left( e^{-\lambda_- s} - e^{-\lambda_+ s} \right) ds + \frac{\beta - \alpha}{\delta t} \int_{0}^{\delta t} s \left( e^{-\lambda_- s} - e^{-\lambda_+ s} \right) ds \\
\alpha \int_{0}^{\delta t} \left( -\lambda_+ e^{-\lambda_+ s} + \lambda_- e^{-\lambda_- s} \right) ds + \frac{\beta - \alpha}{\delta t} \int_{0}^{\delta t} s \left( -\lambda_+ e^{-\lambda_+ s} + \lambda_- e^{-\lambda_- s} \right) ds
\end{pmatrix},
\end{equation}

% Definitions of a and b
where the terms \(\alpha\) and \(\beta\) are defined as,
\begin{equation}
\alpha = [\mathbf{U}^\top \boldsymbol{\phi}(t)]_j, \quad \beta = [\mathbf{U}^\top \boldsymbol{\phi}(t + \delta t)]_j.
\end{equation}

The integrals in eq. (\ref{eq:second_integrals}) can be evaluated by using the fundamental scalar exponential integrals, 
\begin{equation}
 \int_{0}^{\delta t} e^{-\lambda s} ds = \frac{1 - e^{-\lambda \delta t}}{\lambda},
\end{equation}

and
\begin{equation}
\begin{aligned}
\int_0^{\delta t} s e^{-\lambda s} ds = \Bigg[-\frac{s}{\lambda} e^{-\lambda s} - \frac{1}{\lambda^2} e^{-\lambda s} \Bigg]_0^{\delta t}.
\end{aligned}
\end{equation}

For $\zeta_j \neq 1$ and $\zeta_j >0$,  we obtain the matrix expression for \(\mathbf{Q}^{u}_j(\delta t)\) and \(\mathbf{Q}^{d}_j(\delta t)\),
\begin{equation}
\mathbf{Q}^u_j(\delta t) = \mathbf{Q}^d_j(\delta t) = \frac{1}{-2 \omega_j \sqrt{\zeta_j^2 - 1}} \begin{bmatrix}
q_{11} & q_{12} \\
q_{21} & q_{22}
\end{bmatrix},
\end{equation}
where the matrix entries are

\begin{align}
\nonumber q_{11} &= \frac{1 - e^{-\lambda_- \delta t}}{\lambda_-} - \frac{1 - e^{-\lambda_+ \delta t}}{\lambda_+} - \frac{1}{\delta t} \left[ \left( -\frac{\delta t}{\lambda_-} e^{-\lambda_- \delta t} - \frac{1}{\lambda_-^2} e^{-\lambda_- \delta t} + \frac{1}{\lambda_-^2} \right) - \left( -\frac{\delta t}{\lambda_+} e^{-\lambda_+ \delta t} - \frac{1}{\lambda_+^2} e^{-\lambda_+ \delta t} + \frac{1}{\lambda_+^2} \right) \right], \\
\nonumber q_{12} &= \frac{1}{\delta t} \left[ \left( -\frac{\delta t}{\lambda_-} e^{-\lambda_- \delta t} - \frac{1}{\lambda_-^2} e^{-\lambda_- \delta t} + \frac{1}{\lambda_-^2} \right) - \left( -\frac{\delta t}{\lambda_+} e^{-\lambda_+ \delta t} - \frac{1}{\lambda_+^2} e^{-\lambda_+ \delta t} + \frac{1}{\lambda_+^2} \right) \right], \\
\nonumber q_{21} &= e^{-\lambda_+ \delta t} - e^{-\lambda_- \delta t} - \frac{1}{\delta t} \left[ \left( \delta t e^{-\lambda_+ \delta t} + \frac{1}{\lambda_+} e^{-\lambda_+ \delta t} - \frac{1}{\lambda_+} \right) - \left( -\delta t e^{-\lambda_- \delta t} - \frac{1}{\lambda_-} e^{-\lambda_- \delta t} + \frac{1}{\lambda_-} \right) \right], \\
\nonumber q_{22} &= \frac{1}{\delta t} \left[ \left( \delta t e^{-\lambda_+ \delta t} + \frac{1}{\lambda_+} e^{-\lambda_+ \delta t} - \frac{1}{\lambda_+} \right) - \left( -\delta t e^{-\lambda_- \delta t} - \frac{1}{\lambda_-} e^{-\lambda_- \delta t} + \frac{1}{\lambda_-} \right) \right].
\end{align}

For the case when $\zeta_j =1$, we take the limit $\text{lim}_{\zeta_j \to 1}\mathbf{Q}^{u}_j(\delta t)$ to derive the matrix expression for \(\mathbf{Q}^{c}_j(\delta t)\), we obtain
\begin{equation}
\mathbf{Q}^c_j(\delta t) = \begin{bmatrix}
q_{11} & q_{12} \\
q_{21} & q_{22}
\end{bmatrix},
\end{equation}
with the matrix entries,

\begin{align}
\nonumber q_{11} &= \frac{1}{\omega_j^2} \left( 1 - e^{-\omega_j \delta t} \left( 1 + \omega_j \delta t \right) \right), 
& \nonumber q_{12} &= \frac{1}{\omega_j^2} e^{-\omega_j \delta t} \left( \omega_j \delta t - 1 + e^{\omega_j \delta t} \right), \\
\nonumber q_{21} &= \frac{1}{\omega_j} e^{-\omega_j \delta t} \left( -1 + e^{\omega_j \delta t} - \omega_j \delta t \right), &
\nonumber q_{22} &= \frac{1}{\omega_j} e^{-\omega_j \delta t} \left( 1 + \omega_j \delta t - e^{\omega_j \delta t} \right).
\end{align}

\section{GSS results for periodic forcing}
\label{app:periodic_frcs}
In this section, we assume that the external forcing is a quasiperiodic function with $q$ rationally incommensurate frequencies $\boldsymbol{\Omega} \in \mathbb{T}^q$. We can then write the integrand in eq. (\ref{eq:fund_integral}) as
\begin{equation}
\boldsymbol{\Phi}(t) =  \sum_{\mathbf{k}\in \mathbb{Z}^{q}} \mathbf{g}_k e^{i\langle \mathbf{k},\Omega\rangle t}. 
\end{equation}
Using this form, we explicitly evaluate the integrals in eq. (\ref{eq:fund_integral}). We focus here on the general damping case, and compute integral formulas for eq. (\ref{eq:L_j_gd}) in the quasiperiodic setting. This results in the expression, 
\begin{equation}
\label{eq:qp_result}
\mathbf{L}_j(t, \delta t) = \sum_{\mathbf{k} \in \mathbb{Z}^q} [\mathbf{V}^{\mathrm{T}} \mathbf{g}_{\mathbf{k}}]_j e^{i \langle \mathbf{k}, \boldsymbol{\Omega} \rangle t} \left( \frac{e^{(-\lambda_j + i \langle \mathbf{k}, \boldsymbol{\Omega} \rangle) \delta t} - 1}{-\lambda_j + i \langle \mathbf{k}, \boldsymbol{\Omega} \rangle} \right).
\end{equation}
We bring attention to the term in the denominator in eq. (\ref{eq:qp_result}). This term will develop a singularity when the eigenvalues of the system are in near-resonance with integer linear combinations of the  external forcing function's frequencies. We remark that this conclusion is absent if one derives formulas for eq. (\ref{eq:L_j_gd}) using the linear piecewise assumption, which only had a singularity when $\lambda_j =0$. 

We demonstrate the sensitive dependence on parameters of our GSS formulas, when we evaluate forced response curves (FRCs) for the oscillator chain model described in Section \ref{subsec:e2} with periodic forcing now applied to the $5^{th}$ mass in the chain. We further consider three different physical setups of the oscillator chain, with different damping parameters $c = 0.01 \text{ [Ns/m]}$ (low damping), $c = 0.1 \text{ [Ns/m]}$ (medium damping) and $c =3 \text{ [Ns/m]}$  (high damping) but same linear and nonlinear stiffness parameters as in Section \ref{subsec:e2}.

\begin{figure}[H]
    \begin{centering}
    \includegraphics[width=1\textwidth]{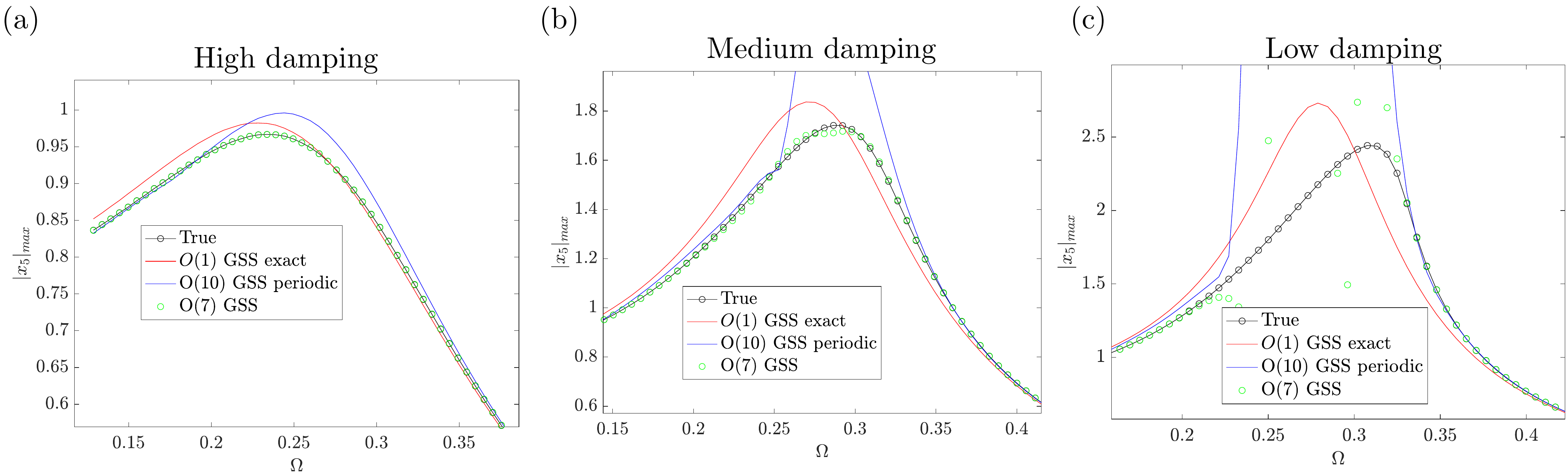}
    \par\end{centering}
    \caption{Forced response curves (FRC) of the full model (black), linear GSS (red), periodic $O(10)$ GSS (blue) and $O(7)$ GSS (green circle)  for the cases when the damping matrix is (a) highly damped  (b) mid damping and (c) lowly damped.}
    \label{fig:app_periodic}
    \end{figure}
    
We refer to the GSS calculated using the formula eq. (\ref{eq:qp_result}) as GSS periodic and at each GSS expansions order we use $5$ Fourier modes to compute eq. (\ref{eq:qp_result}). Since the forcing is periodic, we also calculate the GSS of the linear system exactly, this we label as $O(1)$ GSS exact. Figures \ref{fig:app_periodic} (a)-(c), plots the FRCs for the three different damping regimes. We observe that calculating an $O(7)$ GSS approximation using the piecewise linear approximation of the forcing is reliable across damping cases when the system is not near the resonant frequency. We notice that an $O(10)$ GSS periodic approximation is able to predict well away from the resonant frequency, but in the low and medium damping case the singularity clearly shows up. This exercise has demonstrated that the formulas for the GSS we have derived are consistent with the periodic setting as well. Further, it also cautions the user that \textit{GSSTool} is not suitable to use for predicting responses for near resonant periodic forcing. One should use continuation methods coupled with periodic SSM-reduced models (see \citet{jain2022}, \citet{li22a}, and \citet{li22b}), these have been successful and also are computationally fast to compute FRCs.

\section{LSTM training procedure}
\label{app:lstm}
The architecture and training configurations for the autoencoder (AE) and long short-term memory (LSTM) models, as applied in this study, are as follows:

As proposed for this oscillator chain by \citet{simpson2021}, the encoder consists of three layers: The first two are dense, 20-dimensional linear and tanh activated layers, respectively. A final dense bottleneck layer reduces the dimension of the latent space. 

In contrast, the corresponding decoder mirrors the encoder, reconstructing the input from the compressed representation. While increasing the number of retained dimensions generally improves data recovery, it also raises the number of trainable parameters and computational load in the subsequent computation. As indicated in the example, a bottleneck dimension of nine proved to optimize the trade-off between dimensionality reduction and fidelity in recovery performance (see \citet{simpson2021}). This results in a total of 2,069 trainable parameters.

 We follow the steps outlined in \citet{simpson2021} by training the network on the initial 5,000-time steps of a simulated displacement history using an \textit{Adam} optimizer with gradient clipping and an exponentially decaying learning rate to minimize the risk of overfitting.

Following which a two-layer stacked LSTM network is trained. The LSTM network is composed of two tanh-activated layers with 27 and 20 dimensions, respectively, followed by a linear activation output layer with nine dimensions. This configuration yields a total of 8,241 trainable parameters. The network is trained on the first 5,000 time steps and validated on the following 2,000 steps using a gradient-clipped Adam optimizer with exponential learning rate decay. 

We use a second, LSTM initialized with pre-trained weights and in the stateful mode for response prediction in the reduced space. The stateful LSTM maintains its state across batches, which is necessary to predict sequentially based on previous time steps. We iteratively predict one step ahead based on the 50 time steps. Since the LSTM network's memory dependence limits predictions on the first timesteps, we pad the forcing signal with zeros at the start.

\begin{figure}[H]
\centering
\begin{subfigure}{0.45\textwidth}
    \includegraphics[width=\textwidth]{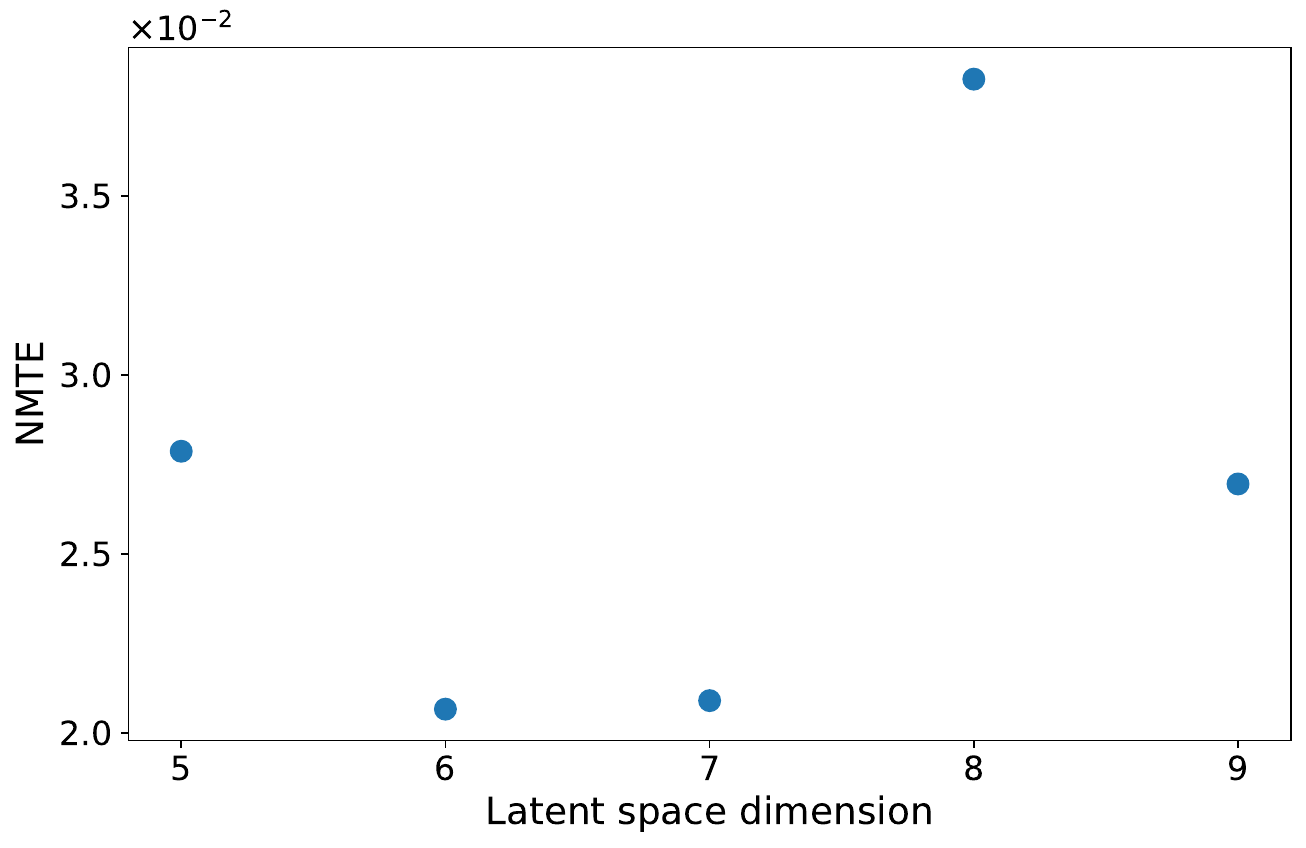}
    \caption{}
    \label{fig:vk_beam_splits1}
\end{subfigure}
\hfill
\begin{subfigure}{0.45\textwidth}
    \includegraphics[width=\textwidth]{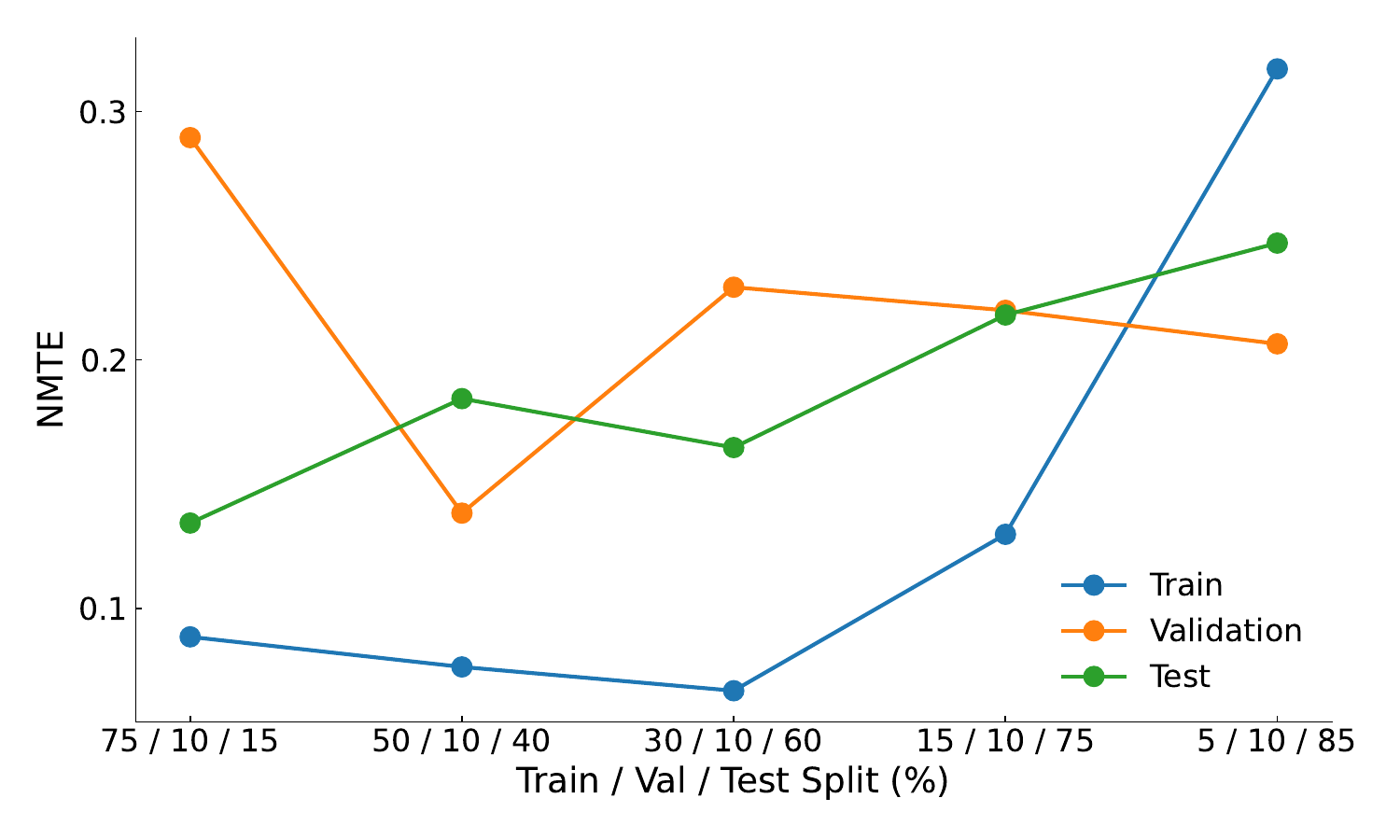}
    \caption{}
    \label{fig:vk_beam_splits2}
\end{subfigure}
\caption{(a) Normalized mean trajectory error comparisons for increasing AE dimension for the von Kármán beam example for 50/10 splitting for training and validation of the forcing data. (b) Normalized mean trajectory error comparisons for different training, testing and validation partitions of the single forcing response for the von Kármán beam example with a $6$D AE+LSTM model.}
\label{fig:vk_beam_splits}
\end{figure}

We repeat a similar procedure for a stochastically forced von K\'arm\'an beam example in section \ref{subsec:e5}. We found the approach to be easily implementable to that setting given one single realization of the forcing. The training procedure resulted in double the number of trainable parameters for the AE compared to the oscillator chain, with the LSTM network dimension remaining more or less the same. The minimal dimension of the latent space of the AE + LSTM model was found to be $6$. The dimension of the AE + LSTM model was selected based on two criteria: low normalized mean trajectory error (NMTE) for the AE reconstruction (see Fig. \ref{fig:vk_beam_splits}a) and low NMTE for the long-term predictions of the AE + LSTM model on the validation dataset (see Fig. \ref{fig:vk_beam_splits}b).

We observe that smaller training sets lead to worse predictions on test data. Throughout our examples, we use the 50/10 splitting for training and validation, which offers lower errors on validation data and also does not use up a lot of the forcing response. 

We also baseline our AE+LSTM model choice with a higher-dimensional AE, which has roughly 33,000 parameters, an 11-fold increase in trainable parameters compared to our autoencoder architecture. Due to the increase in trainable parameters, we use a 70/10 split for training and validation to characterize the new architecture's latent space effectively. Our algorithm found that the minimal dimension of the latent space was 11. In Fig. \ref{fig:ae_compare}, we compare the low- and high-dimensional AE+LSTM models' predictions on unseen forcing data. We see that the 6D AE+LSTM model's predictions have lower errors and faster run times than those of the 11D AE+LSTM model. Overall, higher-dimensional AE+LSTMs may offer more trainable parameters to tune to better fit the dynamics in the latent space. Still, they do not guarantee a performance boost once the predictions are lifted to the physical space.

\begin{figure}[H]
\centering

    \includegraphics[width=0.5\textwidth]{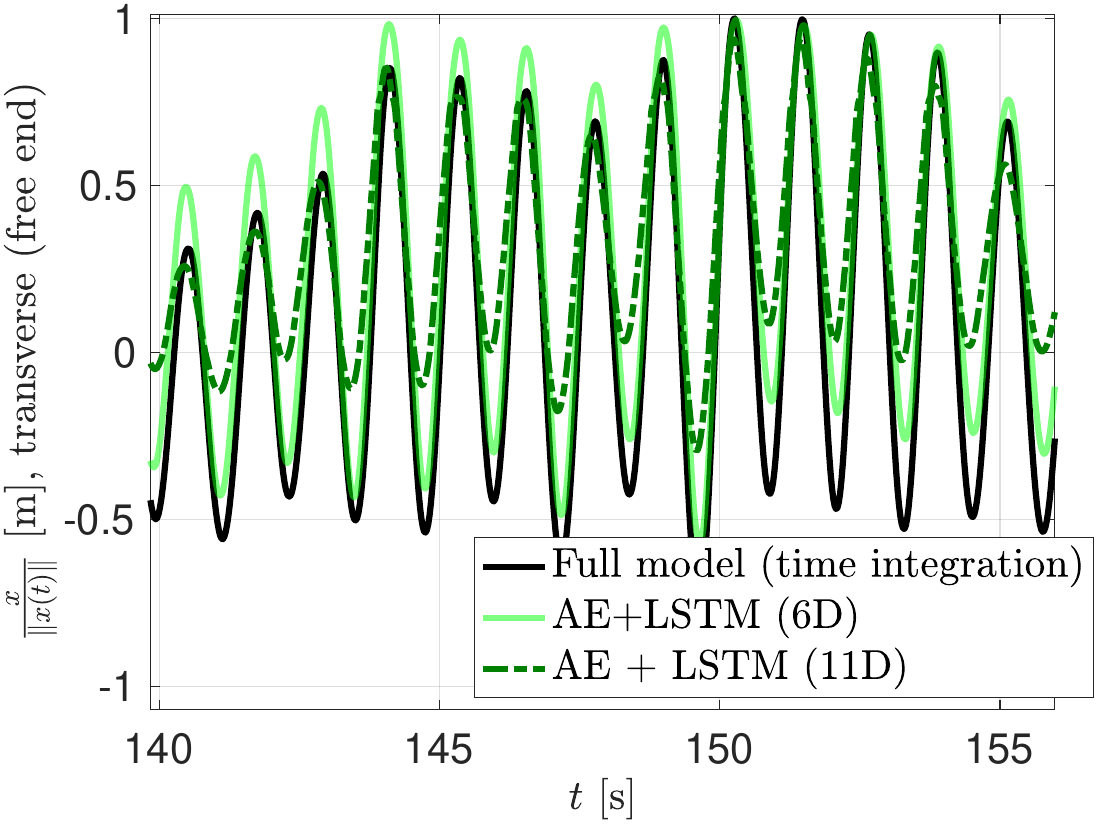}
    \caption{Time evolution for the normalized transverse displacement response of the free end of the von K\'arm\'an beam subject to stochastic ground excitations. The AE+LSTM (6D) prediction is shown in green, and the AE+LSTM (11D) is shown in dark green.}
    \label{fig:ae_compare}

\end{figure}
    
\section{Pad\'e and Taylor GSS comparisons}
\label{app:pade}
We illustrate here the divergent behavior of the Taylor expansions for GSS under large forcing values. This particular GSS Taylor approximation was computed using the leading order forced SSM-reduced model of the cantilevered  von K\'am\'an beam. Figure \ref{fig:pade_app} shows in orange an $O(20)$ GSS approximation, wherein the trajectory oscillates rapidly with large values and grows as time increases. Using formulas eq. (\ref{eq:lin_pade}), we compute a  $[10,10]$ vector Pad\'e approximation using the computed Taylor expansions. These rational functional approximations converge to the full response for both the axial and transverse degrees of freedom of the beam's free end.  

We also plot the $O(1)$ GSS approximation, which has a bounded response but fails to accurately track the amplitudes of the full model's GSS. We recall that this is a linear GSS of the SSM-reduced dynamics, but once lifted via the SSM parametrization it becomes a nonlinear response, and hence provides a reasonably large response in the axial degrees of freedom. Our analysis indicates that a vector Pad\'e approximation improves the accuracy of the GSS approximation dramatically from the Taylor expansions and also provides better refinement than a simple linear GSS calculation.  

\begin{figure}[H]
    \begin{centering}
    \includegraphics[width=1.1\textwidth]{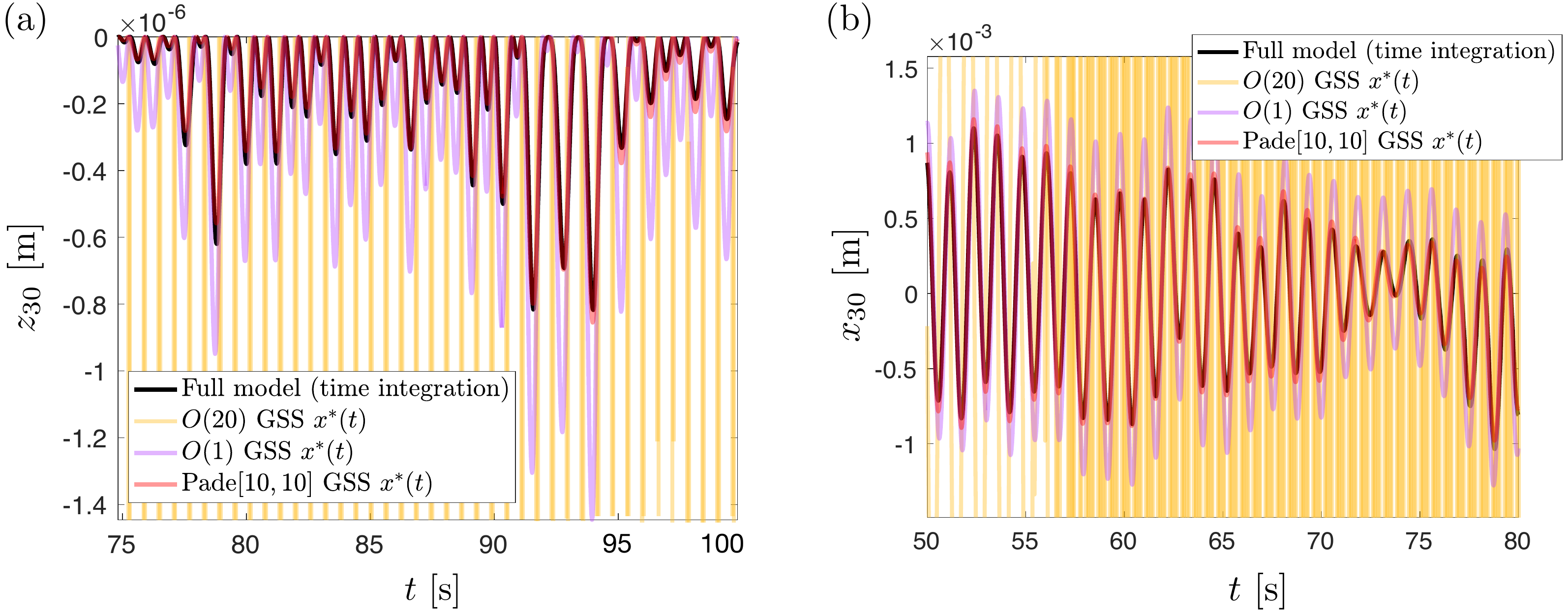}
    \par\end{centering}
    \caption{(a) Axial and (b) transverse displacement trajectory plots for the full model (black), Taylor approximation of the GSS (gold), linear GSS (pink) and vector Pad\'e approximation of the GSS (red).}
    \label{fig:pade_app}
    \end{figure}
\
\end{appendices}

\bibliographystyle{plainnat}
\bibliography{SSM_bibliography,sample}

\end{document}